\documentclass[a4paper,10pt]{amsart}
\usepackage[T1]{fontenc}
\usepackage{lmodern}
\usepackage{amsmath}
\usepackage{amsfonts}
\usepackage{amssymb}
\usepackage{amsthm}
\usepackage{thmtools}
\usepackage{mathtools}
\usepackage{booktabs}
\usepackage{enumitem}
\usepackage{tikz-cd}
\usetikzlibrary{decorations.pathmorphing}
\tikzcdset{row sep/normal=1.35em}
\usepackage[normalem]{ulem}
\usepackage{multirow}
\usepackage{comment}
\usepackage{placeins}
\usepackage{array}
\usepackage{xr}
\usepackage[pdftex,
            pdfauthor={},
            pdftitle={Modular Z/2Z-Crossed Tambara-Yamagami-like Categories for Even Groups},
            pdfsubject={},
            pdfkeywords={},
            pdfproducer={},
            pdfcreator={}]{hyperref}

\theoremstyle{plain}
\newtheorem{thm}{Theorem}[section]
\newtheorem{prop}[thm]{Proposition}
\newtheorem{cor}[thm]{Corollary}
\newtheorem{lem}[thm]{Lemma}

\newtheorem*{thm*}{Theorem}
\newtheorem*{conj*}{Conjecture}
\newtheorem*{prop*}{Proposition}

\theoremstyle{definition}
\newtheorem{defi}[thm]{Definition}

\newtheorem*{nota*}{Notation}
\newtheorem{rem}[thm]{Remark}
\newtheorem{ex}[thm]{Example}
\newtheorem{ass}[thm]{Assumption}

\newtheorem{idea}{Idea}
\newtheorem{prob}[idea]{Problem}
\newtheorem*{prob*}{Problem}

\newtheoremstyle{introTheorems}
  {}
  {}
  {\itshape}
  {}
  {\bfseries}
  {}
  { }
  {\thmname{#1}
  \textnormal{\bf \thmnote{#3}.$\hspace{-.1cm}$}
  }
\theoremstyle{introTheorems}
\newtheorem{introTheorem}{Theorem}

\makeatletter
\newcommand*{\addFileDependency}[1]{
\typeout{(#1)}
\@addtofilelist{#1}
\IfFileExists{#1}{}{\typeout{No file #1.}}
}\makeatother

\newcommand{\Q}{\mathbb{Q}}
\newcommand{\Z}{\mathbb{Z}}
\newcommand{\Ns}{\mathbb{Z}_{>0}}
\newcommand{\N}{\mathbb{Z}_{\geq0}}
\newcommand{\C}{\mathbb{C}}
\newcommand{\R}{\mathbb{R}}

\newcommand{\tr}{\operatorname{tr}}
\renewcommand{\i}{\mathrm{i}}
\newcommand{\e}{\operatorname{e}} 

\newcommand{\Aut}{\operatorname{Aut}}
\newcommand{\Hom}{\operatorname{Hom}}

\newcommand{\rk}{\operatorname{rk}}

\newcommand{\sign}{\operatorname{sign}}
\newcommand{\voa}{vertex operator algebra}

\newcommand{\id}{\operatorname{id}}
\newcommand{\eps}{\varepsilon}

\newcommand{\Irr}{\operatorname{Irr}}

\newcommand{\II}{{I\!I}}

\renewcommand{\d}{\mathrm{d}}

\newcommand{\Vect}{\operatorname{Vect}}
\newcommand{\no}{\,{\raise0.25em\hbox{$\mathop{\hphantom{\cdot}}\limits^{_{\circ}}_{^{\circ}}$}}\,}
\newcommand{\cC}{\mathcal{C}}
\newcommand{\cD}{\mathcal{D}}
\newcommand{\cB}{\mathcal{B}}

\newcommand{\sslash}{/\mkern-6mu/}
\newcommand{\Rep}{\operatorname{Rep}}

\newcommand{\TY}{\mathcal{TY}}
\newcommand{\LM}{\mathcal{G\!LM}}
\newcommand{\X}{\mathsf{X}}
\newcommand{\SMatrix}{\mathcal{S}}
\newcommand{\TMatrix}{\mathcal{T}}
\newcommand{\eqi}{\varphi}

\newcommand{\DiscQ}{{\bar{Q}}}
\newcommand{\DiscS}{{\bar{\sigma}}}
\newcommand{\DiscO}{{\bar{\omega}}}

\newcommand{\eval}{\operatorname{ev}}
\newcommand{\coeval}{\operatorname{coev}}

\newcommand{\unit}{\mathbf{1}}

\newcommand{\baromega}{\omega}

\newcommand{\longsquare}[9]{ 
\begin{center}
\begin{tikzpicture}[commutative diagrams/every diagram]
\node (P1) at (-4cm,0.7cm)
{$g_*(($#1$\otimes$#2$)\otimes$#3$)$};
\node (P2) at (0cm,0.7cm)
{$g_*($#1$\otimes$#2$)\otimes g_*$#3};
\node (P3) at (4cm,0.7cm)
{$(g_*$#1$\otimes g_*$#2$)\otimes g_*$#3};
\node (P4) at (-4cm,-0.7cm)
{$g_*($#1$\otimes ($#2$\otimes$#3$))$};
\node (P5) at (0cm,-0.7cm)
{$g_*$#1$\otimes g_*($#2$\otimes$#3$)$};
\node (P6) at (4cm,-0.7cm)
{$g_*$#1$\otimes (g_*$#2$\otimes g_*$#3$)$};
\path[commutative diagrams/.cd, every arrow, every label]
(P1) edge node {#4} (P4)
(P3) edge node[swap] {#5} (P6)
(P1) edge node[shift={(0,0.25)}] {#6} (P2)
(P2) edge node[shift={(0,0.25)}] {#7} (P3)
(P4) edge node[swap, shift={(0,-0.25)}] {#8} (P5)
(P5) edge node[swap, shift={(0,-0.25)}] {#9} (P6);
\end{tikzpicture}
\end{center}
}

\newcommand{\pentagon}[9]{ 
\begin{center}
\begin{tikzpicture}[commutative diagrams/every diagram]
\node (P1) at (180:3.5cm)
{$(($#1$\otimes$#2$)\otimes$#3$)\otimes$#4};
\node (P2) at (180-45+10:2.5cm)
{$($#1$\otimes ($#2$\otimes$#3$))\otimes$#4};
\node (P3) at (180-45-90-10:2.5cm)
{#1$\otimes (($#2$\otimes $#3$)\otimes $#4$))$} ;
\node (P4) at (180-45-90-45:3.5cm)
{#1$\otimes($#2$\otimes ($#3$\otimes $#4$))$};
\node (P0) at (180+90:1cm)
{$($#1$\otimes$#2$)\otimes($#3$\otimes$#4$)$};
\path[commutative diagrams/.cd, every arrow, every label]
(P1) edge node {#5} (P2)
(P2) edge node[shift={(0,0.25)}] {#6} (P3)
(P3) edge node {#7} (P4)
(P1) edge node[swap] {#8} (P0)
(P0) edge node[swap] {#9} (P4);
\end{tikzpicture}
\end{center}
}

\newcommand{\hexagon}[5]{
\def\SpaceA{#1}%
\def\SpaceB{#2}%
\def\aSpaceB{#3}
\def\SpaceC{#4}%
\def\aSpaceC{#5}
\hexagonContinued
}
\newcommand\hexagonContinued[6]{
\begin{center}
\begin{tikzpicture}[commutative diagrams/every diagram,scale=0.70]
\node (P1) at (180:4.5cm)
{$($\SpaceA$\otimes$\SpaceB$)\otimes$\SpaceC};
\node (P2) at (180-60+20:3cm)
{\SpaceA$\otimes($\SpaceB$\otimes$\SpaceC$)$};
\node (P3) at (180-60-60-20:3cm)
{$($\aSpaceB$\otimes$\aSpaceC$)\otimes$\SpaceA};
\node (P4) at (180-60-60-60:4.5cm)
{\aSpaceB$\otimes($\aSpaceC$\otimes$\SpaceA$)$};
\node (P5) at (180+60-20:3cm)
{$($\aSpaceB$\otimes$\SpaceA$)\otimes$\SpaceC};
\node (P6) at (180+60+60+20:3cm)
{\aSpaceB$\otimes($\SpaceA$\otimes$\SpaceC$)$};
\path[commutative diagrams/.cd, every arrow, every label]
(P1) edge node {#1} (P2)
(P2) edge node[shift={(0,0.25)}] {#2} (P3)
(P3) edge node {#3} (P4)
(P1) edge node[swap] {#4} (P5)
(P5) edge node[swap, shift={(0,-0.25)}] {#5} (P6)
(P6) edge node[swap] {#6} (P4);
\end{tikzpicture}
\end{center}
}

\newcommand{\hexagonInverse}[5]{
\def\SpaceA{#1}%
\def\SpaceB{#2}%
\def\SpaceC{#3}%
\def\bSpaceC{#4}
\def\abSpaceC{#5}
\hexagonInverseContinued
}
\newcommand\hexagonInverseContinued[6]{
\begin{center}
\begin{tikzpicture}[commutative diagrams/every diagram,scale=0.70]
\node (P1) at (180:4.5cm)
{\SpaceA$\otimes($\SpaceB$\otimes$\SpaceC$)$};
\node (P2) at (180-60+20:3cm)
{$($\SpaceA$\otimes$\SpaceB$)\otimes$\SpaceC};
\node (P3) at (180-60-60-20:3cm)
{\abSpaceC$\otimes($\SpaceA$\otimes$\SpaceB$)$};
\node (P4) at (180-60-60-60:4.5cm)
{$($\abSpaceC$\otimes$\SpaceA$)\otimes$\SpaceB};
\node (P5) at (180+60-20:3cm)
{\SpaceA$\otimes($\bSpaceC$\otimes$\SpaceB$)$};
\node (P6) at (180+60+60+20:3cm)
{$($\SpaceA$\otimes$\bSpaceC$)\otimes$\SpaceB};
\path[commutative diagrams/.cd, every arrow, every label]
(P1) edge node {#1} (P2)
(P2) edge node[shift={(0,0.25)}] {#2} (P3)
(P3) edge node {#3} (P4)
(P1) edge node[swap] {#4} (P5)
(P5) edge node[swap, shift={(0,-0.25)}] {#5} (P6)
(P6) edge node[swap] {#6} (P4);
\end{tikzpicture}
\end{center}
}

\begin{document}

\title[$\Z_2$-Crossed Tambara-Yamagami-like Categories for Even Groups]{Modular $\Z_2$-Crossed Tambara-Yamagami-like Categories for Even Groups}
\author[César Galindo, Simon Lentner and Sven Möller]{César Galindo,\textsuperscript{\lowercase{a}} Simon Lentner\textsuperscript{\lowercase{b}} and Sven Möller\textsuperscript{\lowercase{b},\lowercase{c}}}

\thanks{\textsuperscript{a}{Universidad de los Andes, Bogotá, Colombia}}
\thanks{\textsuperscript{b}{Universität Hamburg, Hamburg, Germany}}
\thanks{\textsuperscript{c}{Research Institute for Mathematical Sciences, Kyoto University, Kyoto, Japan}}

\thanks{\SMALL Email: \href{mailto:cn.galindo1116@uniandes.edu.co}{\nolinkurl{cn.galindo1116@uniandes.edu.co}}, \href{mailto:simon.lentner@uni-hamburg.de}{\nolinkurl{simon.lentner@uni-hamburg.de}}, \href{mailto:math@moeller-sven.de}{\nolinkurl{math@moeller-sven.de}}}

\begin{abstract}
We explicitly construct nondegenerate braided $\Z_2$-crossed tensor categories of the form $\smash{\Vect_\Gamma\oplus\Vect_{\Gamma/2\Gamma}}$. They are $\Z_2$-crossed extensions, in the sense of \cite{ENO10}, of the braided tensor category $\smash{\Vect_\Gamma}$ with $\Z_2$-action given by $-\!\id$ on the finite, abelian group $\Gamma$. Thus, we obtain generalisations of the Tambara-Yamagami categories \cite{TY98,Gal22}, where now the abelian group $\Gamma$ may have even order and the nontrivial sector $\smash{\Vect_{\Gamma/2\Gamma}}$ more than one simple object.

The idea for this construction comes from a physically motivated approach in \cite{GLM24a} to construct $\Z_2$-crossed extensions of $\Vect_\Gamma$ for any $\Gamma$ from an infinite Tambara-Yamagami category $\smash{\Vect_{\R^d}\oplus\Vect}$, which itself is not fully rigorously defined, and then using condensation from $\smash{\Vect_{\R^d}}$ to $\Vect_\Gamma$, which we prove commutes with crossed extensions.

The $\mathbb{Z}_2$-equivariantisation of $\smash{\Vect_\Gamma\oplus\Vect_{\Gamma/2\Gamma}}$ yields new modular tensor categories, which correspond to the orbifold of an arbitrary lattice \voa{} under a lift of $-\!\id$, as discussed in \cite{GLM24a}.
\end{abstract}

\maketitle
\setcounter{tocdepth}{1}
\tableofcontents
\setcounter{tocdepth}{2}


\section{Introduction}

Braided tensor categories and their specialisations, such as modular tensor categories, have emerged as important structures in low-dimensional physics, in particular $2$-di\-men\-sion\-al conformal field theories (e.g., in the form of \voa{}s or conformal nets) and topological phases of matter. In the presence of an action of a finite group~$G$, this notion is naturally modified to that of a braided $G$-crossed tensor category in the sense of \cite{Tur00}.

\medskip

At the centre of this article is the notion of a braided $G$-crossed extension, i.e.\ a braided $G$-crossed tensor category extending a given braided tensor category \cite{ENO10}. More precisely, given a braided tensor category $\cB$ together with an action of a finite group $G$ by braided tensor autoequivalences, a \emph{braided $G$-crossed extension} $\cC\supset\cB$ is a $G$-graded tensor category $\cC=\bigoplus_{g\in G}\cC_g$ equipped with a $G$-action satisfying $g_*\cC_h=\cC_{ghg^{-1}}$ and with a $G$-crossed braiding $\cC_g\otimes\cC_h\to\cC_{ghg^{-1}}\otimes\cC_g$ such that the identity component is $\cC_1=\cB$ with the given braiding and $G$-action.

For a braided $G$-crossed extension $\cB\subset\cC$, the $G$-equi\-vari\-an\-ti\-sa\-tion $\cC\sslash G$, also referred to as gauging of $\cB$ by $G$, is a braided tensor category extending $\cB\sslash G$.

Braided $G$-crossed extensions are essential for understanding representation categories of (suitably regular) \voa{}s \cite{McR21b}, as we discuss in more detail in \cite{GLM24a}. See, e.g., \cite{GJ19,DNR25,GR25} for some applications.

\medskip

In this article, we study braided $G$-crossed extensions of \emph{pointed} braided fusion categories $\cB$ \cite{JS93,EGNO15}. The latter have the defining property that all their simple objects are invertible, i.e.\ they are of the form $\smash{\cB=\Vect_\Gamma^Q}$ for some finite, abelian group $\Gamma$ and a quadratic form $Q$ on $\Gamma$ determining the braiding and associator. In the context of lattices, such pairs $(\Gamma,Q)$ appear as discriminant forms, for which $Q$ nondegenerate.

Perhaps the easiest nontrivial examples are the Tambara-Yamagami categories, which provide $G$-crossed extensions of $\smash{\Vect_\Gamma^Q}$ for $G=\Z_2$ acting on $\Gamma$ by multiplication with $-1$. Indeed, in \cite{TY98}, Tambara and Yamagami classified the $\Z_2$-graded fusion categories $\cC$ with exactly one noninvertible simple object in the graded component of nontrivial degree $g\in\Z_2$, i.e.\ of the form
\begin{equation*}
\cC=\cC_1\oplus\cC_g=\Vect_\Gamma^Q\oplus\Vect.
\end{equation*}
In general, this fusion category does not admit a braiding, but with the $G$-action above, there exist $G$-crossed braidings, which were classified in \cite{Gal22}. This produces two explicit $\Z_2$-crossed extensions $\cC=\smash{\Vect_\Gamma^Q[\Z_2,\eps]}$, $\eps\in\{\pm1\}$, of $\cB=\smash{\Vect_\Gamma^Q}$, see \autoref{cor_TY}. Finding precisely two $G$-crossed extensions of the given $\cB$ with the given categorical action of $G$ on $\cB$ matches the general classification result for braided $G$-crossed extensions in \cite{ENO10,DN21}; see \autoref{thm_ourENO_braided}.

One is especially interested in cases where the $\Z_2$-crossed braiding, and thus in particular the given braiding on $\smash{\Vect_\Gamma}$, i.e.\ the quadratic form $Q$, is nondegenerate. In this case, the $\Z_2$-equivariantisation $\cC\sslash\Z_2$ is a modular tensor category (see \autoref{prop_TY_modular}), which was described in \cite{GNN09}. The definition of the Tambara-Yamagami category requires that on $\smash{\Vect_\Gamma^Q}$ the associator is trivial and that the quadratic form $Q$ admits a square root, i.e.\ a quadratic form $q$ with $Q=q^2$, which appears in the braiding and associator involving the $g$-graded component. In the nondegenerate case, this necessarily means that the group $\Gamma$ has \emph{odd order} or, equivalently, that $G$ acts on $\Gamma$ with the only fixed point $0$.

\medskip

In this article, we generalise the results in \cite{TY98,Gal22} by constructing braided $\Z_2$-crossed extensions $\cC$ of $\cB=\smash{\Vect_\Gamma^Q}$ with the quadratic form $Q$ nondegenerate and the finite, abelian group $\Gamma$ of \emph{arbitrary order} where $G\cong\Z_2$ still acts (as a certain categorical action, cf.\ \autoref{rem_minusoneaction}) by multiplication with $-1$ on $\Gamma$.
\begin{introTheorem}[\ref{thm_evenTY}]
For $\eps\in\{\pm1\}$, the data given in \autoref{sec_evenTY} define a $\Z_2$-crossed ribbon fusion category
\begin{equation*}
\cC=\Vect_\Gamma^Q[\Z_2,\eps]=\Vect_\Gamma^Q\oplus\Vect_{\Gamma/2\Gamma},
\end{equation*}
which is a braided $\Z_2$-crossed extension of $\smash{\Vect_\Gamma^Q}$ for a discriminant form $(\Gamma,Q)$ with the above categorical $\Z_2$-action. It is equipped with a natural choice of $\Z_2$-ribbon structure for which all quantum dimensions are positive.
\end{introTheorem}
The definition of the $\Z_2$-crossed braiding on $\cC$ again depends on a (special) choice of square root $q$ of the quadratic form $Q$ on $\Gamma$ (see \autoref{prop_ass_q}).

By the classification of braided $G$-crossed extensions \cite{ENO10,DN21}, for each $\eps$, there is a unique braided $\Z_2$-crossed extension $\smash{\Vect_\Gamma^Q[\Z_2,\eps]}$ of $\smash{\Vect_\Gamma^Q}$ with the given $\Z_2$-action. In \autoref{prob_ENO} we ask if one can verify this directly from our construction?

We also compute the $\Z_2$-equivariantisation $\cC\sslash\Z_2$, which is again a modular tensor category (see \autoref{prop_equimod}). In particular, we determine the simple objects, fusion rules and modular data of $\cC\sslash\Z_2$ (see \autoref{prop_eqivsimple}, \autoref{prop_equi_fusion}, \autoref{prop_equivmodular}).

\medskip

The input data for the construction of the above braided $\Z_2$-crossed categories are taken from \cite{GLM24a}, where they are produced by using the idea that braided $G$-crossed extensions commute in a certain sense with condensations by commutative, associative algebra objects. Concretely, this idea is applied to the situation
\begin{equation*}
\begin{tikzcd}
\Vect_{\R^d}^\DiscQ
\arrow[rightsquigarrow]{d}{\text{cond.}}
\arrow[hookrightarrow]{rrr}{\text{$\Z_2$-crossed ext.}}
&&&
\Vect_{\R^d}^\DiscQ\oplus\Vect
\arrow[rightsquigarrow]{d}{\text{cond.}}
\\
\Vect_\Gamma^Q
\arrow[hookrightarrow]{rrr}{\text{$\Z_2$-crossed ext.}}
&&&
\Vect^Q_\Gamma[\Z_2,\eps]
\end{tikzcd}
\end{equation*}
where $(\Gamma,Q)$ is realised as an isotropic subgroup (meaning an even lattice~$L$) of the quadratic space $\smash{(\R^d,\DiscQ)}$. While the infinite Tambara-Yamagami category $\smash{\Vect_{\R^d}^\DiscQ\oplus\Vect}$ is not fully rigorous (but see \cite{Mar25}, which appeared after the completion of this manuscript, for a step in this direction), we can still use the data obtained in \cite{GLM24a} as the data that \emph{should} define $\smash{\Vect_\Gamma^Q}[\Z_2,\eps]$ and then verify the coherence conditions explicitly (see \autoref{sec_app}). This yields the rigorous (but without \cite{GLM24a} ad hoc seeming) definition of the braided $\Z_2$-crossed extension $\smash{\Vect_\Gamma^Q}[\Z_2,\eps]$ of $\smash{\Vect_\Gamma^Q}$ given in this paper.


\subsection*{Outline}

In \autoref{sec_prel}, we briefly recall discriminant forms and lattices.

In \autoref{sec_ENO}, we recall braided $G$-crossed tensor categories. We also introduce $G$-ribbon structures and discuss pseudo-unitarity in that context.

As a first example, in \autoref{sec_TY}, we describe Tambara-Yamagami categories as braided $\Z_2$-crossed extensions of pointed braided fusion categories.

Then, in \autoref{sec_evenTY}, we define the generalisations of the Tambara-Yamagami categories where the abelian group $\Gamma$ describing the underlying pointed braided fusion category may be of even order. The necessary input data are defined in \autoref{sec_specialcoc}. In \autoref{sec_equiv}, we describe the equivariantisations. In \autoref{sec_latticedata} we describe these categories explicitly in terms of lattice data. The proof of the main theorem, \autoref{thm_evenTY}, is given in \autoref{sec_app}.


\subsection*{Notation}

Unless stated otherwise, all vector spaces will be over the base field $\C$. Categories will be enriched over $\Vect=\smash{\Vect_\C}$, but we could equally well take any algebraically closed field of characteristic zero. By $\Z_n$ we shall always mean the cyclic group $\Z/n\Z$. We write $\e(x)=e^{2\pi\i x}$.

Throughout the text, we denote by $G$ a (usually finite) group, written multiplicatively,
with $g$ denoting an arbitrary element of $G$. However, in the special case of $G=\langle g\rangle\cong\Z_2$, we denote by $g$ the nontrivial element of $G$. By contrast, $\Gamma$ always denotes an abelian group, which we write additively.
Quadratic forms on $\Gamma$ are usually written multiplicatively with values in $\C^\times$; they are related to the usual notion of quadratic forms with values in $\Q/\Z$ or $\R/\Z$ via $\e(\cdot)$.


\subsection*{Acknowledgements}

We are happy to thank Yoshiki Fukusumi, Terry Gannon, Hannes Knötzele, Adrià Marín~Salvador, Christian Reiher, Andrew Riesen, Ingo Runkel and Christoph Schweigert for helpful discussions and comments. We thank the referee for their helpful suggestions and corrections.

César Galindo would like to express his gratitude for the hospitality and excellent working conditions provided by the Department of Mathematics at the University of Hamburg, where part of this research was conducted as a Fellow of the Humboldt Foundation. Simon Lentner would like to express his gratitude for the hospitality of Thomas Creutzig and Terry Gannon at the University of Alberta.

César Galindo was partially supported by Grant INV-2025-213-3452 from the School of Science of Universidad de los Andes. Simon Lentner was supported by the Humboldt Foundation. Sven Möller acknowledges support from the DFG through the Emmy Noether Programme and the CRC 1624 \emph{Higher Structures, Moduli Spaces and Integrability}, project numbers 460925688 and 506632645. Sven Möller was also supported by a Postdoctoral Fellowship for Research in Japan and Grant-in-Aid KAKENHI 20F40018 by the \emph{Japan Society for the Promotion of Science}.


\section{Lattices and Discriminant Forms}\label{sec_prel}

By a (rational) \emph{lattice}, we mean a free abelian group $L$ of finite rank equipped with a nondegenerate bilinear form $\langle\cdot,\cdot\rangle\colon L\times L\to\Q$. The lattice $L$ is called \emph{integral} if $\langle v,w\rangle\in\Z$ for all $v,w\in L$. If $\langle v,v\rangle\in 2\Z$ for all $v\in L$, then $L$ is called \emph{even}. An integral lattice that is not even is called \emph{odd}. Given a lattice $L$, we denote by $L^*=\{v\in L\otimes_\Z\R\mid\langle v,w\rangle\in\Z\text{ for all }w\in L\}$ the \emph{dual} of $L$, embedded via $\langle\cdot,\cdot\rangle$ into the ambient space $L\otimes_\Z\R$ of $L$.

If $L$ is even, then $L\subset L^*$ and the quotient $L^*\!/L$ is a finite, abelian group endowed with a nondegenerate quadratic form $Q\colon L^*\!/L\to\C^\times$ given by $Q(v+L)=\e(\langle v,v\rangle/2)$ for all $v\in L^*$. In that situation, we call $L^*\!/L$ the \emph{discriminant form} of $L$.

More generally, we call any pair $(\Gamma,Q)$ of a finite, abelian group $\Gamma$ together with a nondegenerate quadratic form $Q\colon\Gamma\to\C^\times$ a \emph{discriminant form} (sometimes called \emph{metric group}). The function $Q$ being quadratic means that $\smash{Q(na)=Q(a)^{n^2}}$ for all $a\in\Gamma$ and $n\in\Z$ and that $B_Q\colon\Gamma\times\Gamma \to\C^\times$ with $B_Q(a,b)\coloneqq Q(a+b)Q(a)^{-1}Q(b)^{-1}$ for $a,b\in\Gamma$ is bimultiplicative (and nondegenerate). We call $B_Q$ the \emph{associated bimultiplicative form}. Any discriminant form $(\Gamma,Q)$ can be realised as dual quotient $L^*\!/L$ for some even lattice~$L$ \cite{Nik80}.

Given two discriminant forms $(\Gamma_1,Q_1)$ and $(\Gamma_2,Q_2)$, we denote their (orthogonal) direct sum by $(\Gamma_1\oplus\Gamma_2,Q_1\oplus Q_2)$, which is the group $\Gamma_1\oplus\Gamma_2$ with the quadratic form $(Q_1\oplus Q_2)((a_1,a_2))=Q_1(a_1)Q_2(a_2)$ for $a_1\in\Gamma_1$, $a_2\in\Gamma_2$.

A discriminant form is called \emph{indecomposable} if it cannot be written as an orthogonal direct sum of two nonzero discriminant forms. Such discriminant forms are also referred to as indecomposable Jordan components. \autoref{tab_disc} describes the indecomposable Jordan blocks (see, e.g., \cite{CS99,Sch09}). Here, $k\in\Ns$, $p$ is an odd prime and $\big(\frac{\cdot}{\cdot}\big)$ denotes the Kronecker symbol.

\begin{table}[ht]
\caption{Indecomposable Discriminant Forms}
\label{tab_disc}
\begin{tabular}{@{}l l p{6.2cm}@{}}
Symbol & Group & Quadratic Form \\
\toprule
$(p^k)^{\pm1}$ & $\Z_{p^k} = \langle x \rangle$ & $Q(x) = \e(p^{-k} u)$ for some $u \in \Z$ with $(u, p) = 1$ and $\bigl( \frac{2u}{p} \bigl) = \pm 1$. \\
\cmidrule{1-3}
$(2^k)_t^{\pm1}$ & $ \Z_{2^k} = \langle x \rangle $ & $Q(x) = \e(2^{-k-1} t)$ for some $t \in \Z_8$ with $\bigl( \frac{t}{2} \bigl) = \pm 1$ \\
\cmidrule{1-3}
$(2^k)^{+2}_\II$ & $\Z_{2^k} \times \Z_{2^k} = \langle x, y \rangle$ & $Q(x) = Q(y) = 1$ and $B_Q(x,y) = \e(2^{-k})$ \\
\cmidrule{1-3}
$(2^k)^{-2}_\II$ & $\Z_{2^k} \times \Z_{2^k} = \langle x, y \rangle$ & $Q(x) = Q(y) = B_Q(x,y) = \e(2^{-k})$
\end{tabular}
\end{table}


\section{Braided \texorpdfstring{$G$}{G}-Crossed Tensor Categories}\label{sec_ENO}

In this section, we introduce the main categorical notions relevant for this text, in particular braided $G$-crossed tensor categories and braided $G$-crossed extensions. We also state the results from \cite{ENO10,DN21} on the classification of the latter. Then, we describe $G$-ribbon structures, pseudo-unitarity and their relation to the $G$-equivariantisation.


\subsection{Tensor Categories and \texorpdfstring{$G$}{G}-Actions}

In this text, tensor categories are $\C$-linear abelian tensor categories, similar to \cite{DGNO10}. Moreover, we consider braided tensor categories, rigid tensor categories and ribbon categories. A fusion category is a tensor category that is finite, semisimple, rigid and has a simple tensor unit $\unit$, as in \cite{EGNO15}. A modular tensor category is a ribbon fusion category with nondegenerate braiding.

A rigid tensor category is called \emph{pointed} if every simple object is invertible (or a simple current). In the following, we consider the typical description of pointed (braided) fusion categories associated with finite (abelian) groups and some cohomological data (see, e.g., \cite{JS93,EGNO15}).

\begin{ex}[Pointed Fusion Categories]\label{ex_tensorVect}
Suppose that $G$ is a finite group and $\omega\in Z^3(G,\C^\times)$ a $3$-cocycle. Consider the pointed fusion category $\smash{\Vect_G^\omega}$ of $G$-graded vector spaces, where the simple objects are $\C_g$, $g\in G$, and the tensor product is given by $\C_g\otimes\C_h=\C_{gh}$ with associator
\begin{equation*}
(\C_g\otimes\C_h)\otimes\C_k\overset{\omega(g,h,k)}{\longrightarrow}\C_g\otimes(\C_h\otimes\C_k)
\end{equation*}
for all $g,h,k\in G$. The pentagon identity for the associator holds precisely because $\omega$ is a $3$-cocycle.
\end{ex}

\begin{ex}[Pointed Braided Fusion Categories]\label{ex_braidedVect}
Let $\Gamma$ be a finite, abelian group and $(\sigma,\omega)\in Z_{\mathrm{ab}}^3(\Gamma,\C^\times)$ an abelian $3$-cocycle on it \cite{EM50}. We denote by $\smash{\Vect_\Gamma^{\sigma,\omega}}$ the pointed braided fusion category given by the above fusion category $\smash{\Vect_\Gamma^\omega}$ equipped with the braiding
\begin{equation*}
\C_a\otimes\C_b\overset{\sigma(a,b)}{\longrightarrow}\C_b\otimes\C_a
\end{equation*}
for $a,b\in\Gamma$. The hexagon identity for the braiding corresponds exactly to the defining property of an abelian $3$-cocycle, i.e.
\begin{equation}\label{eq_abelian cocycle equations}
\frac{\omega(b,a,c)}{\omega(a,b,c)\omega(b,c,a)}=\frac{\sigma(a,b+c)}{\sigma(a, b)\sigma(a,c)},\quad\frac{\omega(a,b,c)\omega(c,a,b)}{\omega(a,c,b)}=\frac{\sigma(a+b,c)}{\sigma(a,c)\sigma(b,c)}
\end{equation}
for $a,b,c\in\Gamma$, in addition to the $3$-cocycle property of $\omega$. If $\omega=1$, \eqref{eq_abelian cocycle equations} states that $\sigma$ is bimultiplicative; else, $\omega$ measures its deviation from bimultiplicativity.

Abelian $3$-coboundaries are of the form $(\sigma_\kappa,\d\kappa)$ with $\sigma_\kappa(a,b)\coloneqq\kappa(a,b)\kappa(b,a)^{-1}$ for any function $\kappa\colon\Gamma\times\Gamma\to\C^\times$, and they correspond to braided tensor equivalences $\smash{\Vect_\Gamma^{\sigma,\omega}}\cong\smash{\Vect_\Gamma^{\sigma\sigma_\kappa,\omega\d\kappa}}$ that map each object $\C_a$ to itself but with a possibly nontrivial tensor structure given by $\kappa$. Hence, cohomology classes of abelian $3$-cocycles on $\Gamma$ correspond to equivalence classes of braided tensor categories on the abelian category $\smash{\Vect_\Gamma}$ with the given tensor product.

The abelian $3$-cocycles $(\sigma,\omega)$, up to coboundaries, correspond bijectively to quadratic forms $Q\colon\Gamma\to\C^\times$, as shown in \cite{EM50}. Recall that $B_Q\colon\Gamma\times\Gamma\to\C^\times$ with $B_Q(a,b)=Q(a+b)Q(a)^{-1}Q(b)^{-1}$ for $a,b\in\Gamma$ denotes the associated bimultiplicative form. In this correspondence, $B_Q(a,b)=\sigma(a,b)\sigma(b,a)$ is the double-braiding and $Q(a)=\sigma(a,a)$ the self-braiding. If the quadratic form $Q$ is nondegenerate, meaning that $B_Q$ is, then the braiding is nondegenerate. We denote by $\smash{\Vect_\Gamma^Q}$ the equivalence class of braided tensor categories corresponding to a cohomology class of abelian $3$-cocycles defined by the quadratic form $Q$ on $\Gamma$.

There is a rigid structure $\C_a^*=\C_{-a}$, $a\in\Gamma$, with the obvious evaluation and coevaluation. There is also a canonical choice of a ribbon structure $\theta_{\C_a}=Q(a)$. As we shall discuss in \autoref{sec_GRibbon} and \autoref{sec_pseudo}, this is the unique choice of ribbon structure for which all simple objects $\C_a$ have quantum dimension~$1$, thus coinciding with the Frobenius-Perron dimension. This is known as the pseudo-unitary choice. All possible ribbon structures differ from this by a homomorphism $\Gamma\to\{\pm1\}$.

If $|\Gamma|$ is odd, then the cohomology class associated with the quadratic form $Q$ on $\Gamma$ can be represented by a distinguished abelian $3$-cocycle $(\sigma,\omega)$ with $\omega=1$ and $\sigma(a,b)=\smash{B_Q^{1/2}(a,b)}$ for $a,b\in\Gamma$. Here, we recall that when the abelian group $\Gamma$ has odd order, every bimultiplicative map $B\colon\Gamma\times\Gamma\to\C^\times$ has a unique bimultiplicative square root $\smash{B^{1/2}}$; an analogous statements holds for quadratic forms.
\end{ex}

\medskip

We describe group actions on tensor categories:
\begin{defi}\label{def_Gcrossed}
Let $\cC$ be a tensor category with the associativity constraint given by $(X\otimes Y)\otimes Z\overset{\alpha}{\to}X\otimes (Y\otimes Z)$ and $G$ a finite group. Then a \emph{$G$-action} on $\cC$ consists of the following data:
\begin{enumerate}[label=(\roman*)]
\item For every element $g\in G$, a functor
\begin{equation*}
g_*\colon\cC\to\cC
\end{equation*}
satisfying $g_*(\unit)=\unit$. We denote the image of an object $X\in\cC$ by $g_*(X)$ and that of a morphism $f$ by $g_*(f)$.
\item For every pair of elements $g,h\in G$, a natural isomorphism
\begin{equation*}
T_2^{g,h}(X)\colon(gh)_*(X)\to g_*(h_*(X))
\end{equation*}
such that associativity holds: for all $g,h,l\in G$,
\begin{equation*}
T_2^{g,h}(l_*(X))\circ T_2^{gh,l}(X)=g_*(T_2^{h,l}(X))\circ T_2^{g,hl}.
\end{equation*}
\item The $g_*$ are tensor functors with tensor structure $\tau^g$, i.e.\ for every $g\in G$ and every pair of objects $X,Y\in\cC$, a natural isomorphism
\begin{equation*}
\tau^g_{X,Y}\colon g_*(X\otimes Y)\to g_*(X)\otimes g_*(Y)
\end{equation*}
such that the diagrams
\begin{equation*}
\begin{tikzcd}[column sep=small]
g_*((X \!\otimes\! Y) \!\otimes\! Z) \ar[r, "\tau^g"] \ar[d, "g_*(\alpha)"'] & g_*(X \!\otimes\! Y) \!\otimes\! g_*(Z) \ar[r, "\tau^g"] & (g_*(X) \!\otimes\! g_*(Y)) \!\otimes\! g_*(Z) \ar[d, "\alpha"] \\
g_*(X \!\otimes\! (Y \!\otimes\! Z)) \ar[r, "\tau^g"'] & g_*(X) \!\otimes\! g_*(Y \!\otimes\! Z) \ar[r, "\tau^g"'] & g_*(X) \!\otimes\! (g_*(Y) \!\otimes\! g_*(Z))
\end{tikzcd}
\end{equation*}
and
\begin{equation*}
\begin{tikzcd}
(gh)_*(X \otimes Y) \ar[r, "\tau^{gh}"] \ar[d, "T_2^{g,h}(X \otimes Y)"'] & (gh)_*(X) \otimes (gh)_*(Y) \ar[d, "T_2^{g,h}(X) \otimes T_2^{g,h}(Y)"] \\
g_*(h_*(X \otimes Y)) \ar[r, "\tau^g \tau^h"'] & g_*(h_*(X)) \otimes g_*(h_*(Y))
\end{tikzcd}
\end{equation*}
commute for all $g,h\in G$.
\end{enumerate}
A $G$-action is called \emph{strict} if all the tensor functors $(g_*, \tau^g)$ are strict and the natural transformations $T_2$ are equalities.
\end{defi}

\begin{rem}
The data of an action of $G$ on $\cC$ are equivalent to a monoidal functor from the discrete tensor category $\smash{\underline{G}}$, where the set $G$ serves as objects and the tensor product is defined by the multiplication in $G$, to the tensor category of tensor functors $\smash{\underline{\Aut}_\otimes(\cC)}$ along with tensor natural isomorphisms.

In our definition, we have assumed that the tensor functors $(g_*,\tau^g)$ are unital, i.e.\ that $g_*(\unit)=\unit$, and strong, i.e.\ that $\tau^g$ is a natural isomorphism. More generally, there would be an isomorphism $T_{0}^g\colon g_*(\unit)\to\unit$ satisfying a coherence condition. However, since every strong tensor functor is isomorphic to a unital one, we shall always assume this property and have therefore included it in our definition.
\end{rem}

\smallskip

Given a $G$-action on a tensor category, we can define the following category, whose objects arise from $G$-invariant objects of the original tensor category.
\begin{defi}
Let $\cC$ be a tensor category and $G$ a finite group acting on $\cC$. The \emph{$G$-equivariantisation} of $\cC$, denoted by $\cC\sslash G$, is the tensor category with objects the pairs $(X,\eqi)$ where $X\in\cC$ and $\eqi$ is a family of isomorphisms, $\eqi_g\colon g_*(X) \to X$ for each $g \in G$, satisfying
\begin{equation*}
\eqi_{gh}=\eqi_g\circ g_*(\eqi_h)\circ T_2^{g,h}
\end{equation*}
for $g,h\in G$. Morphisms $f\colon(X,\eqi)\to(Y,\eqi')$ in $\cC\sslash G$ satisfy
\begin{equation*}
\eqi'_g\circ g_*(f)=f\circ\eqi_g,
\end{equation*}
the tensor product is $(X,\eqi)\otimes(Y,\eqi')=(X\otimes Y,\eqi'')$ with
\begin{equation*}
\eqi''_g=\eqi_g\otimes\eqi'_g\circ\tau^g_{X,Y}
\end{equation*}
for $X,Y\in\cC$ and $g\in G$, and the tensor unit is $(\unit,\id_\unit)$.
\end{defi}

We describe the simple objects of $\cC\sslash G$. The group $G$ acts on the set $\Irr(\cC)$ of equivalence classes of simple objects. Fixing a representative $X_i$ for every orbit and denoting the corresponding stabiliser subgroup by $G_i$, we can choose isomorphisms $t_g\colon g_*(X_i)\to X_i$. The action restricted to $G_i$ defines a $2$-cocycle $\chi_i\in Z^2(G_i,\C^\times)$,
\begin{equation}\label{eq_equivariant}
\chi_i(g,h)\id_{X_i}=t_{gh}^{-1}\circ t_g\circ g_*(t_h)\circ T_2^{g,h}
\end{equation}
for all $g,h\in G_i$. The cohomology class of $\chi_i$ does not depend on the choice of $\{t_g\mid g\in G_i\}$. The simple objects of $\cC\sslash G$, up to isomorphism, are in bijection with isomorphism classes of irreducible $\chi_i$-projective representations of $G_i$ for every $i$. For more details on the correspondence, see \cite{BN13}.

\begin{ex}\label{ex_classification equivariant simple order 2}
Let $G=\langle g\rangle\cong\Z_2$ be a cyclic group of order~$2$ acting on a tensor category~$\cC$. In this case, the classification of simple objects of $\cC\sslash G$ is straightforward. If a simple object $X$ has a trivial stabiliser, then $X^g\coloneqq X\oplus g_*(X)$ is an equivariant object with $\eqi_g=\smash{\id_{g_*(X)}\oplus T_2^{g,g}(X)}$. If the stabiliser is $G$, then take $t_g\colon g_*(X)\to X$ and, using the equation~\eqref{eq_equivariant}, define $\gamma=\chi(g,g)^{-1/2}$. The isomorphisms $\smash{\eqi_g^{(\pm)}}=\pm\gamma t_g\colon g_*(X)\to X$ endow $X$ with two nonisomorphic equivariant structures.
\end{ex}


\subsection{Braided \texorpdfstring{$G$}{G}-Crossed Tensor Categories}\label{sec_Gcrossed}

We now endow the categories from the previous section with the additional structure of a braiding. The so obtained braided $G$-crossed tensor categories (see, e.g., \cite{Tur00,EGNO15}) are the central objects in this text.

For a tensor category $\cC$ and a group $G$, a (faithful) $G$-grading on $\cC$ is a decomposition $\cC=\bigoplus_{g\in G}\cC_g$ such that the tensor product $\otimes$ maps $\cC_g\times\cC_h$ to $\cC_{gh}$, the unit object $\unit$ is in $\cC_1$ and $\cC_g\neq0$ for all $g\in G$.

\begin{defi}\label{def_G-crossed}
A \emph{braided $G$-crossed tensor category} is a tensor category $\cC$ endowed with the following structures:
\begin{enumerate}[label=(\roman*)]
\item an action of $G$ on $\cC$,
\item a faithful $G$-grading $\cC=\oplus_{g\in G}\cC_g$,
\item isomorphisms, called the \emph{$G$-braiding},
\begin{equation*}
c_{X,Y}\colon X\otimes Y\to g_*(Y)\otimes X
\end{equation*}
for $g\in G$, $X\in\cC_g$, $Y\in\cC$, natural in $X$ and $Y$.
\end{enumerate}
These structures are subject to the following conditions, where we omit the associativity constraints in the diagrams for better readability:
\begin{enumerate}[label=(\alph*)]
\item $g_*(\cC_h)\subset\cC_{ghg^{-1}}$ for all $g, h \in G$.
\item The diagrams
\begin{equation*}
\begin{tikzcd}
g_*(X\otimes Y) \ar{dd}{\tau^g_{X,Y}} \ar{rr}{g_*(c_{X,Y})} && g_*(h_*(Y)\otimes X)\ar{dd}{T_2^{ghg^{-1},h}(T_2^{g,h})^{-1}\tau^g_{h_*Y,X}}\\\\
g_*(X)\otimes g_*(Y) \ar{rr}{c_{g_*(X),g_*(Y)}} && (ghg^{-1})_*g_*(Y)\otimes g_*(X)
\end{tikzcd}
\end{equation*}
commute for all $g,h\in G$, $X\in\cC_h$, $Y\in\cC$.
\item The hexagon diagrams
\begin{equation*}
\begin{tikzcd}
X \otimes Y\otimes Z \ar{rr}{c_{X,Y\otimes Z}} \ar{d}{c_{X,Y}\otimes\id_{Z}} && g_*(Y\otimes Z)\otimes X \ar{d}{\tau^g_{Y,Z}} \\
g_*(Y)\otimes X\otimes Z \ar{rr}{\id_{g_*(Y)}\otimes c_{X,Z}}&& g_*(Y)\otimes g_*(Z) \otimes X
\end{tikzcd}
\end{equation*}
commute for all $g\in G$, $X\in\cC_g$, $Y,Z\in\cC$ and the inverse hexagon diagrams
\begin{equation*}
\begin{tikzcd}
X \otimes Y\otimes Z \ar{rr}{c_{X\otimes Y, Z}} \ar{d}{\id_X\otimes c_{Y,Z}} && (gh)_*(Z)\otimes X\otimes Y \ar{d}{T_2^{g,h}} \\
X\otimes h_*(Z)\otimes Y\ar{rr}{c_{X,h_*(Z)}\otimes \id_Y}&& g_*h_*(Z)\otimes X \otimes Y
\end{tikzcd}
\end{equation*}
commute for all $g,h\in G$, $X\in\cC_g$, $Y\in\cC_h$, $Z\in\cC$.
\end{enumerate}
\end{defi}
The definition of equivalence of braided $G$-crossed tensor categories is given in \cite{G17}, Section~5.2.

We call $\cC$ as in the definition a \emph{braided $G$-crossed extension} of a braided tensor category $\cB$ if $\cC_1=\cB$.

\medskip

An important application of braided $G$-crossed tensor categories is the following, related to the equivariantisation, which can be found in the present context in \cite{Tur10b} (appendix by Michael Müger):
\begin{prop}[\cite{DGNO10}]
Let $\cC$ be a braided $G$-crossed tensor category. The equivariantisation $\cC \sslash G$ is a braided tensor category that contains a fusion subcategory braided equivalent to the symmetric category $\Rep(G)$. This corresponds to those equivariant objects $(V,(\eqi_g)_{g\in G})$ where $V$ is in the tensor subcategory generated by the unit object of $\cC$.

Conversely, for every braided tensor category $\cD$ containing the symmetric fusion category $\Rep(G)$ with trivial braiding, we can define a braided $G$-crossed tensor category $\cC\coloneqq\cD_G$ as the \emph{de-equivariantisation} by $G$. This is the tensor category of modules over $\smash{\C^G}$, the algebra of functions on $G$, which is an algebra object in $\Rep(G)$ and thus in $\cD$.
\end{prop}

\begin{rem}
The equivariantisation of a braided $G$-crossed tensor category $\cC$ is a nondegenerate braided tensor category if and only if the neutral component $\cC_1$ is a nondegenerate braided tensor category. For the existence of a ribbon element, see \autoref{sec_GRibbon}.
\end{rem}

We discuss some examples of braided $G$-crossed tensor categories and their equivariantisations.
\begin{ex}\label{ex_DrinfeldCenter}~
\begin{enumerate}[wide]
\item The archetypal example of a braided $G$-crossed tensor category is $\smash{\Vect_G}$, the pointed fusion category of $G$-graded vector spaces with the obvious $G$-grading. Up to equivalence, it has a unique braided $G$-crossed structure where the $G$-action is strict and given by $g_*(\C_h)=\C_{ghg^{-1}}$ for $g,h\in G$, with the $G$-braiding being the identity.

In this case, the equivariantisation $\smash{\Vect_G}\sslash G$ is canonically braided equivalent to $\mathcal{Z}(\smash{\Vect_G})$, the Drinfeld centre of $\smash{\Vect_G}$. See \cite{NNW09} for more details.

\item For an arbitrary $G$-graded extension of $\Vect$, we need to \emph{twist} the previous example and consider the braided $G$-crossed tensor category $\smash{\Vect_G^\omega}$ for a $3$-cocycle $\omega\in Z^3(G,\C^\times)$, as in \autoref{ex_tensorVect}. Given $g,h\in G$, define the maps
\begin{equation}\label{eq_maps for twisting g-crossed}
\begin{split}
\gamma_{g,h}(x)&\coloneqq\frac{\omega(g,h,x)\omega(ghx(gh)^{-1},g,h)}{\omega(g, hxh^{-1},h)},\\
\mu_g(x,y)&\coloneqq\frac{\omega(gxg^{-1},g,y)}{\omega(gxg^{-1},gy,g)\omega(g,x,y)}
\end{split}
\end{equation}
for $x,y\in G$. The $G$-action on $\smash{\Vect_G^\omega}$ is on objects the same as before, and the $G$-braiding is again the identity. However, the action now is nonstrict. Indeed, the constraints are, for each $g,h\in G$ and $x,y\in G$,
\begin{equation*}
\tau^g_{x,y}=\mu_g(x,y)\id_{g_*(\C_{xy})},\quad T_2^{g,h}(\C_x)=\gamma_{g,h}(x)\id_{g_*(h_*(\C_x))}.
\end{equation*}

The equivariantisation $\smash{\Vect_G^\omega}\sslash G$ is canonically braided equivalent to the twisted Drinfeld centre $\mathcal{Z}(\smash{\Vect_G^\omega})$. Again, see \cite{NNW09} for more details.

\item Now, assume that $G$ is abelian. Hence the $G$-action is trivial on objects. In this case, $\mu_a\in Z^2(G,\C^\times)$ is a $2$-cocycle for each $a\in G$. The cohomologies $\mu_a$ for all $a\in G$ measure whether the category $\smash{\Vect_G^\omega}\sslash G$ is pointed or not. If $\gamma_a=\delta(l_a)$ for some $l_a\in C^1(G,\C^\times)$, then the simple objects of $\smash{\Vect_G^\omega}\sslash G$ correspond to elements in $\smash{G\times\widehat{G}}$. Specifically, the pair $(a,\gamma)$ corresponds to the object $\C_a$ with the equivariant structure $\eqi_b=\gamma(b)l_a(b)\id_{\C_a}$. The tensor product is determined by the extension of $G$ by $\smash{\widehat{G}}$ given by the $2$-cocycle $\gamma_{a,b}\frac{l_{ab}}{l_a l_b}\in\smash{Z^2(G,\widehat{G})}$. See \cite{GJ16} for details.
\end{enumerate}
\end{ex}


\subsection{Classification of Braided \texorpdfstring{$G$}{G}-Crossed Extensions}

We briefly recall the important result from \cite{ENO10}, Theorem~7.12, the classification of all braided $G$-crossed fusion categories $\cC$ for a given braided fusion category $\cC_1$ equipped with an action of a finite group~$G$. This is also thoroughly treated in \cite{DN21}, where the assumption of semisimplicity and rigidity is dropped, but finiteness in the sense of \cite{EGNO15} is retained (that is, they are working in the setting of finite braided tensor categories). A more detailed discussion of this result, in the form presented below, is given in \autoref*{GLM1sec_ENOENO} of \cite{GLM24a}.

\begin{thm}[\cite{DN21}, Section~8.3; \cite{ENO10}, Theorem~7.12]\label{thm_ourENO_braided}
Let $\cB$ be a finite tensor category with nondegenerate braiding, with an action of a finite group $G$ on $\cB$ by braided autoequivalences. Then, there is a certain obstruction $\mathrm{O}_4\in H^4(G,\C^\times)$, and if and only if this obstruction vanishes, there exists a braided $G$-crossed tensor category $\cC$, where $\cC_1=\cB$ as a braided tensor category with $G$-action. The equivalence classes of extensions associated with the $G$-action form a torsor over $\omega\in H^3(G,\C^\times)$; see \autoref{rem_torsor H3}.
\end{thm}

\begin{rem}\label{rem_torsor H3}~
\begin{enumerate}[wide]
\item For a braided $G$-crossed tensor category $\cC$, we can define for $\omega\in Z^3(G,\C^\times)$ a \emph{twisted} $\cC^\omega$ by modifying the associativity constraint with $\omega$ and \emph{twisting} the $G$-action by multiplying with the scalars in equation \eqref{eq_maps for twisting g-crossed}. See Proposition~2.2 in \cite{EG18b} for details.
\item The braided $G$-crossed extensions and their $G$-equivariantisations of the trivial modular tensor category $\cB=\Vect$ correspond to $\smash{\Vect_G^\omega}$ and its Drinfeld centre, respectively; see \autoref{ex_DrinfeldCenter}.
\end{enumerate}
\end{rem}

We remark that the main assertion in this work on the existence of certain braided $G$-crossed extensions (see \autoref{thm_evenTY}) do not make use of \autoref{thm_ourENO_braided}. Nonetheless, it would be interesting to match our constructions to the work of \cite{ENO10} (see \autoref{prob_ENO}).


\subsection{\texorpdfstring{$G$}{G}-Ribbon Structure}\label{sec_GRibbon}

We also consider a $G$-crossed version of ribbon categories. The following definition is adapted from Lemma~2.3 in \cite{Kir04}; see also Section~4.5 of \cite{Gal22}.
\begin{defi}\label{def_Gtwist}
A \emph{$G$-twist} $\theta$ for a braided $G$-crossed tensor category $\cC=\smash{\bigoplus_{g\in G}\cC_g}$ is a natural isomorphism $\theta_{X_g}\colon X_g\to g_*(X_g)$ for $X_g\in\cC_g$, $g\in G$ such that $\theta_\unit=\id_\unit$ and the diagrams
\begin{equation}\label{eq_G-twist}
\begin{tikzcd}
X_g \otimes Y_h \arrow[rr, "{\theta_{X_g \otimes Y_h}}"] \arrow[d, "{c_{X_g, Y_h}}"'] && (gh)_*(X_g \otimes Y_h) \arrow[d, "{\tau^{gh}_{X_g, Y_h}}"]\\
g_*(Y_h) \otimes X_g \arrow[d, "{c_{g_*(Y_h), X_g}}"'] && (gh)_*(X_g) \otimes (gh)_*(Y_h) \arrow[d, "{T_2^{ghg^{-1}\!\!,g}(X_g) \otimes T_2^{g, h}(Y_h)}"]\\
(ghg^{-1})_*(X_g) \otimes g_*(Y_h) \arrow[rr, "{\substack{(ghg^{-1})_*(\theta_{X_g})\\\otimes g_*(\theta_{Y_h})}}"'] && (ghg^{-1})_*(g_*(X_g))\otimes g_*(h_*(Y_h))
\end{tikzcd}
\end{equation}
and
\begin{equation}\label{eq_G-invarianza}
\begin{tikzcd}
g_*(Y_h) \arrow[d,"\theta_{g_*(Y_h)}"']\arrow[rrr, "g_*(\theta_{Y_h})"]&&& g_*(h_*(Y_h))\\
(ghg^{-1})_*(g_*(Y_h)) &&& (gh)_*(Y_h) \arrow[u, "T_2^{g,h}(Y_h)"'] \arrow[lll, "T_2^{ghg^{-1}\!\!,g}(Y_h)"]
\end{tikzcd}
\end{equation}
commute for all $X_g\in\cC_g$, $Y_h\in\cC_h$, $g,h\in G$.

A $G$-twist in a rigid braided $G$-crossed tensor category is called a \emph{$G$-ribbon} if the following diagram commutes for all $g\in G$ and $X\in\cC_g$:
\begin{equation}\label{eq_selft-duality}
\begin{tikzcd}
\unit \arrow[rr, "(g^{-1})_*(\coeval_{X})"] \arrow[d, "\coeval_{X}"']
 && (g^{-1})_*(X \otimes X^*) \arrow[rr, "g_*^{-1}(\theta_{X} \otimes \id_{X^*})"]
 && g^{-1}_*(g_*(X) \otimes X^*) \arrow[d, "\tau^{g^{-1}}_{g_*(X),X^*}"]\\
X \otimes X^* \arrow[rr, "\id_{X} \otimes \theta_{X^*}"']
 && X \otimes g^{-1}(X^*) \arrow[rr, "T_2^{g^{-1}\!\!,g}(X) \otimes \id"']
 && g_*^{-1}g_*(X) \otimes g_*^{-1}(X^*)
\end{tikzcd}
\end{equation}
\end{defi}

\begin{rem}
In the definition of a $G$-ribbon, if we take $G$ to be the trivial group, the condition in diagram~\eqref{eq_selft-duality} reduces to $(\id_{X} \otimes\theta_{X^*})\circ\coeval_X=(\theta_X\otimes\id_{X^*})\circ\coeval_X$. This condition is equivalent to the more common requirement $\theta_{X^*}=(\theta_X)^*$; see the discussion following equation~\eqref{eq_twist condition} below.
\end{rem}

In the following, we show that $G$-twist and $G$-ribbon on a braided $G$-crossed tensor category descend to the usual notions of twist and ribbon, respectively, on the equivariantisation.

\begin{lem}\label{lem_natural isomorhism}
Let $\cC$ be a braided $G$-crossed tensor category. If $\theta_{X_g}\colon X_g\to g_*X_g$ for $X_g\in\cC_g$, $g\in G$ is a natural isomorphism that satisfies the commutativity of diagram~\eqref{eq_G-invarianza}, then for any equivariant object $(X,\eqi) \in\cC\sslash G$, the isomorphism
\begin{equation*}
\omega_{(X,\eqi)}\colon X=\bigoplus_{h\in G} X_h\overset{\oplus_{h\in G}\theta_{X_h}}{\longrightarrow}\bigoplus_{h\in G}h_*(X_h)\overset{\oplus_{h\in G}\eqi_h^{(h)}}{\longrightarrow}\bigoplus_{h\in G}X_h=X
\end{equation*}
is a natural isomorphism of the identity in $\cC\sslash G$, where $\eqi^{(h)}_h\colon h_*(X_h)\to X_h$ is the restriction of $\eqi_h$ to $X_h\subset X$ for $h\in G$.
\end{lem}
\begin{proof}
In order to verify that $\omega_{(X,\eqi)}$ is a morphism in $\cC\sslash G$, we need to check that $\eqi_g g_*(\omega)=\omega\eqi_g$ for all $g\in G$. Since $\eqi_g g_*(\omega)=\omega\eqi_g$ is an equality of morphisms in $\cC=\bigoplus_{h} \cC_h$, it is enough to check the equality in the components. Hence, denoting
\begin{equation*}
\eqi_g=\bigoplus_{h\in G}\eqi_g^{(ghg^{-1})}\quad\text{and}\quad\omega=\bigoplus_{h\in G}\omega^{(h)},
\end{equation*}
where $\smash{\eqi_g^{(ghg^{-1})}\colon g_*(X_h)\to X_{ghg^{-1}}}$ and $\smash{\omega^{(h)}=\eqi_h^{(h)}\theta_{X_h}\colon X_h\to X_h}$, we need to check for all $g,h\in G$ that
\begin{equation*}
\eqi_g^{(ghg^{-1})}g_*(\omega^{(h)})=\omega^{(ghg^{-1})}\eqi_g^{(ghg^{-1})}.
\end{equation*}

We shall assume, without loss of generality, that the $G$-action on $\cC$ is strict. Thus, diagram~\eqref{eq_G-invarianza} simply becomes
\begin{equation}\label{eq_invarianza}
g_*(\theta_X)=\theta_{g_*X}
\end{equation}
for every object $X\in\cC$ and $g \in G$. Now, we can check the equality, for all $g,h\in G$,
{\allowdisplaybreaks
\begin{align*}
\eqi_g^{(ghg^{-1})} g_*(\omega^{(h)}) &= \eqi^{(ghg^{-1})}_g g_*(\eqi_h^{(h)}) g_*(\theta_{X_h})\\
&= \eqi_{gh} \circ \theta_{g_*(X_h)}\\
&= \eqi_{ghg^{-1}}^{(ghg^{-1})} (ghg^{-1})_*(\eqi_g) \circ \theta_{g_*(X_h)}\\
&= \eqi_{ghg^{-1}}^{(ghg^{-1})} \theta_{X_{ghg^{-1}}} \eqi_g^{ghg^{-1}}\\
&= \omega^{(ghg^{-1})} \eqi_g^{(ghg^{-1})},
\end{align*}
}%
where in the second equality we have used the $G$-equivariance of $\eqi$ and \eqref{eq_invarianza}, in the third the $G$-equivariance of $\eqi$ again and in the fourth the naturality of $\theta$.

Finally, if $f\colon(X,\eqi)\to(Y,\eqi')$ is a morphism in $\cC\sslash G$, then $\eqi_g' g_*(f) = f \eqi_g$ for all $g\in G$ by definition. Taking the $g$-component of $f$, we get $\smash{{\eqi'}_g^{(g)}} g_*(f^{(g)})= f^{(g)}\smash{\eqi_g^{(g)}}$. Using the naturality of $\theta$ and the previous equation, we obtain for all $g \in G$:
{\allowdisplaybreaks
\begin{align*}
\omega^{(g)}_{(Y, \eqi')} f^{(g)} &= \eqi'^{(g)}_g \theta_{Y_g} f^{(g)} \\
&= \eqi'^{(g)}_g g_*(f^{(g)}) \theta_{X_g} \\
&= f^{(g)} \eqi_h^{(g)} \theta_{X_g} \\
&= f^{(g)} \omega^{(g)}_{(X, \eqi)}.
\end{align*}
}%
Therefore, $\omega_{(Y,\eqi')}f=f\omega_{(X,\eqi)}$, i.e.\ $\omega$ is a natural isomorphism of the identity.
\end{proof}

Recall that if $\cB$ is a braided tensor category, with braiding $c$, then a \emph{twist} is a natural isomorphism of the identity $\omega_X\colon X \to X$ such that $\omega_\unit = \id_\unit$ and
\begin{equation}\label{eq_twist condition}
\omega_{X \otimes Y} = (\omega_X \otimes \omega_Y) c_{Y, X} c_{X, Y}
\end{equation}
for all $X,Y\in\cB$. If $\cB$ is left rigid and additionally $\omega_{X^*}=(\omega_X)^*$ for all $X$, then $\omega$ is called a \emph{ribbon}.

\begin{prop}\label{lem_ribbon}
Let $\smash{\cC=\bigoplus_{g\in G}\cC_g}$ be a braided $G$-crossed tensor category and $\theta$ a $G$-twist. Then the natural isomorphism $\omega_{(X, \eqi)}$ of the identity from \autoref{lem_natural isomorhism} defines a twist for $\cC\sslash G$. If $\theta$ is a $G$-ribbon, then $\omega_{(X,\eqi)}$ is a ribbon.
\end{prop}
\begin{proof}
The braiding of $\cC \sslash G$ is given by $\smash{\mathbf{c}_{(X,\eqi),(Y,\gamma)}=\bigoplus_{g\in G}\mathbf{c}_{X_g,Y}}$, where
\begin{equation*}
\begin{tikzcd}
\mathbf{c}_{X_g, Y}\colon X_g \otimes Y \arrow[r, "{c_{X_g, Y}}"] & g_*(Y) \otimes X_g \arrow[rr, "{\gamma_g \otimes \id_{X_g}}"] && Y \otimes X_g.
\end{tikzcd}
\end{equation*}
\autoref{lem_natural isomorhism} states that $\omega$ is indeed a natural isomorphism of the identity. Then, equality~\eqref{eq_twist condition} follows, assuming without loss of generality that the $G$-action is strict, using the naturality of $\theta$ and diagram~\eqref{eq_G-twist}.

Now, if $\cC$ is rigid, then so is $\cC \sslash G$. Indeed, if $(X, \eqi) \in \cC \sslash G$ and $(X^*, \eval_X, \coeval_X)$ is a dual of $X$ as an object in $\cC$, then $(X^*, (\eqi^*)^{-1}) \in \cC \sslash G$ and $(\eval_X, \coeval_X)$ define a dual in $\cC \sslash G$. Now, it is again straightforward to verify, assuming the $G$-action is strict, that diagram \ref{eq_selft-duality} implies that $\omega$ is a ribbon.
\end{proof}

For a tensor category $\cC$, we shall denote by $\smash{\Aut_\otimes(\id_{\cC})}$ the abelian group of tensor automorphisms of the identity. If a group $G$ acts on $\cC$, it induces a $G$-action on $\smash{\Aut_\otimes(\id_{\cC})}$ by group automorphisms. We define
\begin{align*}
\Aut_\otimes^G(\id_{\cC}) &= \{\gamma \in \Aut_\otimes(\id_{\cC}) \mid g_*(\gamma_X) = \gamma_{g_*(X)} \text{ for all } X \in \cC, g \in G\},\\
\Aut_\otimes^{G,*}(\id_{\cC}) &= \{\gamma \in \Aut_\otimes^G(\id_{\cC}) \mid \gamma^*_{X} = \gamma_{g^{-1}(X)} \text{ for all } X_g\in \cC, g\in G\}.
\end{align*}

\begin{prop}\label{prop_torsor twist and ribbon}
Let $\smash{\cC=\bigoplus_{g\in G}\cC_g}$ be a braided $G$-crossed tensor category. The set of all $G$-twists is a torsor over the abelian group $\smash{\Aut_\otimes^G(\id_{\cC})}$. The set of all $G$-ribbons is a torsor over $\smash{\Aut_\otimes^{G,*}(\id_{\cC})}$.
\end{prop}
\begin{proof}
We shall assume without loss of generality that $\cC$ is a \emph{strict} braided $G$-crossed tensor category; see Theorem~5.6 in \cite{G17}. The commutativity of diagram~\eqref{eq_G-twist} corresponds to the equation, for $X \in \cC_g$, $Y \in \cC_h$, $g,h\in G$,
\begin{equation*}
\theta_{X\otimes Y} = (ghg^{-1})_*(\theta_X) \otimes g_*(\theta_Y) \circ c_{g_*Y, X} \circ c_{X,Y}.
\end{equation*}
Now, using the naturality of the $G$-braiding, we also obtain
\begin{equation*}
\theta_{X \otimes Y} = c_{gh_*Y, g_*X} \circ c_{g_*X, h_*Y} \circ (\theta_X \otimes \theta_Y).
\end{equation*}
Let $\theta$ and $\theta'$ be $G$-twists of $\cC$. We shall prove that $\smash{\gamma_X \coloneqq \theta^{-1}_X \theta'_X \colon X \to X}$ is a natural isomorphism in $\smash{\Aut_\otimes^G(\id_{\cC})}$. We observe that
{\allowdisplaybreaks
\begin{align*}
\gamma_{X\otimes Y}&=\theta_{X \otimes Y}^{-1} \circ \theta_{X \otimes Y}' \\
&=\bigl((ghg^{-1})_*(\theta_X) \otimes g_*(\theta_Y) \circ c_{g_*Y, X} \circ c_{X, Y}\bigr)^{-1}\\
&\quad\circ\bigl(c_{gh_*Y, g_*X} \circ c_{g_*X, h_*Y} \circ (\theta_X' \otimes \theta_Y')\bigr)\\
&=c_{X, Y}^{-1}\circ c_{g_*Y, X}^{-1}\circ (ghg^{-1})_*(\theta_X)^{-1} \otimes g_*(\theta_Y)^{-1}\\
&\quad\circ c_{gh_*Y, g_*X} \circ c_{g_*X, h_*Y} \circ (\theta_X' \otimes \theta_Y')\\
&= c_{X, Y}^{-1}\circ g(\theta_Y)^{-1}\otimes \theta_X^{-1} \circ c_{g_*X, h_*Y} \circ (\theta_X' \otimes \theta_Y')\\
&=(\theta_X^{-1} \otimes \theta_Y^{-1})\circ (\theta_X' \otimes \theta_Y')= (\theta_X^{-1}\theta_X')\otimes (\theta_Y^{-1}\theta_Y')\\
&=\gamma_X\otimes \gamma_Y,
\end{align*}
}%
for $X \in \cC_g$, $Y \in \cC_h$, $g,h\in G$, where we have used
\begin{equation*}
c_{g_*Y, X}^{-1} \circ (ghg^{-1})_*(\theta_X)^{-1} \otimes g_*(\theta_Y)^{-1} \circ c_{gh_*Y, g_*X} = g_*(\theta_Y)^{-1} \otimes \theta_X^{-1}
\end{equation*}
and
\begin{equation*}
\theta_X^{-1} \otimes \theta_Y^{-1} = c_{X, Y}^{-1} \circ (g_*(\theta_Y)^{-1} \otimes \theta_X^{-1}) \circ c_{g_*X, h_*Y}
\end{equation*}
in the penultimate and last equality, respectively. These equations follow directly from the naturality of the $G$-braiding. Now,
\begin{equation*}
g_*(\gamma_X) = g_*(\theta_X^{-1} \theta'_X) = g_*(\theta_X^{-1}) g_*(\theta'_X) = \theta_{g_*(X)}^{-1} \theta_{g_*(X)}' = \gamma_{g_*(X)}
\end{equation*}
for all $X \in \cC_g$, $g \in G$.

Conversely, given $\gamma\in\smash{\Aut_\otimes^G(\id_{\cC})}$ and a $G$-twist $\theta$, it is straightforward to check that $\gamma \theta$ is a new twist. Then the group $\smash{\Aut_\otimes^G(\id_{\cC})}$ acts freely and transitively on the set of $G$-twists.

Finally, if $\theta$ and $\theta'$ are $G$-ribbons, then
{\allowdisplaybreaks
\begin{align*}
\gamma_{X^*} &= \theta_{X^*}^{-1} \theta'_{X^*} = g^{-1}_*(\theta_{X})^{-1} g^{-1}_*(\theta'_X) \\
             &= g^{-1}_*(\theta_X^{-1} \theta'_X) = g^{-1}_*(\gamma_X)\\
             &= \gamma_{g^{-1}(X)}
\end{align*}
}%
for all $X \in \cC_g$ and $g \in G$. It is easy to verify that, given $\gamma\in\smash{\Aut_\otimes^{G,*}(\id_{\cC})}$ and a $G$-twist $\theta$, the natural isomorphism $\gamma \theta$ is a new $G$-ribbon.
\end{proof}

Given a braided $G$-crossed tensor category $\cC$, we define a group homomorphism
\begin{align*}
\Omega\colon \Aut_\otimes^G(\id_{\cC}) &\to \Aut_\otimes(\id_{\cC \sslash G}),\\
\Omega(\gamma)_{(X, \eqi)} &= \gamma_X.
\end{align*}
The group homomorphism $\Omega$ restricted to $\smash{\Aut_\otimes^{G,*}(\id_{\cC})}$ takes values in $\smash{\Aut_\otimes^{*}(\id_{\cC \sslash G})}$, the abelian group of tensor isomorphisms of the identity such that $\gamma_{X^*} = \gamma_X^*$.

\begin{cor}
Let $\cC$ be a braided $G$-crossed tensor category. Two $G$-twists $\theta, \theta'$ (respectively $G$-ribbons) produce the same twist (respectively ribbon) on the equivariantisation $\cC \sslash G$ if and only if $\Omega(\theta^{-1} \theta') = \id$.
\end{cor}

\begin{rem}
If $\cC$ is fusion, then $\smash{\Aut_\otimes(\id_{\cC})}$ is canonically isomorphic to $\smash{\widehat{U(\cC)}}$, the character group of $U(\cC)$, called the universal grading group; see \cite{GN08}. The subgroup $\smash{\Aut_\otimes^*(\id_{\cC})}$ consists of elements of order~$2$, i.e.\ it is in particular an elementary abelian $2$-group. Now, in the case of a braided $G$-crossed fusion category, $\smash{\Aut_\otimes^G(\id_{\cC})}=\smash{\Aut_\otimes(\id_{\cC})}$ and $\smash{\Aut_\otimes^{G,*}(\id_{\cC})}=\smash{\Aut_\otimes^*(\id_{\cC})}$. In particular, in the fusion category case, the number of pivotal and spherical structures corresponds to $G$-twists and $G$-ribbons, respectively.
\end{rem}

\begin{ex}
Consider the braided $G$-crossed tensor category $\smash{\Vect_G^\omega}$ from \autoref{ex_DrinfeldCenter}. This category has a canonical $G$-ribbon given by the identity functor. Hence, it follows from \autoref{prop_torsor twist and ribbon} that the $G$-twist structures of $\smash{\Vect_G^\omega}$ are in bijective correspondence with the linear characters $\smash{\widehat{G}}$, and the set of $G$-ribbon structures with $\smash{\widehat{G}_2}$, the linear characters taking values in $\{\pm 1\}$.
\end{ex}


\subsection{Pseudo-Unitary Fusion Categories and \texorpdfstring{$G$}{G}-Graded Extensions}\label{sec_pseudo}

Finally, we also study the notion of pseudo-unitarity in the context of $G$-graded fusion categories. This will later allow us to single out a particular choice of $G$-ribbon structure (see \autoref{cor_TYpseudounitary} and \autoref{thm_evenTY}).

The global dimension, denoted by $\dim(\cC)$, is defined for any fusion category $\cC$ in, e.g., Definition~2.2 of \cite{ENO10}. This definition is intrinsic to $\cC$. Another dimension of a fusion category $\cC$ is the Frobenius-Perron dimension $\operatorname{FPdim}(\cC)\in\R_{>0}$, which only depends on the Grothendieck ring $K_0(\cC)$ of $\cC$; see Section~8 of \cite{ENO10}.

We are interested in the relationship between the global dimension and the Frobenius-Perron dimension of fusion categories, as it relates to the existence of spherical structures for which the associated dimension function coincides with the Frobenius-Perron dimensions, as indicated by the following result; see Proposition~8.23 in \cite{ENO05}.
\begin{prop}\label{prop_equi def pseudo-unitary}
Let $\cC$ be a fusion category (over~$\C$). Then $\cC$ admits a (unique) spherical structure whose quantum (or categorical) dimensions match the Frobenius-Perron dimensions if and only if $\operatorname{FPdim}(\cC) = \dim(\cC)$. Moreover, this spherical structure is unique.
\end{prop}
The above property naturally leads to the following definition (see, e.g., Section~8.4 of \cite{ENO10}).
\begin{defi}\label{def_pseudounitary}
Let $\cC$ be a fusion category (over $\C$). The category $\cC$ is called \emph{pseudo-unitary} if $\operatorname{FPdim}(\cC) = \dim(\cC)$, or equivalently, by \autoref{prop_equi def pseudo-unitary}, if there exists a spherical structure with respect to which the dimension of every simple object is positive (and equal to the Frobenius-Perron dimension).
\end{defi}

The following result is an analogue of Proposition~8.20 in \cite{ENO05}, focusing on the global dimension rather than the Frobenius-Perron dimension.
\begin{prop}\label{prop_global dimension of grraded}
Let $G$ be a finite group and $\smash{\cC=\bigoplus_{g\in G}\cC_g}$ a faithfully $G$-graded fusion category. Then $\dim(\cC_1)=\dim(\cC_g)$ for all $g\in G$, and the global dimension of $\cC$ is given by $\dim(\cC)=|G|\dim(\cC_1)$.
\end{prop}
\begin{proof}
Consider the $\cC_1$-module category $\smash{\bigoplus_{g\in G}\cC_g}$ and the connected multi-fusion category
\begin{equation*}
\operatorname{Fun}_{\cC_1}\biggl(\,\bigoplus_{g \in G} \cC_g, \bigoplus_{h \in G} \cC_h\biggr) = \bigoplus_{g,h\in G} \Hom_{\cC_1}(\cC_g, \cC_h).
\end{equation*}
Now, each $\cC_g$ is an invertible $\cC_1$-bimodule; see Theorem~6.1 of \cite{ENO10}. Therefore, $\smash{\Hom_{\cC_1}(\cC_g, \cC_g) \cong \cC_1}$ as fusion categories, and $\smash{\cC_{h^{-1}g} = \Hom_{\cC_1}(\cC_h, \cC_g)}$ as a $\cC_1$-bimodule category. It then follows from Proposition~2.17 in \cite{ENO10} that the global dimension of each $\cC_g$ is $\dim(\cC_1)$. Therefore,
\begin{equation*}
\dim(\cC) = \sum_{g \in G} \dim(\cC_g) = |G| \dim(\cC_1).\qedhere
\end{equation*}
\end{proof}

We now relate the pseudo-unitarity of a (braided) $G$-graded ($G$-crossed) fusion category to that of its neutral component and that of its equivariantisation.
\begin{cor}\label{cor_pseudo}
Let $G$ be a finite group.
\begin{enumerate}[wide]
\item Let $\cC$ be a faithfully $G$-graded fusion category. Then $\cC$ is pseudo-unitary if and only if the neutral component $\cC_1$ is pseudo-unitary.
\item Suppose that $\cC$ is a fusion category with a $G$-action. Then the equivariantisation satisfies $\dim(\cC \sslash G) = |G| \dim(\cC)$. Hence, $\cC \sslash G$ is pseudo-unitary if and only if $\cC$ is pseudo-unitary.
\item Let $\cB$ be a nondegenerate braided fusion category with a $G$-action by braided tensor autoequivalences. Then the associated (in the sense of \autoref{thm_ourENO_braided}) nondegenerate braided fusion categories $\cC\sslash G$ are pseudo-unitary if and only if $\cB=\cC_1$ is pseudo-unitary.
\end{enumerate}
\end{cor}
\begin{proof}
\begin{enumerate}[wide]
\item Proposition~8.20 in \cite{ENO05} implies $\operatorname{FPdim}(\cC) = |G| \operatorname{FPdim}(\cC_1)$. Then, by \autoref{prop_global dimension of grraded}, $\operatorname{FPdim}(\cC) = \dim(\cC)$ if and only if $\operatorname{FPdim}(\cC_1) = \dim(\cC_1)$.
\item If $G$ acts on $\cC$, the semidirect-product fusion category $\cC \rtimes G$ is $G$-graded and Morita equivalent to $\cC \sslash G$ (see \cite{Nik08}, Proposition~3.2). Since Frobenius-Perron dimension and global dimension are invariant under Morita equivalence (see Proposition 2.15 and Proposition 8.12 in \cite{ENO10}, respectively), it follows that $\dim(\cC \sslash G) = |G| \dim(\cC)$ and $\operatorname{FPdim}(\cC \rtimes G) = |G| \operatorname{FPdim}(\cC)$.
\item This follows directly from the previous items.\qedhere
\end{enumerate}
\end{proof}

\begin{rem}
In particular, any gauging of a pseudo-unitary fusion category is again pseudo-unitary. In this paper, we are primarily interested in gauging pointed braided fusion categories, which are indeed pseudo-unitary.
\end{rem}

Below, in \autoref{cor_TYpseudounitary} and \autoref{thm_evenTY}, we construct certain braided $G$-crossed extensions and also write down a pseudo-unitary $G$-ribbon element, which we then know is unique. More systematic would be a version of the above classification statement, \autoref{thm_ourENO_braided}, that inherently uses the setting of a ribbon category with ribbon autoequivalences (see \autoref*{GLM1problem_RibbonENO} in \cite{GLM24a}).


\section{\texorpdfstring{$\Z_2$}{Z\_2}-Crossed Extensions of Tambara-Yamagami Type}\label{sec_TY}

We give a first nontrivial example of a braided $G$-crossed extension, namely one for which the tensor structure is of Tambara-Yamagami type. The example is also studied in greater detail in \autoref*{GLM1sec_TY} of \cite{GLM24a}. In \autoref{sec_TYEven}, we shall then generalise this to allow for more than one simple object in the nontrivial graded component.

From now on, let $G=\langle g\rangle\cong\Z_2$. Tambara and Yamagami classified all $\Z_2$-graded fusion categories in which all but one of the simple objects are invertible \cite{TY98}. The braided $G$-crossed structures on these were classified in \cite{Gal22}.
\begin{defi}
Let $\Gamma$ be a finite, abelian group, $\sigma\colon\Gamma\times\Gamma\to\C^\times$ a symmetric, nondegenerate bimultiplicative form and $\eps\in\{\pm1\}$ a sign choice. The \emph{Tambara-Yamagami category} $\TY(\Gamma,\sigma,\eps)$ is the semisimple $\Z_2$-graded tensor category
\begin{equation*}
\TY(\Gamma,\sigma,\eps) = \Vect_\Gamma \oplus \Vect
\end{equation*}
with simple objects $\C_a$, $a\in \Gamma$, and $\X$, fusion rules
\begin{equation*}
\C_a\otimes \C_b=\C_{a+b},\quad
\C_a\otimes \X=\X\otimes \C_a =\X,\quad
\X\otimes \X=\bigoplus_{t\in\Gamma} \C_t
\end{equation*}
for $a,b\in\Gamma$ and the following nontrivial associators
{\allowdisplaybreaks
\begin{align*}
(\C_a\otimes \X)\otimes \C_b
&=X
\overset{\sigma(a,b)}{\xrightarrow{\hspace*{2.97cm}}}
\X=\C_a\otimes (\X\otimes \C_b),\\
(\X\otimes \C_a)\otimes \X
&=\bigoplus_{t\in \Gamma} \C_t
\overset{\sigma(a,t)}{\xrightarrow{\hspace*{1.6cm}}}
\bigoplus_{t\in \Gamma} \C_t
=\X\otimes (\C_a\otimes \X),\\
(\X\otimes \X)\otimes \X
&=\bigoplus_{t\in \Gamma} \X
\overset{\eps|\Gamma|^{-1/2}\sigma(t,r)^{-1}}{\xrightarrow{\hspace*{1.775cm}}}
\bigoplus_{r\in \Gamma}\X
=\X\otimes (\X\otimes \X).
\end{align*}
}%
\end{defi}

There is a rigid structure on $\TY(\Gamma,\sigma,\eps)$ with (left) dual objects $\smash{\C_a^*=\C_{-a}}$ and $\X^*=\X$, and $\smash{\coeval_\X=\iota_{\C_0}}$ and $\smash{\eval_\X=\eps|\Gamma|^{1/2}\pi_{\C_0}}$, where $\iota_{\C_t}$ and $\pi_{\C_t}$ are the canonical embeddings and projections, respectively, for the direct sum $\X\otimes\X=\bigoplus_{t\in\Gamma}\C_t$, and otherwise the obvious choices. The two zigzag identities relating $\eval_\X$ and $\coeval_\X$ hold as both the associator and inverse associator on $\X\otimes \X\otimes\X$ at $t=0$ are $\smash{\eps|\Gamma|^{-1/2}}$.

\medskip

We then upgrade the above categories to braided $\Z_2$-crossed tensor categories. To this end, we first need to introduce the following Gauss sum. Given a finite, abelian group $\Gamma$ and a nondegenerate quadratic form $q\colon\Gamma\to\C^\times$, we define
\begin{equation}\label{eq_gauss}
G(\Gamma,q^{-1})\coloneqq|\Gamma|^{-1/2}\sum_{a\in \Gamma} q(a)^{-1}=\e(\sign(\Gamma,q^{-1})/8)=\e(-\sign(\Gamma,q)/8),
\end{equation}
where $\sign$ denotes the signature (a number in $\Z_8$) of a discriminant form.

\begin{thm}[\cite{Gal22}, Lemma~4.8, Theorem~4.9, Proposition~4.12]\label{thm_TYribbon}
Consider the Tambara-Yamagami category $\TY(\Gamma,\sigma,\eps)$ and $G = \langle g \rangle \cong \Z_2$.
\begin{enumerate}
\item There are two actions of $G$ on $\TY(\Gamma,\sigma,\eps)$ such that the action on $\smash{\Vect_\Gamma}$ is  $\smash{g_*\C_a=\C_{-a}}$ with trivial tensor structure $\smash{\tau^g_{\C_a,\C_b}=\id}$ and trivial composition $\smash{T_2^{g,g}(\C_a)=\id}$, $a,b\in\Gamma$. Namely, the actions $g_*\X = \X$ with trivial tensor structures $\smash{\tau^g_{\C_a,\X}=\tau^g_{\X,\C_a}=\tau^g_{\X,\X}=\id}$ and with composition structures $T_2^{g,g}(\C_a) = \id$ and $T_2^{g,g}(\X) = \pm \id$.
\smallskip
\item The nonstrict action $T_2^{g,g}(\X) = -\id$ does not admit a $\Z_2$-braiding.
\smallskip
\item For the strict $\Z_2$-action, the $\Z_2$-braidings are in bijection with pairs $(q,\alpha)$ of a nondegenerate quadratic form $q\colon \Gamma \to \C^\times$ with associated bimultiplicative form $\sigma(a,b) = q(a+b) q(a)^{-1} q(b)^{-1}$ and a choice $\alpha$ of square root of
\begin{equation*}
\alpha^2 = \eps\, G(\Gamma,q^{-1}).
\end{equation*}
Then the $\Z_2$-braiding is, for $a,b\in\Gamma$,
{\allowdisplaybreaks
\begin{align*}
\C_a \otimes \C_b
&= \C_{a+b}
\overset{\sigma(a,b)}{\xrightarrow{\hspace*{1.27cm}}}
\C_{a+b} = \C_a \otimes \C_b,\\
\C_a \otimes \X
&= \X
\overset{q(a)^{-1}}{\xrightarrow{\hspace*{2.32cm}}}
\X = \X \otimes \C_a,\\
\X \otimes \C_a
&= \X
\overset{q(a)^{-1}}{\xrightarrow{\hspace*{2.32cm}}}
\X = \C_{-a} \otimes \X,\\
\X \otimes \X
&= \bigoplus_{t \in \Gamma} \C_t
\overset{\alpha \, q(t)}{\xrightarrow{\hspace*{0.86cm}}}
\bigoplus_{t \in \Gamma} \C_t = \X \otimes \X.
\end{align*}
}%
\item The $\Z_2$-ribbon structures for a given $(q,\alpha)$ are in correspondence with a choice of $\beta = \pm \alpha^{-1}$. Then the ribbon twist is, for $a\in\Gamma$,
\begin{equation*}
\C_a\overset{q(a)^{2}}{\xrightarrow{\hspace*{1cm}}}\C_a\quad\text{and}\quad\X\overset{\beta}{\xrightarrow{\hspace*{1cm}}}\X.
\end{equation*}
\end{enumerate}
We denote by $\TY(\Gamma,\sigma,\eps\,|\,q,\alpha,\beta)$ the $\Z_2$-crossed ribbon Tambara-Yamagami category associated with these data.
\end{thm}

Note that because $\alpha^2=\eps\,\e(-\sign(\Gamma,q)/8)$ by equation~\eqref{eq_gauss}, it follows that $\alpha$ and $\beta$ are $16$-th roots of unity, and in particular $|\alpha|=|\beta|=1$. Moreover, $\alpha\beta\in\{\pm1\}$ by definition.

\begin{rem}\label{rem_alphaEquivalence}
As we explain in more detail in \autoref*{GLM1rem_alphaEquivalence} of \cite{GLM24a}, the two choices of $\alpha$ yield equivalent $G$-crossed ribbon categories $\TY(\Gamma,\sigma,\eps\,|\,q,\alpha,\beta)$.
\end{rem}

\begin{rem}\label{rem_TYext}
$\TY(\Gamma,\sigma,\eps\,|\,q,\alpha,\beta)$ is a braided $\Z_2$-crossed extension of the braided tensor category $\smash{\Vect_\Gamma^{q^2}}\cong\smash{\Vect_{\Gamma}^{\sigma,1}}$ (with $\sigma$ symmetric) discussed in \autoref{ex_braidedVect}, with ribbon structure given by $\theta_{\C_a}=q(a)^2$. In other words, in order to define a braided $\Z_2$-crossed extension of Tambara-Yamagami type, we need to, among other things, choose a square root of a quadratic form. (But note that the quadratic form $q^2$ on $\Gamma$ may be degenerate if $|\Gamma|$ is even.)
\end{rem}

\begin{rem}\label{rem_TY}
A symmetric, nondegenerate bimultiplicative form $\sigma\colon \Gamma\times \Gamma\to\C^\times$ admits $|\Gamma/2\Gamma|$ many choices of quadratic form $q\colon \Gamma\to\C^\times$ with associated bimultiplicative form $\sigma$.
\end{rem}

There is a somewhat natural choice of the sign $\alpha\beta$. Indeed, the following assertion is not difficult to verify, but see \autoref*{GLM1cor_TYpseudounitary} in \cite{GLM24a} for a proof.
\begin{prop}\label{cor_TYpseudounitary}
The quantum dimensions of the above $\Z_2$-crossed ribbon fusion category $\TY(\Gamma,\sigma,\eps\,|\,q,\alpha,\beta)$ coincide with the Frobenius-Perron dimensions if and only if $\alpha \beta = \eps$.
\end{prop}
In particular, $\TY(\Gamma,\sigma,\eps\,|\,q,\alpha,\beta)$ is pseudo-unitary; see also \autoref{cor_pseudo}. By the definition of $\alpha$ and $\beta$, the condition $\alpha\beta=\eps$ is equivalent to the equation $\beta/\alpha=\e(\sign(\Gamma,q)/8)$.

\begin{prop}[\cite{GNN09}, Proposition~5.1]\label{prop_TY_modular}
The equivariantisation of the braided $\Z_2$-crossed tensor category $\TY(\Gamma,\sigma,\eps\,|\,q,\alpha,\beta)\sslash\Z_2$ has a nondegenerate braiding and is hence a modular tensor category if and only if $|\Gamma|$ is odd.
\end{prop}

Because it is often natural to demand that the equivariantisation of a braided $G$-crossed tensor category be modular, e.g.\ in the context of \voa{}s (cf.\ \cite{GLM24a}), the special case of $\Gamma$ of odd order is the most interesting one.

\medskip

Assume in the following that $|\Gamma|$ is odd. Then, as explained in \autoref{ex_braidedVect}, for any nondegenerate quadratic form $Q$ on $\Gamma$, the equivalence class $\smash{\Vect_\Gamma^Q}$ has a distinguished representative $\smash{\Vect_\Gamma^{\sigma,1}}$, for which the representing abelian $3$-cocycle $(\sigma,\omega)$ has trivial associator $\omega=1$ and braiding given by the symmetric bimultiplicative form $\smash{\sigma=B_Q^{1/2}}$.

Hence, by \autoref{rem_TYext}, the Tambara-Yamagami categories $\TY(\Gamma,\sigma,\eps\,|\,q,\alpha,\beta)$ are braided $\Z_2$-crossed extensions of $\smash{\Vect_\Gamma^Q}$, concretely represented by $\smash{\Vect_\Gamma^{\sigma,1}}$. Since $\Gamma$ is of odd order, there is a unique quadratic form $q$ with associated bimultiplicative form $\smash{B_q=\sigma}$ (see \autoref{rem_TY}), and this quadratic form must coincide with the unique square root of $Q$, i.e.\ $\smash{q=Q^{1/2}}$.

Also, since $|\Gamma|$ is odd, $\smash{\Vect_\Gamma^Q}$ has a unique ribbon structure (cf.\ \autoref{sec_GRibbon}). The ribbon twist is given by $\theta_{\C_a}=Q(a)=q(a)^2$ for $a\in\Gamma$, and with this ribbon structure the quantum dimensions agree with the Frobenius-Perron dimensions; in particular, $\smash{\Vect_\Gamma^Q}$ is pseudo-unitary. Hence, the $\Z_2$-crossed ribbon structure on $\TY(\Gamma,\sigma,\eps\,|\,q,\alpha,\beta)$ is an extension of the ribbon structure on $\smash{\Vect_\Gamma^Q}$, and yields coinciding quantum and Frobenius-Perron dimensions if and only if $\beta=\eps/\alpha$ by \autoref{cor_TYpseudounitary}.

Summarising the above considerations, we obtain the following statement:
\begin{cor}\label{cor_TY}
Suppose $|\Gamma|$ is odd. The modular tensor category $\smash{\Vect_\Gamma^Q}$ with strict $\Z_2$-action $g_*\C_a=\C_{-a}$ has two braided $\Z_2$-crossed extensions with positive quantum dimensions, namely, for $\eps=\pm 1$,
\begin{equation*}
\Vect_\Gamma^Q[\Z_2,\eps] \coloneqq \TY(\Gamma,B_Q^{1/2},\eps \mid Q^{1/2}, \alpha,\eps/\alpha).
\end{equation*}
Here, $\alpha$ is one of the solutions of $\alpha^2=\eps\,G(\Gamma,q^{-1})=\eps\,\e(-\sign(\Gamma,q)/8)$, which both give equivalent extensions. The equivariantisations $\smash{\Vect_\Gamma^Q[\Z_2,\eps]\sslash\Z_2}$ are again modular.
\end{cor}
We can conclude that these are the unique two braided $\Z_2$-crossed extensions of $\smash{\Vect_\Gamma^Q}$ with the given $\Z_2$-action, in agreement with the predictions of \cite{ENO10}; see \autoref{thm_ourENO_braided}.

The equivariantisation $\smash{\Vect_\Gamma^Q}[\Z_2,\eps]\sslash\Z_2$ was described in detail in \cite{GNN09} and will be recovered as a special case of our generalisation in \autoref{sec_equiv}.

\begin{rem}\label{rem_Gauss}
If $|\Gamma|$ is odd, the Gauss sum $G(\Gamma,q^{-1})$ over $q^{-1}=Q^{-1/2}$ can be expressed in a simple way in terms of the signature of the discriminant form $\sign(\Gamma,Q)$ (see, e.g., \cite{Sch09}):
\begin{align*}
G(\Gamma,q^{-1})&=|\Gamma|^{-1/2}\sum_{a\in \Gamma}q(a)^{-1}
=\e(\sign(\Gamma,q^{-1})/8)\\
&=\left(\frac{2}{|\Gamma|}\right)G(\Gamma,Q^{-1})=\left(\frac{2}{|\Gamma|}\right)\e(-\sign(\Gamma,Q)/8),
\end{align*}
where $(\frac{\cdot}{\cdot})$ is the Kronecker symbol, and for $n$ odd $(\frac{2}{n})=(\frac{n}{2})=(-1)^{(n^2-1)/8}$. With the above formula, the ribbon twist eigenvalue on $\X$, up to an irrelevant sign, is
\begin{equation*}
\theta_\X=\beta=\eps/\alpha=\pm \left(\eps\left(\frac{2}{|\Gamma|}\right)\e(\sign(\Gamma,Q)/8)\right)^{1/2}.
\end{equation*}
\end{rem}


\section{\texorpdfstring{$\Z_2$}{Z\_2}-Crossed Extensions for Even Groups}\label{sec_TYEven}

In this section, as the main result of this work, we construct for $G=\langle g\rangle\cong\Z_2$ explicitly the braided $G$-crossed extension
\begin{equation*}
\Vect_\Gamma^Q[\Z_2,\eps]=\Vect_\Gamma\oplus\Vect_{\Gamma/2\Gamma}
\end{equation*}
of the (nondegenerate) braided fusion category $\smash{\Vect_\Gamma^Q}$ for a discriminant form $(\Gamma, Q)$ with a $\Z_2$-action given by $-\!\id$ on $\Gamma$. This includes, in particular, a description of the associators, $\Z_2$-braiding and $\Z_2$-ribbon structure.

Here, in contrast to the previously known results in \autoref{sec_TY}, the finite, abelian group $\Gamma$ may have even order, leading to new examples of braided $\Z_2$-crossed tensor categories. Through the process of equivariantisation, these also provide new examples of modular tensor categories (see \autoref{sec_equiv}).

As discussed in \autoref*{GLM1thm_latmain} of \cite{GLM24a}, these categories correspond to the representation categories of orbifolds of lattice \voa{}s under reflection automorphisms, whose Grothendieck rings were described in \cite{ADL05}.

\smallskip

The input data for the construction of these braided $G$-crossed categories are produced in \autoref*{GLM1sec_idea2_infTY} of \cite{GLM24a} by using the idea that braided $G$-crossed extensions commute in a certain sense with condensations (i.e.\ going to the local modules over a commutative, associative algebra object in a braided tensor category in the sense of \cite{Par95,KO02}). Concretely, this idea is applied to the situation
\begin{equation*}
\begin{tikzcd}
\Vect_{\R^d}^\DiscQ
\arrow[rightsquigarrow]{d}{\text{cond.}}
\arrow[hookrightarrow]{rrr}{\text{$\Z_2$-crossed ext.}}
&&&
\Vect_{\R^d}^\DiscQ \oplus \Vect
\arrow[rightsquigarrow]{d}{\text{cond.}}
\\
\Vect_\Gamma^Q
\arrow[hookrightarrow]{rrr}{\text{$\Z_2$-crossed ext.}}
&&&
\Vect^Q_\Gamma[\Z_2,\eps]
\end{tikzcd}
\end{equation*}
where $(\Gamma,Q)$ corresponds to an isotropic subgroup (meaning an even lattice~$L$) in the quadratic space $(\R^d,\DiscQ)$, in order to produce the braided $\Z_2$-crossed tensor category $\smash{\Vect_\Gamma^Q}[\Z_2,\eps]$ with all structures.

In principle, for this we need an infinite version of a Tambara-Yamagami category associated with the infinite, abelian group $\R^d$ with automorphism $-\!\id$, which should be $\smash{\Vect_{\R^d}^\DiscQ}\oplus\Vect$ as abelian category, but for which we cannot define a tensor product in the usual sense. (On the level of tensor categories, without braiding, this was solved in \cite{Mar25}, which appeared after the completion of this manuscript.)

Nonetheless, we use this approach in \cite{GLM24a} to very explicitly determine the data that \emph{should} define $\smash{\Vect_\Gamma^Q}[\Z_2,\eps]$. We then use this in \autoref{sec_evenTY} as input for a rigorous (but without \cite{GLM24a} ad hoc seeming) definition of a braided $\Z_2$-crossed tensor category $\LM(\Gamma,\sigma,\omega,\bar{\delta},\eps\,|\,q,\alpha,\beta)$. We prove that this is indeed a braided $\Z_2$-crossed extension of $\smash{\Vect_\Gamma^Q}$ and hence must coincide with $\smash{\Vect_\Gamma^Q}[\Z_2,\eps]$.

\smallskip

Finally, in \autoref{sec_latticedata} we realise the discriminant form $(\Gamma,Q)$ explicitly as dual quotient $\Gamma=L^*\!/L$ of some even lattice $L$. That is, we view $\smash{\Vect_\Gamma^Q=\Vect_\Gamma^{(\sigma,\omega)}}$ concretely as the condensation
\begin{equation*}
\Vect_{\R^d}^\DiscQ=\Vect_{\R^d}^{(\DiscS,\DiscO)}\leadsto\Vect_\Gamma^Q=\Vect_\Gamma^{(\sigma,\omega)}
\end{equation*}
of the infinite pointed braided fusion category $\smash{\Vect_{\R^d}^\DiscQ}$ by the even lattice $L$, as discussed in detail in \autoref*{GLM1sec_latticeDiscriminantForm} of \cite{GLM24a}. Then, we discuss how these lattice data produce the data used to define the braided $\Z_2$-crossed tensor category $\LM(\Gamma,\sigma,\omega,\bar{\delta},\eps\,|\,q,\alpha,\beta)$. If $L$ is positive-definite, the braided $\Z_2$-crossed tensor category will in this way appear for the $\Z_2$-orbifold of the corresponding lattice \voa{}, which we discuss in \cite{GLM24a}.


\subsection{Special Abelian Cocycles}\label{sec_specialcoc}

Recall from \autoref{ex_braidedVect} that for a quadratic form $Q$ over a finite, abelian group~$\Gamma$ (i.e.\ a discriminant form if $Q$ is nondegenerate), we can consider a representing abelian $3$-cocycle $(\sigma, \omega)$ on~$\Gamma$, namely such that $\sigma(a, a) = Q(a)$ for all $a \in \Gamma$. In the following, we aim to show that every pair $(\Gamma,Q)$ has a special representing abelian $3$-cocycle that is particularly suited for the purposes of this text.

\medskip

For the following discussion, we sometimes refer to the orthogonal decomposition of a discriminant form into its (indecomposable) \emph{Jordan components}. For notation and further details, we refer the reader to \autoref{tab_disc} and \cite{CS99,Sch09}.

\begin{defi}
Let $\Gamma$ be an abelian group. A normalised symmetric 2-cochain $\sigma\colon \Gamma \times \Gamma \to \C^\times$ is called \emph{abelian} if $\partial(\sigma)(a;b,c)\coloneqq\sigma(a,b)\sigma(a,c)\sigma(a,b+c)^{-1}\in \{\pm 1\}$ for $a,b,c\in\Gamma$ and the map
\begin{equation}\label{eq_hom}
\partial(\sigma)(\,\cdot\,;b,c)\colon \Gamma \to \{\pm 1\}, \quad a \mapsto \partial(\sigma)(a;b,c),
\end{equation}
is a group homomorphism for every pair $b, c \in \Gamma$. The subset of all abelian symmetric 2-cochains is an abelian subgroup of the 2-cochains and will be denoted by $C^2_{\mathrm{sym}}(\Gamma, \C^\times)$.
\end{defi}
Note that $\partial(\sigma)$ is by definition symmetric in the last two arguments.

\begin{lem}\label{lem_definition partial homomorphism}
If $\sigma \in C^2_{\mathrm{sym}}(\Gamma, \C^\times)$, then the pair $(\sigma, \partial(\sigma))$ is an abelian 3-cocycle. The map $C^2_{\mathrm{sym}}(\Gamma, \C^\times) \to Z^3_\mathrm{ab}(\Gamma, \C^\times)$, $\sigma \mapsto (\sigma, \partial(\sigma))$ is a group homomorphism that induces a well-defined homomorphism
\begin{equation*}
\partial\colon C^2_{\mathrm{sym}}(\Gamma, \C^\times) \to H^3_\mathrm{ab}(\Gamma, \C^\times)
\end{equation*}
with kernel $\ker(\partial) = \{\sigma \in C^2_{\mathrm{sym}}(\Gamma, \C^\times) \mid \sigma(a, a) = 1 \text{ for all } a \in \Gamma\}$.
\end{lem}
\begin{proof}
By equation~\eqref{eq_hom}, the map $\partial$ can be viewed as a group homomorphism
\begin{equation*}
C^2_{\mathrm{sym}}(\Gamma, \C^\times) \to Z^2(\Gamma, \widehat{\Gamma}), \quad \sigma \mapsto \partial(\sigma).
\end{equation*}
It can be readily verified that, in general, for each $\alpha \in Z^2(\Gamma, \widehat{\Gamma})$, the function $\alpha(a; b, c)$ defines a 3-cocycle in $Z^3(\Gamma, \C^\times)$.

Moreover, since $\partial(\sigma)(a; b, c) = \partial(\sigma)(a; c, b)$, the abelian (or hexagon) condition in equation \eqref{eq_abelian cocycle equations} for $\omega=\partial(\sigma)$ is given by
\begin{equation*}
\partial(\sigma)(a; b, c)^{-1} = \partial(\sigma)(a; b, c)^{-1}, \quad
\partial(\sigma)(c; a, b) = \partial(\sigma)(c; a, b)^{-1}
\end{equation*}
for $a,b,c\in\Gamma$. Since $\partial(\sigma)(a, b, c) \in \{\pm 1\}$, these conditions are satisfied. Altogether, $(\sigma, \partial(\sigma))$ lies in $Z^3_\mathrm{ab}(\Gamma, \C^\times)$ for $\sigma \in C^2_{\mathrm{sym}}(\Gamma, \C^\times)$.

Finally, an arbitrary abelian 3-cocycle $(\sigma,\omega)\in Z^3_\mathrm{ab}(\Gamma, \C^\times)$ is cohomologous to the trivial cocycle if and only if $\sigma(a,a)=1$ for all $a\in \Gamma$. In particular, the same holds for $(\sigma,\partial(\sigma))$.
\end{proof}

We shall show that every quadratic form has an associated abelian 3-cocycle that comes from an abelian symmetric 2-cochain. To do this, and to present concrete nontrivial examples, we need to construct an abelian symmetric 2-cochain for two families of discriminant forms (cf.\ \autoref{tab_disc}).

\begin{ex}\label{ex_abelina symmetric cyclic}
Let $\Gamma = \Z_n$ be a cyclic group of even order. We represent the elements in $\Z_n$ by integers $0\leq x<n$. A generator of the abelian group of quadratic forms over $\Gamma$ is given by
\begin{equation*}
Q_{2n}^0(x) \coloneqq \e\left(\frac{x^2}{2n}\right),\quad 0\leq x< n,
\end{equation*}
that is, any other quadratic form over $\Gamma$ is obtained as a power of $Q_{2n}^0$. (In the special case where $n=2^k$, $k\geq1$, is a power of 2, this discriminant form is indecomposable and denoted by the Jordan symbol $n_1^{+1}$, see \autoref{tab_disc}. The other Jordan symbols $n_t^{\pm1}$ are realised by raising $Q_{2n}^0$ to the power of $t$.)

An abelian symmetric 2-cochain associated with $Q_{2n}^0$ is given by
\begin{equation*}
\sigma_{2n}^0(x, y) \coloneqq \e\left(\frac{xy}{2n}\right),\quad 0 \leq x, y < n.
\end{equation*}
The associated 3-cocycle, measuring the failure of $\sigma_{2n}^0$ to be bimultiplicative, is
\begin{equation*}
\partial(\sigma^0_{2n})(a;x,y)=(-1)^{a\rho(x,y)}
\end{equation*}
where
\begin{equation}\label{eq_rho}
\rho(x, y) \coloneqq
\begin{cases}
1 & \text{if } x + y \geq n, \\
0 & \text{otherwise}
\end{cases}
\end{equation}
for $0\leq x,y< n$.
\end{ex}

\begin{ex}\label{ex_abelina symmetric center drinfeld}
Let $\Gamma_k = \Z_{2^k}\oplus \Z_{2^k}$ for any $k \geq 1$, with representatives $(x_1,x_2)$ for $0\leq x_1,x_2<2^k$. We endow $\Gamma_k$ with the nondegenerate quadratic forms
\begin{equation*}
Q_k^+(x_1,x_2) \coloneqq \e\left(\frac{x_1x_2}{2^k}\right), \quad Q_k^-(x_1, x_2) \coloneqq \e\left(\frac{x_1^2+ x_1x_2+ x_2^2}{2^k}\right),
\end{equation*}
where $0\leq x_1,x_2<2^k$.

The discriminant form $(\Gamma_k, Q_k^+)$ corresponds to the Drinfeld centre of the pointed category $\smash{\Vect_{\Z_{2^k}}}$, while $(\Gamma_k, Q_k^-)$ is a generalisation of the 3-fermion, in the sense that it is anisotropic. Both discriminant forms are indecomposable and their Jordan symbols are $(2^k)_\II^{+2}$ and $(2^k)_\II^{-2}$, respectively (see \autoref{tab_disc}).

Abelian symmetric 2-cochains associated with $Q_k^{\pm}$ are given by
\begin{align*}
\sigma_k^+((x_1, x_2), (y_1, y_2)) &\coloneqq \e\left(\frac{x_1 y_2 + x_2 y_1}{2^{k+1}}\right), \\
\sigma_k^-((x_1, x_2), (y_1, y_2)) &\coloneqq \e\left(\frac{2x_1y_1 + 2x_2y_2 + x_1 y_2 + x_2 y_1}{2^{k+1}}\right),
\end{align*}
where $0 \leq x_i, y_i < 2^k$.

In these cases, the 3-cocycles $\partial(\sigma_k^{\pm})$ coincide and are given by
\begin{align*}
\partial(\sigma_k^{\pm})((a_1, a_2); (x_1, x_2), (y_1, y_2)) &= (-1)^{a_1\rho(x_2, y_2) + a_2\rho(x_1, y_1)},
\end{align*}
where $0 \leq x_i, y_i < 2^k$, and $\rho$ is defined in equation \eqref{eq_rho}.
\end{ex}

Without loss of generality, given a quadratic form $Q$ on $\Gamma$, we may assume that the representing abelian 3-cocycle $(\sigma,\omega)$ has the following nice form, facilitating the subsequent computations:
\begin{prop}\label{ass_strongdisc}
For any abelian group $\Gamma$, the group homomorphism
\begin{equation*}
\partial\colon C^2_{\mathrm{sym}}(\Gamma, \C^\times) \to H^3_\mathrm{ab}(\Gamma, \C^\times)
\end{equation*}
is surjective. In other words, any (possibly degenerate) quadratic form $Q$ on $\Gamma$ has a representing abelian 3-cocycle of the form $(\sigma,\omega)=(\sigma, \partial(\sigma))$ for some abelian symmetric 2-cochain~$\sigma$. In particular, these abelian 3-cocycles $(\sigma,\omega)$ satisfy the following properties:
\begin{enumerate}
\item\label{item_prop1} $\sigma$ is normalised, i.e.\ $\sigma(a, 1) = \sigma(1, a) = 1$ for all $a \in \Gamma$,
\item\label{item_prop2} $\sigma$ is symmetric, i.e.\ $\sigma(a, b) = \sigma(b, a)$ for all $a, b \in \Gamma$,
\item\label{item_prop3} for fixed $b, c \in \Gamma$, $\omega(\cdot, b, c)$ defines a group homomorphism $\Gamma \to \C^\times$,
\item\label{item_prop4} $\sigma(a, b+c)\sigma(a, b)^{-1}\sigma(a, c)^{-1} = \omega(a, b, c)$ for all $a, b, c \in \Gamma$.
\end{enumerate}
\end{prop}

This implies that $\sigma^2 = B_Q$, the bilinear form associated with $Q$, and that $\omega$ only takes values in $\{\pm1\}$, so that $\omega(\cdot, b, c)$ defines a group homomorphism $\Gamma/2\Gamma \to \C^\times$ for all $b, c \in \Gamma$. Additionally, $\omega(a, b, c) = \omega(a, c, b)$ for all $a, b, c \in \Gamma$.

In the following, by a slight abuse of notation, we shall allow the first argument of $\omega$ to be in the quotient $\Gamma/2\Gamma$ rather than $\Gamma$ itself, provided that $\omega$ is of the above special form.

We remark that if $|\Gamma|$ is odd, then the above properties reduce to the choice of abelian $3$-cocycle $(\sigma,\omega)$ in the second half of \autoref{sec_TY}. That is, the only solution in this case is given by $\omega=1$ and $\sigma$ the unique bimultiplicative square root of $B_Q$ (see \autoref{ex_braidedVect}).

\begin{proof}
First, suppose that $\Gamma$ is odd. For a quadratic form $Q$ on $\Gamma$, the symmetric bilinear form $\sigma(a, b) \coloneqq B_Q(a,b)^{1/2}=(Q(a + b)/(Q(a)Q(b)))^{1/2}$ is an abelian symmetric 2-cochain such that $\omega\coloneqq\partial(\sigma) = 1$ and $\sigma(a, a) = Q(a)$ for all $a\in\Gamma$.

Now, to the general case. First, we assume that $Q$ is nondegenerate, i.e.\ that $(\Gamma,Q)$ is a discriminant form. Note that it suffices to prove the assertion for the indecomposable orthogonal components of $(\Gamma, Q)$. Indeed, all of the assertions are preserved under taking orthogonal direct sums. As we will see, the decomposition and \autoref{ex_abelina symmetric cyclic} and \autoref{ex_abelina symmetric center drinfeld} provide us with a concrete way to construct an associated abelian symmetric 2-cochain on~$\Gamma$. Although the orthogonal decomposition of $\Gamma$ is not entirely unique, this does not affect the argument, as we are primarily asserting an existence statement.

Without loss of generality, assume now that $(\Gamma, Q)$ is indecomposable and of order a power of 2; we have treated the odd case above. The indecomposable discriminant forms are either over cyclic groups of order a power of 2 or the families constructed in \autoref{ex_abelina symmetric center drinfeld}. In both cases, the examples provide specific abelian symmetric 2-cochains, thus allowing us, via the orthogonal decomposition, to construct a concrete symmetric abelian 2-cochain associated with the nondegenerate quadratic form $Q$ on $\Gamma$.

Finally, if $(\Gamma, Q)$ is degenerate, we can embed it into a (nondegenerate) discriminant form using the construction of the double:
\begin{equation*}
(\Gamma, Q) \to (\Gamma \oplus \widehat{\Gamma}, \widetilde{Q}), \quad a \mapsto (a, 0),
\end{equation*}
where $\widetilde{Q}(a, \gamma) = Q(a)\gamma(a)$ for $a\in\Gamma$ and $\gamma\in\widehat{\Gamma}$; see Proposition~5.8 in \cite{DN21}. By the previous argument, $(\Gamma \oplus \widehat{\Gamma}, \widetilde{Q})$ has an associated abelian symmetric 2-cochain, and by restriction, this defines another one over $(\Gamma, Q)$.
\end{proof}

In \autoref{cor_special_lat_disc}, we will see how abelian $3$-cocycles with the properties described in \autoref{ass_strongdisc} arise naturally in the context of discriminant forms $\Gamma = L^*\!/L$ of even lattices, where the bilinear form $\langle\cdot, \cdot\rangle$ takes values in $2\Z$, i.e.\ integral lattices scaled by $\sqrt{2}$.

\begin{rem}
The lattices in \autoref{cor_special_lat_disc} thus provide an alternative way to prove \autoref{ass_strongdisc} (in the nondegenerate case). While not all discriminant forms can be realised as dual quotients $L^*\!/L$ of even lattices $L = \sqrt{2}K$ with $K$ integral, all indecomposable Jordan components do indeed appear in $L^*\!/L$ for some lattice $L = \sqrt{2}K$. As explained in the above proof, this is sufficient to prove the assertion.
\end{rem}

We fix a further structure on the pair $(\Gamma,Q)$. In addition to \autoref{ass_strongdisc}, where we described a square root of the bimultiplicative form $B_Q$ associated with~$Q$, we now describe a square root of the quadratic form $Q$ itself. Again, we first consider specific examples, before using these to prove the general case.
\begin{ex}\label{ex_square}~
\begin{enumerate}[wide]
\item Consider the discriminant form $(\Z_n,Q_{2n}^0)$, where $n\geq 2$ is even, and the abelian symmetric 2-cochain $\sigma_{2n}^0$ on it from \autoref{ex_abelina symmetric cyclic}. Set
\begin{equation*}
q_{2n}^0(x) \coloneqq \e\left(\frac{x^2}{4n}\right), \quad 0\leq x< n.
\end{equation*}
Then $q_{2n}^0(a)^2 = Q_{2n}(a)$, and moreover, with $\omega\coloneqq\partial(\sigma_{2n}^0)$,
\begin{equation*}
\frac{q_{2n}^0(a+b)}{q_{2n}^0(a)q_{2n}^0(b)} = \sigma_{2n}^0(a,b)\omega(a+b,a,b)\baromega(\bar{\delta},a,b)
\end{equation*}
for all $a,b \in \Z_n$, where $\bar{\delta}=0$ if $4\mid n$ and $\bar{\delta}=1\in\Z_2\cong\Z_n/2\Z_n$ otherwise.

\item Consider the discriminant form $(\Z_{2^k}\oplus\Z_{2^k},Q_k^+)$ for $k\geq1$ and the abelian symmetric 2-cochain $\sigma_k^+$ from \autoref{ex_abelina symmetric center drinfeld}. Define
\begin{equation*}
q_k^+(x_1, x_2) \coloneqq \e\left(\frac{x_1x_2}{2^{k+1}}\right),\quad 0\leq x_1, x_2 < 2^k.
\end{equation*}
Then $q_k^+(a_1,a_2)^2 = Q_k^+(a_1,a_2)$, and moreover, with $\omega\coloneqq\partial(\sigma_k^+)$,
\begin{equation*}
\frac{q_k^+(a_1+b_1, a_2+b_2)}{q_k^+(a_1, a_2)q_k^+(b_1, b_2)} = \sigma_{k}^+((a_1, a_2), (b_1, b_2))\omega((a_1+b_1, a_2+b_2), (a_1, a_2), (b_1, b_2))
\end{equation*}
for all $a_i,b_i\in\Z_{2^k}$.

\item Consider the discriminant form $(\Z_{2^k}\oplus\Z_{2^k}, Q_k^-)$ for $k\geq1$ and the abelian symmetric 2-cochain $\sigma_k^-$ from \autoref{ex_abelina symmetric center drinfeld}. Define
\begin{equation*}
q_k^-(x_1, x_2) \coloneqq \e\left(\frac{x_1^2 + x_1x_2 + x_2^2}{2^{k+1}}\right),
\end{equation*}
where $0 \leq x_1, x_2 < 2^k$. It follows that $q_k^-(a_1,a_2)^2 = Q_k^-(a_1,a_2)$, and moreover
\begin{align*}
\frac{q_k^-(a_1+b_1, a_2+b_2)}{q_k^-(a_1, a_2)q_k^-(b_1, b_2)} &= \sigma_{k}^-((a_1, a_2), (b_1, b_2))\omega((a_1+b_1, a_2+b_2), (a_1, a_2), (b_1, b_2))
\end{align*}
for all $a_i,b_i\in\Z_{2^k}$, with $\omega\coloneqq\partial(\sigma_k^-)$.
\end{enumerate}
\end{ex}

\begin{prop}\label{prop_ass_q}
Given a representing abelian $3$-cocycle $(\sigma,\omega)=(\sigma,\partial(\sigma))$ for $(\Gamma,Q)$ as in \autoref{ass_strongdisc}, there is a class $\bar{\delta}\in \Gamma/2\Gamma$ and a function $q\colon \Gamma \to \C^\times$ such that:
\begin{enumerate}
\item $q(a)^2 = \sigma(a,a) = Q(a)$,
\item $q(a+b)q(a)^{-1}q(b)^{-1} = \sigma(a,b) \omega(a+b,a,b) \baromega(\bar{\delta},a,b)$
\end{enumerate}
for all $a,b \in \Gamma$.
\end{prop}
\begin{proof}
The proof is similar to the proof of \autoref{ass_strongdisc}. Again, it suffices to prove the assertion for all indecomposable discriminant forms $(\Gamma,Q)$.

For odd $|\Gamma|$, where $\sigma$ had to be the unique bimultiplicative square root of $B_Q$ and $\omega$ trivial, the unique solution is $\bar{\delta}=0$ since $\Gamma/2\Gamma=\{0\}$ and $q$ the unique quadratic square root of $Q$.

For the indecomposable $2$-adic Jordan components, all verified in \autoref{ex_square}, a choice of $q$ is possible with $\bar{\delta}=0$, except for the Jordan components $2_t^{\pm 1}$, where we need to choose the nontrivial class $\bar{\delta}=1\in\Z_2\cong\Gamma/2\Gamma$. Hence, for any decomposition of $(\Gamma, Q)$, the class $\bar{\delta}$ is determined in $\Gamma/2\Gamma$.
\end{proof}

In other words, in the case of a discriminant form $(\Gamma,Q)$ and fixing a decomposition into indecomposable Jordan blocks, $\bar{\delta}\in\Gamma/2\Gamma$ equals
\begin{equation*}
\bar{\delta}=\sum_{D}
\begin{cases}
1\in\Z_2\cong D/2D&\text{if }D\cong 2_t^{\pm1},\\
0\in D/2D&\text{otherwise},
\end{cases}
\end{equation*}
where the sum runs over the indecomposable Jordan components $D$ of $\Gamma$.

Since $\omega$ is a homomorphism into $\{\pm1\}$ in the first argument, only the class in $\Gamma/2\Gamma$ is relevant. Moreover, nontrivial contributions to $\bar{\delta}$ can only come from orthogonal direct summands of $\Gamma$ of order $2$.

In view of \autoref{cor_special_lat_disc} we remark that the Jordan components $2_t^{\pm1}$ only appear in the context of even lattices $L=\sqrt{2}K$ where $K$ is integral but not even. On the other hand, all other indecomposable Jordan components already appear for lattices $L=\sqrt{2}K$ where $K$ is even. We observe that the somewhat natural choice of $\bar{\delta}$ given in \autoref{defi_latqqbar} for $\Gamma=L^*\!/L$ agrees with the above formula for $\bar{\delta}$, as it must.
\begin{rem}
\autoref{prop_ass_q} can also be proved using the concrete lattice realisation in \autoref{cor_special_lat_disc} because all all indecomposable Jordan components appear in $L^*\!/L$ for some lattice $L=\sqrt{2}K$ with $K$ integral.

For instance, if $K$ is even, then it is apparent from the construction there that $\bar{\delta}=0$. On the other hand, for $K=\Z$ and $L=\sqrt{2}\Z$ so that $L^*\!/L\cong 2_1^{+1}$, the construction there yields the nontrivial class $\bar{\delta}=1\in\Z_2$, in agreement with the above claim.
\end{rem}


\subsection{Braided \texorpdfstring{$\Z_2$}{Z\_2}-Crossed Tensor Category}\label{sec_evenTY}

In the following, we explicitly define a braided $G$-crossed extension $\LM(\Gamma,\sigma,\omega,\bar{\delta},\eps\,|\,q,\alpha,\beta)$ with $G=\langle g\rangle\cong\Z_2$ of the braided tensor category $\smash{\Vect_\Gamma^Q}$ with an action of $\Z_2$ by $g_*\C_a=\C_{-a}$ and a certain tensor structure $\tau^g$.

\medskip

We assume in the following that the representing abelian $3$-cocycle $(\sigma,\omega)$ for the discriminant form $(\Gamma,Q)$ has the form given in \autoref{ass_strongdisc}. We further fix a sign $\eps=\pm 1$. Moreover, let $\bar{\delta}\in\Gamma/2\Gamma$ be (for now) arbitrary.

Given these data, consider the $\Z_2$-graded abelian category
\begin{equation*}
\cC=\cC_1\oplus \cC_g=\Vect_{\Gamma} \oplus \Vect_{\Gamma/2\Gamma},
\end{equation*}
where the simple objects are denoted by $\C_a$ for $a\in \Gamma$ and by $\X^{\bar{x}}$ for $\bar{x}\in \Gamma/2\Gamma$. In general, we also denote the coset of an element $x\in\Gamma$ by $\bar{x}\in\Gamma/2\Gamma$. Recall that $\omega$ only depends on $\Gamma/2\Gamma$ in the first argument.

Then we shall prove below (see \autoref{sec_app}) that the following endows $\cC$ with the structure of a $\Z_2$-graded tensor category $\LM(\Gamma,\sigma,\omega,\bar{\delta},\eps)$:
{\allowdisplaybreaks
\begin{align*}
\C_a\otimes \C_b
&=\C_{a+b},\\
\C_a\otimes \X^{\bar{x}}
&= \X^{\bar{a}+\bar{x}},\\
\X^{\bar{x}}\otimes \C_a
&= \X^{\bar{x}+\bar{a}},\\
\X^{\bar{x}} \otimes \X^{\bar{y}}
&=\bigoplus_{\substack{t \in \Gamma \\ \bar{t}=\bar{\delta}+\bar{x}+\bar{y}}}\C_t=\bigoplus_{t \in \bar{\delta}+\bar{x}+\bar{y}}\C_t
\end{align*}
}%
for $a,b\in\Gamma$ and $\bar{x},\bar{y}\in\Gamma/2\Gamma$, where we take the associator to be
{\allowdisplaybreaks
\begin{align*}
(\C_a\otimes \C_b)\otimes \C_c
&\overset{\omega(a,b,c)}{\xrightarrow{\hspace*{2cm}}}
\C_a\otimes (\C_b\otimes \C_c),\\
(\X^{\bar{x}}\otimes \C_a)\otimes \C_b
&\overset{\baromega(\bar{x}+\bar{\delta},a,b)}{\xrightarrow{\hspace*{2cm}}}
\X^{\bar{x}}\otimes (\C_a\otimes \C_b),\\
(\C_a\otimes \X^{\bar{x}})\otimes \C_b
&\overset{\sigma(a,b)}{\xrightarrow{\hspace*{2cm}}}
\C_a\otimes (\X^{\bar{x}}\otimes \C_b),\\
(\C_a\otimes \C_b)\otimes \X^{\bar{x}}
&\overset{\baromega(\bar{a}+\bar{b}+\bar{x},a,b)}{\xrightarrow{\hspace*{2cm}}}
\C_a\otimes (\C_b\otimes \X^{\bar{x}}),\\
e_{a+t}\in
(\C_a\otimes \X^{\bar{x}})\otimes \X^{\bar{y}}
&\overset{\baromega(\bar{a}+\bar{x},a,t)}{\xrightarrow{\hspace*{2cm}}}
\C_a\otimes (\X^{\bar{x}}\otimes \X^{\bar{y}})
\ni e_a\otimes e_t,\\
e_t\otimes e_a\in
(\X^{\bar{x}}\otimes \X^{\bar{y}})\otimes \C_a
&\overset{\baromega(\bar{x},t,a)}{\xrightarrow{\hspace*{2cm}}}
\X^{\bar{x}}\otimes (\X^{\bar{y}}\otimes \C_a)
\ni e_{{a}+{t}},\\
e_t \in
(\X^{\bar{x}}\otimes \C_a)\otimes \X^{\bar{y}}
&\overset{\sigma(a,t)}{\xrightarrow{\hspace*{2cm}}}
\X^{\bar{x}}\otimes (\C_a\otimes \X^{\bar{y}})
\ni e_t,\\
v_t \in
(\X^{\bar{x}}\otimes \X^{\bar{y}})\otimes \X^{\bar{z}}
&\overset{\eps |2\Gamma|^{-1/2}\sigma(t,r)^{-1}}{\xrightarrow{\hspace*{2cm}}}
\X^{\bar{x}}\otimes (\X^{\bar{y}}\otimes \X^{\bar{z}})
\ni v_r
\end{align*}
}%
with the normalisation factor $|2\Gamma|$ for $a,b,c\in\Gamma$ and $\bar{x},\bar{y},\bar{z}\in\Gamma/2\Gamma$. Here, $e_t\in\C_t$ denotes a vector in the tensor product $\X^{\bar{x}}\otimes\X^{\bar{y}}=\bigoplus_{t\in\bar{\delta}+\bar{x}+\bar{y}}\C_t$ and $v_t\in\X^{\bar{t}+\bar{z}}$ in the tensor product $\smash{(\X^{\bar{x}}\otimes \X^{\bar{y}})\otimes \X^{\bar{z}}=\bigoplus_{t\in\bar{\delta}+\bar{x}+\bar{y}}\X^{\bar{t}+\bar{z}}}$.
The matrix $(\sigma(t,r)^{-1})_{t,r}$ is invertible by \autoref{lm_characterSum} (cf.\ \autoref{lem_natural isomorhism}).

\medskip

Similarly to \autoref{sec_TY}, there is a rigid structure with dual objects $\C_a^*=\C_{-a}$ and $(\X^{\bar{x}})^*=\X^{-\bar{x}-\bar{\delta}}$ and with $\coeval_\X=\iota_{\C_0}$ and $\eval_\X=\eps|2\Gamma|^{1/2}\pi_{\C_0}$. The two zigzag identities relating $\eval_\X$ and $\coeval_\X$ hold because the associator and also the inverse associator on $\X\otimes \X\otimes\X$ at $t=0$ are $\eps|2\Gamma|^{-1/2}$, by the slightly unusual character sum in \autoref{lm_characterSum} below.

\medskip

We shall then show that this tensor category $\cC=\LM(\Gamma,\sigma,\omega,\bar{\delta},\eps)$ admits an action of $G=\langle g\rangle\cong\Z_2$ defined by
\begin{equation*}
g_*\C_a=\C_{-a},\quad g_*\X^{\bar{x}}=\X^{\bar{x}}
\end{equation*}
for $a\in\Gamma$, $\bar{x}\in \Gamma/2\Gamma$, with strict composition and with tensor structures $\tau$ given by
{\allowdisplaybreaks
\begin{equation}\label{eq_LM_tau}
\begin{aligned}
g_*(\C_a\otimes \C_b)
&\overset{\omega(a,b,-b)}{\xrightarrow{\hspace*{2cm}}}
g_*\C_a\otimes g_*\C_b,\\
g_*(\C_a\otimes \X^{\bar{x}})
&\overset{\baromega(\bar{a}+\bar{x},a,-a)}{\xrightarrow{\hspace*{2cm}}}
g_*\C_a\otimes g_*\X^{\bar{x}},\\
g_*(\X^{\bar{x}} \otimes \C_a)
&\overset{\baromega(\bar{x}+\bar{\delta},a,-a)}{\xrightarrow{\hspace*{2cm}}}
g_*\X^{\bar{x}} \otimes g_*\C_a,\\
g_*(\X^{\bar{x}}\otimes \X^{\bar{y}})
&\overset{\baromega(\bar{x},t,-t)}{\xrightarrow{\hspace*{2cm}}}
g_*\X^{\bar{x}}\otimes g_*\X^{\bar{y}}
\end{aligned}
\end{equation}
}%
for $a,b\in\Gamma$ and $\bar{x},\bar{y}\in\Gamma/2\Gamma$ and again with $t$ the summation index appearing in the tensor product of $\X\X$-type.

\medskip

Now, suppose in addition that $\bar{\delta}\in\Gamma/2\Gamma$ and $q\colon\Gamma\to\C^\times$ are given as in \autoref{prop_ass_q}.

We shall establish that for every choice of $q$, the $\Z_2$-graded tensor category $\cC$ above becomes a braided $\Z_2$-crossed tensor category $\LM(\Gamma,\sigma,\omega,\bar{\delta},\eps\,|\,q,\alpha)$ with the following braiding:
{\allowdisplaybreaks
\begin{equation}\label{eq_LM_braiding}
\begin{aligned}
\C_a \otimes \C_b
&\overset{\sigma(a,b)}{\xrightarrow{\hspace*{3cm}}}
\C_b \otimes \C_a,\\
\C_a \otimes \X^{\bar{x}}
&\overset{q(a)^{-1}}{\xrightarrow{\hspace*{3cm}}}
\X^{\bar{x}} \otimes \C_a,\\
\X^{\bar{x}} \otimes \C_a
&\overset{q(a)^{-1}\baromega(\bar{x}+\bar{a},a,-a)}{\xrightarrow{\hspace*{3cm}}}
\C_{-a} \otimes \X^{\bar{x}},\\
\X^{\bar{x}}\otimes \X^{\bar{y}}
&\overset{\alpha\,q(t)}{\xrightarrow{\hspace*{3cm}}}
\X^{\bar{y}} \otimes \X^{\bar{x}}
\end{aligned}
\end{equation}
}%
for $a,b\in\Gamma$ and $\bar{x},\bar{y}\in\Gamma/2\Gamma$ with a normalisation factor $\alpha$ given by a choice of square root of
\begin{equation*}
\alpha^2=\eps\,G_{\bar{\delta}}(\Gamma,q^{-1})
\end{equation*}
with the partial Gauss sum
\begin{equation*}
G_{\bar{\delta}}(\Gamma,q^{-1})\coloneqq |2\Gamma|^{-1/2}\sum_{{a\in \bar{\delta}}} q(a)^{-1}.
\end{equation*}
Here, $t$ is again the summation index in the tensor product $\X^{\bar{x}}\otimes\X^{\bar{y}}$ and $\X^{\bar{y}}\otimes\X^{\bar{x}}$.

We study the appearing partial Gauss sum $G_{\bar{\delta}}(\Gamma,q^{-1})$ in more detail. Recall that the sign $s=\pm$ (not to be confused with the signature) is well-defined for a discriminant form as long as the latter does not contain the Jordan component~$2_t^{\pm1}$. Indeed, for the Jordan components $q^{\pm n}$ ($q$ a power of an odd prime) and $q_t^{\pm n}$ and $q_{\II}^{\pm n}$ ($q$ a power of an even prime) the sign $s$ is nothing but the sign in the exponent, and then extends to the whole discriminant form multiplicatively, but for $q=2$ there are the exceptional isomorphisms $2_1^{+1}\cong 2_5^{-1}$ and $2_3^{-1}\cong 2_7^{+1}$, meaning that the sign is in general not well-defined (see also \autoref{tab_disc}). Also recall that $(\frac{\cdot}{\cdot})\in\{\pm1\}$ denotes the Kronecker symbol.
\begin{prop}\label{conj_partialGauss1}
Given a discriminant form $(\Gamma,Q)$ and a choice of representing abelian $3$-cocycle $(\sigma,\omega)$ as in \autoref{ass_strongdisc}, as well as a choice $(q,\bar{\delta})$ as in \autoref{prop_ass_q}, the above partial Gauss sum $G_{\bar{\delta}}(\Gamma,q^{-1})$ is well-defined (i.e.\ independent of these choices) if and only if $\Gamma$ does not contain a Jordan component $2_t^{\pm1}$ in its orthogonal decomposition, in which case $\bar{\delta}=0\in\Gamma/2\Gamma$ and the sum evaluates to
\begin{equation*}
G_0(\Gamma,q^{-1})=|2\Gamma|^{-1/2}\sum_{{a\in 2\Gamma}} q(a)^{-1}=\e(-\sign(\Gamma)/8)\,s(\Gamma_{\mathrm{ev}})\left(\frac{2}{|\Gamma_\mathrm{odd}|}\right),
\end{equation*}
where $\Gamma=\Gamma_\mathrm{ev}\oplus\Gamma_\mathrm{odd}$ is the decomposition of $\Gamma$ into the $2$-adic part and the $p$-adic part for all odd primes $p$.
\end{prop}
\autoref{conj_partialGauss1} generalises the Gauss sum $G(\Gamma,q^{-1})$ in \autoref{sec_TY}, see equation~\eqref{eq_gauss} and \autoref{rem_Gauss}.
\begin{proof}
A proof is obtained by proving the assertion for all indecomposable Jordan components. For the $p$-adic components for odd primes $p$, the assertion is simply \autoref{rem_Gauss}. It is not difficult to compute the Gauss sum for the indecomposable $2$-adic Jordan components except for $2_t^{\pm1}$ and see that the assertion holds.
\end{proof}

For $2_1^{+1}\cong 2_5^{-1}$ with signature $1\pmod{8}$, the above partial Gauss sum evaluates to $\e(7/8)$ or $\e(3/8)$, depending on the choice of $q$. These are exactly the values one would obtain from the formula in \autoref{conj_partialGauss1} by inserting either $1$ or $-1$ for the sign $s$. Similarly, for $2_3^{-1}\cong 2_7^{+1}$ with signature $7\pmod{8}$, the partial Gauss sum yields $\e(5/8)$ or $\e(1/8)$, again compatible with \autoref{conj_partialGauss1}.

Finally, in view of \autoref{conj_partialGauss2}, we mention that if $\Gamma=L^*\!/L$ comes from an even lattice $L=\sqrt{2}K$ where $K$ is itself even, then the sign factors in \autoref{conj_partialGauss1} cancel and the partial Gauss sum evaluates exactly to $\e(-\sign(\Gamma)/8)$. Moreover, if $L=\sqrt{2}K$ where $K$ is only integral (in which case Jordan components $\smash{2_t^{\pm n}}$ appear so that the partial Gauss sum is not well-defined), then the choice of $q$ given in \autoref{defi_latqqbar} fixes the value of the Gauss sum, which still evaluates to $\e(-\sign(\Gamma)/8)$.

\begin{rem}\label{rem_gauss}
Alternatively, in order to prove \autoref{conj_partialGauss1}, we note once again that all indecomposable Jordan components appear in discriminant forms $\Gamma=L^*\!/L$ for even lattices $L=\sqrt{2}K$ where $K$ is integral. In fact, since we are excluding $2_t^{\pm1}$, it suffices to consider $L=\sqrt{2}K$ for $K$ even. Then the Gauss sum $G_0(\Gamma,q^{-1})$ takes the simple form in \autoref{conj_partialGauss2}, allowing us to infer the value of the Gauss sum for all indecomposable Jordan components.

For example, suppose that $q=2^k$ for $k\geq2$. Then the even lattice $\sqrt{q/2}A_1$ has the discriminant form $\Gamma=q_1^{+1}$, whose signature is $1\pmod{8}$. By \autoref{conj_partialGauss2}, the Gauss sum equals $G_0(\Gamma,q^{-1})=\e(7/8)$ (cf.\ \autoref{ex_gauss}).
\end{rem}

\medskip

Finally, we will see that the following $\Z_2$-ribbon twist turns $\cC$ into a $\Z_2$-crossed ribbon fusion category $\LM(\Gamma,\sigma,\omega,\bar{\delta},\eps\,|\,q,\alpha,\beta)$: for $a\in\Gamma$ and $\bar{x}\in\Gamma/2\Gamma$
\begin{equation}\label{eq_Z2-ribbon}
\begin{aligned}
\C_a
&\overset{\sigma(a, a)}{\xrightarrow{\hspace*{2cm}}}
\C_a, \\
\X^{\bar{x}}
&\overset{\beta}{\xrightarrow{\hspace*{2cm}}}
g_*\X^{\bar{x}}
\end{aligned}
\end{equation}
for a choice $\beta=\pm\alpha^{-1}$. We point out that the ribbon twist takes the same value on all the objects $\X^{\bar{x}}$ for $\bar{x}\in\Gamma/2\Gamma$ (cf.\ \autoref{rem_minusoneaction}).

\medskip

Summarising the above, we state the main result of this section (cf.\ \autoref{thm_TYribbon}, \autoref{cor_TYpseudounitary} and \autoref{cor_TY}). The proof is given in \autoref{sec_app}.
\begin{thm}\label{thm_evenTY}
Let $G=\langle g\rangle\cong\Z_2$. The data given above define a braided $\Z_2$-crossed tensor category
\begin{equation*}
\LM(\Gamma,\sigma,\omega,\bar{\delta},\eps\,|\,q,\alpha,\beta)=\Vect_\Gamma\oplus\Vect_{\Gamma/2\Gamma},
\end{equation*}
which is a braided $\Z_2$-crossed extension of the modular tensor category $\smash{\Vect_\Gamma^Q}$ for a discriminant form $(\Gamma,Q)$ with representing abelian $3$-cocycle satisfying \autoref{ass_strongdisc}, with the categorical $\Z_2$-action being $g_*\C_a=\C_{-a}$ with tensor structure $g_*(\C_a\otimes\C_b)\to g_*\C_a\otimes g_*\C_b$ given by $\omega(a,b,-b)$ for $a,b\in\Gamma$. The ribbon twist defined above yields positive quantum dimensions if and only if $\alpha\beta=\eps$.
\end{thm}
In particular, $\LM(\Gamma,\sigma,\omega,\bar{\delta},\eps\,|\,q,\alpha,\beta)$ is pseudo-unitary; cf.\ \autoref{cor_pseudo}.

Note that if $\Gamma$ is odd and we choose $\sigma=B^{1/2}$, $\omega=1$ and $q=Q^{1/2}$, then the definition exactly reduces to the Tambara-Yamagami category in \autoref{sec_TY}.

\begin{prob}\label{prob_ENO}
By the results in \cite{ENO10,DN21}, for each $\eps$, there is a unique braided $\Z_2$-crossed extension $\smash{\Vect_\Gamma^Q[\Z_2,\eps]}$ of $\smash{\Vect_\Gamma^Q}$ (say with a ribbon structure that has positive quantum dimensions) with the given $\Z_2$-action (which depended on $\omega$), i.e.\ independent of the choice of $(q,\bar{\delta})$. Can we verify this directly from the construction of $\LM(\Gamma,\sigma,\omega,\bar{\delta},\eps\,|\,q,\alpha,\eps/\alpha)$? Cf.\ \autoref{cor_TY}. The independence of the sign of $\alpha$ follows as in \autoref{rem_alphaEquivalence}.
\end{prob}

\begin{rem}\label{rem_minusoneaction}
We point out that the braided $\Z_2$-crossed extension in \autoref{thm_evenTY} corresponds to a certain categorical action of $G=\langle g\rangle\cong\Z_2$ on $\smash{\Vect_\Gamma^Q}$, which in particular permutes the simple objects as $g_*\C_a=\C_{-a}$ for all $a\in\Gamma$.

However, not all $\Z_2$-actions with the latter property necessarily give the same extension. Indeed, if we consider the special case of discriminant forms $\Gamma$ with $2\Gamma=\{0\}$, then any $\Z_2$-action fixes the objects of $\smash{\Vect_\Gamma^Q}$, which are indexed by $\Gamma$, but our action is \emph{not} the trivial action, as can be seen by looking at the tensor structure (see also \autoref*{GLM1exm_counterexampleA1} in \cite{GLM24a}).

Our $\Z_2$-action for $2\Gamma=\{0\}$ and the truly trivial $\Z_2$-action both produce braided $\Z_2$-crossed extensions of the form $\cC_1\oplus\C_g=\smash{\Vect_\Gamma\oplus\Vect_\Gamma}$ as abelian categories, but they differ in other aspects. For instance, for our $\Z_2$-action the ribbon twists are the same for all objects in $\cC_g=\smash{\Vect_\Gamma}$, while this will not typically be the case for the trivial $\Z_2$-action.
\end{rem}


\subsection{Equivariantisation}\label{sec_equiv}

In this section, we consider the equivariantisation
\begin{equation*}
\LM(\Gamma,\sigma,\omega,\bar{\delta},\eps \,|\, q,\alpha,\beta) \sslash \Z_2
\end{equation*}
of the braided $\Z_2$-crossed tensor category from \autoref{thm_evenTY}, which is a (pseudo-unitary) modular tensor category. In particular, we determine the simple objects, the fusion rules and the modular data, i.e.\ the $\SMatrix$- and $\TMatrix$-matrix.

\medskip

First, similarly to \autoref{prop_TY_modular}, we show:
\begin{prop}\label{prop_equimod}
\!The equivariantisation $\LM(\Gamma,\sigma,\omega,\bar{\delta},\eps \,|\, q,\alpha,\beta)\sslash\Z_2$ has a nondegenerate braiding, and hence is a modular tensor category.
\end{prop}
\begin{proof}
As proved in Proposition~4.56 of \cite{DGNO10}, the equivariantisation $\cC\sslash G$ of a braided $G$-crossed tensor category $\cC$ is nondegenerate if and only if the grading is faithful and $\cC_1$ is nondegenerate. Since $\LM(\Gamma,\sigma,\omega,\bar{\delta},\eps \,|\, q,\alpha,\beta)$ is faithfully graded and the trivial sector is nondegenerate by construction, its $\Z_2$-equivariantisation is hence nondegenerate. Furthermore, since it has a $\Z_2$-ribbon structure, the equivariantisation is modular.
\end{proof}

We describe the simple objects of $\LM(\Gamma, \sigma, \omega, \bar{\delta}, \eps \,|\, q, \alpha, \beta)\sslash\Z_2$. Consider the exact sequence $\{0\} \to \Gamma_2 \to \Gamma \smash{\overset{2}{\to}} 2\Gamma \to \{0\}$, where $\Gamma_2 = \{a \in \Gamma \mid 2a = 0\}$. In particular, $2\Gamma \cong \Gamma / \Gamma_2$ and $|2\Gamma||\Gamma_2| = |\Gamma|$.

\begin{prop}\label{prop_eqivsimple}
The simple objects of $\LM(\Gamma, \sigma, \omega, \bar{\delta}, \eps \,|\, q, \alpha, \beta)\sslash\Z_2$, up to isomorphism, are given by:
\begin{enumerate}
\item $2|\Gamma_2|$ invertible objects (of quantum dimension~$1$) indexed by $\Gamma_2 \times \{\pm 1\}$. They are given by $X_{a,s} \coloneqq \C_a$ with equivariant structure $\eqi_g = s \id_{\C_a}$ for $a\in\Gamma_2$ and $s\in\{\pm 1\}$.

\item $\frac{1}{2}(|\Gamma| - |\Gamma_2|) = \frac{1}{2}|\Gamma|(1 - 1/|2\Gamma|)$ simple objects of quantum dimension~$2$ indexed by the set of unordered pairs $\{\{a,-a\} \mid a \in \Gamma \setminus \Gamma_2\}$. They are given by $Y_{\{a, -a\}} \coloneqq \C_a \oplus \C_{-a}$ for $a$ in $\Gamma \setminus \Gamma_2$, with the equivariant structure given by interchanging the direct summands.

\item $2|\Gamma_2|$ simple objects of quantum dimension $(|\Gamma|/|\Gamma_2|)^{1/2}=|2\Gamma|^{1/2}$ indexed by $\Gamma/2\Gamma \times \{\pm 1\}$. They are given by $Z_{\bar{x}, s}\coloneqq\X^{\bar{x}}$ with equivariant structure $\eqi_g=s \id_{\X^{\bar{x}}}$ for $\bar{x}\in\Gamma/2\Gamma$ and $s\in\{\pm 1\}$.
\end{enumerate}
\end{prop}
\begin{proof}
The assertion follows directly from an application of \autoref{ex_classification equivariant simple order 2}.
\end{proof}

Next, we determine the fusion rules among the simple objects of the equivariantisation. We know that $\sigma(a, -b)\sigma(a, b) = \omega(a, b, -b)$ for $a,b\in\Gamma$. In particular, for $b \in \Gamma_2$, $\omega(a, b, -b) = \sigma(a,b)^2 = B_Q(a,b)$ for $a\in\Gamma$. Moreover, it defines a group homomorphism $\omega(\cdot,a, -a)\colon \Gamma/2\Gamma \to \{\pm 1\}$ for each $a\in\Gamma$, or $B_Q(\cdot,a)\colon \Gamma/2\Gamma \to \{\pm 1\}$ for $a\in\Gamma_2$.

\begin{prop}\label{prop_equi_fusion}
The fusion rules of $\LM(\Gamma, \sigma, \omega, \bar{\delta}, \eps \,|\, q, \alpha, \beta)\sslash\Z_2$ are:
{\allowdisplaybreaks
\begin{align*}
X_{a,s} \otimes X_{b,t} &= X_{a+b, st \omega(a,b,-b)}=X_{a+b,stB_Q(a,b)},\\
X_{a,s} \otimes Y_{\{b,-b\}} &= Y_{\{a+b,a-b\}},\\
X_{a,s} \otimes Z_{\bar{x}, r} &= Z_{\bar{a} + \bar{x}, sr \baromega(\bar{a}+\bar{x},a,-a)}=Z_{\bar{a} + \bar{x}, sr B_Q(\bar{a}+\bar{x},a)}, \\
Y_{\{a,-a\}} \otimes Z_{\bar{x}, r} &= Z_{\bar{x}+\bar{a}, 1} \oplus Z_{\bar{x}+\bar{a}, -1},\\
Y_{\{b,-b\}} \!\otimes\! Y_{\{c,-c\}} &=
\begin{cases}
X_{b+c,1} \!\oplus\! X_{b+c,-1} \!\oplus\! X_{b-c,1} \!\oplus\! X_{b-c,-1} & \text{\!\!\!if } b+c, b-c \in \Gamma_2, \\
Y_{\{b+c,-b-c\}} \!\oplus\! X_{b-c,1} \!\oplus\! X_{b-c,-1} & \text{\!\!\!if } b+c \notin \Gamma_2, \, b-c \in \Gamma_2, \\
Y_{\{b-c,-b+c\}} \!\oplus\! X_{b+c,1} \!\oplus\! X_{b+c,-1} & \text{\!\!\!if } b+c \in \Gamma_2, \, b-c \notin \Gamma_2, \\
Y_{\{b+c,-b-c\}} \!\oplus\! Y_{\{b-c,-b+c\}} & \text{\!\!\!if } b+c, b-c \notin \Gamma_2,
\end{cases}\\
Z_{\bar{x}, r} \otimes Z_{\bar{y}, s} &= \bigoplus_{\substack{t \in \Gamma_2 \\ \bar{t}=\bar{\delta}+\bar{x}+\bar{y}}} X_{t,rs\baromega(\bar{x},t,-t)} \oplus \bigoplus_{\substack{t \in (\Gamma \setminus \Gamma_2)/2 \\ \bar{t}=\bar{\delta}+\bar{x}+\bar{y}}} Y_{\{t,-t\}}\\
&= \bigoplus_{\substack{t \in \Gamma_2 \\ \bar{t}=\bar{\delta}+\bar{x}+\bar{y}}} X_{t,rsB_Q(\bar{x},t)} \oplus \bigoplus_{\substack{t \in (\Gamma \setminus \Gamma_2)/2 \\ \bar{t}=\bar{\delta}+\bar{x}+\bar{y}}} Y_{\{t,-t\}}.
\end{align*}
}%
Here, $(\Gamma \setminus \Gamma_2)/2$ indicates that we select one of $t$ or $-t$ for each pair $\{t,-t\}\subset \Gamma \setminus \Gamma_2$.
\end{prop}

\begin{rem}
We remark that the fusion rules of $\LM(\Gamma, \sigma, \omega, \bar{\delta}, \eps \,|\, q, \alpha, \beta)$ and those of its $\Z_2$-equi\-vari\-an\-ti\-sa\-tion $\LM(\Gamma, \sigma, \omega, \bar{\delta}, \eps \,|\, q, \alpha, \beta)\sslash \Z_2$ do not depend on~$\eps$. This is analogous to Tambara-Yamagami categories. Both $\LM(\Gamma, \sigma, \omega, \bar{\delta}, \eps \,|\, q, \alpha, \beta)$ and $\LM(\Gamma, \sigma, \omega, \bar{\delta}, \eps \,|\, q, \alpha, \beta)\sslash \Z_2$ are $\Z_2$-graded categories. We can twist the associativity constraint using the nontrivial $3$-cocycle $\omega \in Z^3(\Z_2, \C^\times)$ defined by $\omega(1,1,1) = -1$. This twist only alters $\eps$ to $-\eps$, leaving the fusion rules unaffected.
\end{rem}

Finally, we determine the modular data. Recall that, given a ribbon fusion category $\cB$, its \emph{modular data} are defined as the pair of $|\Irr(\cB)|\times|\Irr(\cB)|$-matrices
\begin{equation*}
\SMatrix_{X,Y} = \tr(c_{Y,X} \circ c_{X,Y}),\quad
\TMatrix_{X,Y} = \delta_{X,Y} \theta_X^{-1},
\end{equation*}
where $X, Y \in \Irr(\cB)$.

\begin{prop}\label{prop_equivmodular}
The modular data of $\LM(\Gamma, \sigma, \omega, \bar{\delta}, \eps \,|\, q, \alpha, \beta)\sslash\Z_2$ are given by the $\TMatrix$-matrix with entries
\begin{equation*}
\TMatrix_{X_{a,s},X_{a,s}} = Q(a)^{-1},\quad
\TMatrix_{Y_{\{a,-a\}},Y_{\{a,-a\}}} = Q(a)^{-1},\quad
\TMatrix_{Z_{\bar{x}, s},Z_{\bar{x}, s}} = s\beta^{-1}
\end{equation*}
and the $\SMatrix$-matrix with entries
\begin{align*}
\SMatrix_{X_{a,s},X_{b,t}} &= B_Q(a,b), &
\SMatrix_{X_{a,s},Y_{\{b,-b\}}} &= 2B_Q(a,b), \\
\SMatrix_{X_{a,s},Z_{\bar{x}, r}} &= s|2\Gamma|^{1/2}B_Q(\bar{x},a)Q(a), &
\SMatrix_{Y_{\{a,-a\}}, Z_{\bar{x}, r}} &= 0, \\
\SMatrix_{Y_{\{b,-b\}}, Y_{\{c,-c\}}} &= 2\bigl(B_Q(b,c)+B_Q(b,c)^{-1}\bigr),\\
\SMatrix_{Z_{\bar{x}, r}, Z_{\bar{y}, s}} &= \eps \, rs |2\Gamma|^{1/2}G_{\bar{\delta}}(\Gamma,q^{-1})G_{\bar{\delta}+\bar{x}+\bar{y}}(\Gamma,Q),
\end{align*}
with the Gauss sum $G_{\bar{\delta}+\bar{z}}(\Gamma,Q)=|2\Gamma|^{-1/2}\sum_{a\in \bar{\delta}+\bar{z}}Q(a)$ for $\bar{z}\in\Gamma/2\Gamma$ and the Gauss sum $G_{\bar{\delta}}(\Gamma,q^{-1})$ over $q(a)^{-1}$ from \autoref{conj_partialGauss1}.
\end{prop}
Recall that $\eps\in\{\pm1\}$, that $\beta$ is a square root of $\beta^2=1/\alpha^2=\eps/G_{\bar{\delta}}(\Gamma,q^{-1})$ and that the choice $\beta=\eps/\alpha$ results in positive quantum dimensions. (Note that $\alpha$ itself does not explicitly appear in the modular data.)
\begin{proof}
First, we use \autoref{lem_ribbon} to compute the ribbon element of the equivariantisation. For simple objects graded only over the neutral component, the ribbon is simply the $\Z_2$-twist of one of the simple constituents. For simple objects graded only over the nontrivial element $g$, the ribbon is the $\Z_2$-twist of one of the simple constituents composed with $\eqi_g$. Using the $\Z_2$-ribbons in \eqref{eq_Z2-ribbon}, we obtain
\begin{equation*}
\theta_{X_{a,s}} = Q(a),\quad
\theta_{Y_{\{b,-b\}}} = Q(b),\quad
\theta_{Z_{\bar{x}, s}} = s\beta.
\end{equation*}
The expression for the $\TMatrix$-matrix follows.

Then we compute the $\SMatrix$-matrix using the fusion rules in \autoref{prop_equi_fusion} and the following formula, which holds in any modular tensor category $\cB$:
\begin{equation*}
\SMatrix_{X,Y} = \theta_X^{-1} \theta_Y^{-1} \sum_{Z\in\Irr(\cB)} N_{X,Y}^Z \theta_Z \dim(Z)
\end{equation*}
for all $X,Y\in\Irr(\cB)$, with the fusion rules written as $\smash{X\otimes Y=\bigoplus_{Z\in\Irr(\cB)}N_{X,Y}^ZZ}$ with fusion coefficients $N_{X,Y}^Z\in\N$.
\begin{enumerate}[wide]
\item Since $X_{a,s} \otimes X_{b,t}=X_{a+b,stB_Q(a,b)}$, it follows that
\begin{equation*}
\SMatrix_{X_{a,s},X_{b,t}}=Q(a)^{-1}Q(b)^{-1}Q(a+b)=B_Q(a,b).
\end{equation*}

\item Because $X_{a,s} \otimes Y_{\{b,-b\}} = Y_{\{a+b,a-b\}}$, we obtain
\begin{equation*}
\SMatrix_{X_{a,s},Y_{\{b,-b\}} }=2Q(a)^{-1}Q(b)^{-1}Q(a+b)=2B_Q(a,b).
\end{equation*}

\item Since $X_{a,s} \otimes Z_{\bar{x}, r} =Z_{\bar{a} + \bar{x}, sr B_Q(\bar{a}+\bar{x},a)}$,
\begin{equation*}
\SMatrix_{X_{a,s},Z_{\bar{x}, r}}=|2\Gamma|^{1/2}Q(a)^{-1}r\beta^{-1}rs\beta B_Q(\bar{a}+\bar{x},a)=s|2\Gamma|^{1/2}B_Q(\bar{x},a)Q(a).
\end{equation*}

\item The fusion rule $Y_{\{a,-a\}} \otimes Z_{\bar{x}, r} = Z_{\bar{x}+\bar{a}, 1} \oplus Z_{\bar{x}+\bar{a}, -1}$ implies
\begin{equation*}
\SMatrix_{Y_{\{a,-a\}},Z_{\bar{x}, r}}=Q(a)^{-1}r\beta^{-1}\bigl(|2\Gamma|^{1/2}\beta-|2\Gamma|^{1/2}\beta\bigr)=0.
\end{equation*}

\item From the fusion rule of $Y_{\{b,-b\}} \otimes Y_{\{c,-c\}}$ we obtain
\begin{align*}
\SMatrix_{Y_{\{b,-b\}},Y_{\{c,-c\}}} &= Q(b)^{-1}Q(c)^{-1} \bigl( 2 Q(b+c) + 2 Q(b-c) \bigr) \\
&= 2 \bigl( B_Q(b,c) + B_Q(b,-c) \bigr) \\
&= 2 \bigl( B_Q(b,c) + B_Q(b,c)^{-1} \bigr).
\end{align*}

\item Finally, the fusion rule of $Z_{\bar{x}, r} \otimes Z_{\bar{y}, s}$ implies that
\begin{align*}
\SMatrix_{Z_{\bar{x}, r},Z_{\bar{y}, s}} &= rs\beta^{-2}\!\!\sum_{t \in \bar{\delta}+\bar{x}+\bar{y}}\!\!Q(t)\\
&= rs\alpha^2|2\Gamma|^{1/2}G_{\bar{\delta}+\bar{x}+\bar{y}}(\Gamma,Q) \\
&= \eps \, rs |2\Gamma|^{1/2}G_{\bar{\delta}}(\Gamma,q^{-1})G_{\bar{\delta}+\bar{x}+\bar{y}}(\Gamma,Q).\qedhere
\end{align*}
\end{enumerate}
\end{proof}
We comment on the product $|2\Gamma|^{1/2}G_{\bar{\delta}}(\Gamma,q^{-1})G_{\bar{\delta}+\bar{x}+\bar{y}}(\Gamma,Q)$ of the two partial Gauss sums appearing in the $\SMatrix$-matrix entry $\smash{\SMatrix_{Z_{\bar{x},r},Z_{\bar{y},s}}}$. The first Gauss sum is well-defined only for discriminant forms not containing $2_t^{\pm1}$ and evaluated in \autoref{conj_partialGauss1}. The second Gauss sum is always well-defined but takes exceptional values for $\smash{2_t^{\pm1}}$, $\smash{4_t^{\pm1}}$ and $\smash{2_\II^{\pm1}}$. For ease of presentation, let us assume that the discriminant form $(\Gamma,Q)$ does not contain these Jordan components. Then $\bar{\delta}=0\in\Gamma/2\Gamma$ and
\begin{equation*}
G_{\bar{\delta}+\bar{x}+\bar{y}}(\Gamma,Q)=|\Gamma/2\Gamma|^{1/2}\e(\sign(\Gamma,Q)/8)\delta_{\bar{x}+\bar{y},0}.
\end{equation*}
Hence, the product of the Gauss sums evaluates to
\begin{equation}\label{eq_SZZ}
|2\Gamma|^{1/2}G_{\bar{\delta}}(\Gamma,q^{-1})G_{\bar{\delta}+\bar{x}+\bar{y}}(\Gamma,Q)=|\Gamma|^{1/2}s(\Gamma_{\mathrm{ev}})\left(\frac{2}{|\Gamma_\mathrm{odd}|}\right)\delta_{\bar{x}+\bar{y},0}
\end{equation}
if we exclude the possibility of these exceptional cases. Below, we list also the (exceptional) $\SMatrix$-matrix entry for the discriminant forms $4_t^{\pm1}$.

\medskip

If $\cC$ and $\cD$ are braided $\Z_2$-crossed tensor categories, then the tensor category $\smash{\cC\boxtimes^{\Z_2}\cD} = (\cC_1\boxtimes\cD_1) \oplus (\cC_g\boxtimes\cD_g) \subset \cC\boxtimes\cD$ is a braided $\Z_2$-crossed tensor category with the diagonal $\Z_2$-action and the obvious $\Z_2$-braiding. In particular,
\begin{align*}
&\LM(\Gamma_1,\sigma_1,\omega_1,\bar{\delta}_1,\eps_1 \,|\, q_1,\alpha_1,\beta_1)
\boxtimes^{\Z_2}
\LM(\Gamma_2,\sigma_2,\omega_2,\bar{\delta}_2,\eps_2 \,|\, q_2,\alpha_2,\beta_2) \\
&\quad \cong \LM(\Gamma_1\oplus \Gamma_2,\sigma_1\oplus \sigma_2,\omega_1\oplus \omega_2,\bar{\delta}_1\oplus \bar{\delta}_2,\eps_1\eps_2 \mid q_1\oplus q_2,\alpha_1\alpha_2,\beta_1\beta_2).
\end{align*}
Hence, it suffices to present concrete examples associated with \emph{indecomposable} discriminant forms $(\Gamma,Q)$.


\subsubsection*{Odd Groups}

Let $(\Gamma, Q)$ be a discriminant form (indecomposable or not) where $\Gamma$ is a group of odd order. Then, we already noted that $\LM(\Gamma,\sigma,\omega,\bar{\delta},\eps\,|\,q,\alpha,\beta)$ reduces to the (odd) Tambara-Yamagami category $\TY(\Gamma,\sigma,\eps\,|\,q,\alpha,\beta)$, as considered in \autoref{sec_TY}. In that case, the equivariantisation was already computed in Section~5 of \cite{GNN09}.

Recall that in the odd case, $q$ is the unique quadratic square root of $Q$, $\sigma$ is the unique bimultiplicative square root of $B_Q$ and $\bar{\delta}=0\in\Gamma/2\Gamma$. Moreover, $|\Gamma_2| = 1$ and $2\Gamma = \Gamma$.

It is not difficult to reduce our formulae for the fusion rules and the modular data to the special case of odd $|\Gamma|$, and we find perfect agreement with \cite{GNN09}. Indeed, the simple objects are then classified as follows:
\begin{enumerate}
\item Two invertible objects $X_{s}$ (of quantum dimension~$1$) for $s\in \{\pm 1\}$.
\item $(|\Gamma| - 1)/2$ simple objects $Y_{\{a, -a\}}$ of quantum dimension~$2$ for $a\in\Gamma\setminus\{0\}$.
\item Two simple objects $Z_{s}$ of quantum dimension $|\Gamma|^{1/2}$ for $s\in \{\pm 1\}$.
\end{enumerate}
We forego an explicit description of the fusion rules; they are readily read off from \autoref{prop_equi_fusion}.

In the odd case, the modular data are given by the $\TMatrix$-matrix
\begin{equation*}
\TMatrix_{X_{s},X_{s}} = 1,\quad
\TMatrix_{Y_{\{a,-a\}},Y_{\{a,-a\}}} = Q(a)^{-1},\quad
\TMatrix_{Z_{s},Z_{s}} = s\beta^{-1}
\end{equation*}
and the $\SMatrix$-matrix
{\allowdisplaybreaks
\begin{align*}
\SMatrix_{X_{s},X_{t}} &= 1, &\!\!\!\!\!\!\!\!\!\!\!\!\!\!\!\!\!\!\!\!\!\!\!\!\!\!\!\!\!\!\!\!
\SMatrix_{X_{s},Y_{\{b,-b\}}} &= 2, \\
\SMatrix_{X_{s},Z_{r}} &= s|\Gamma|^{1/2}, &\!\!\!\!\!\!\!\!\!\!\!\!\!\!\!\!\!\!\!\!\!\!\!\!\!\!\!\!\!\!\!\!
\SMatrix_{Y_{\{a,-a\}}, Z_{r}} &= 0, \\
\SMatrix_{Y_{\{b,-b\}}, Y_{\{c,-c\}}} &= 2\bigl(B_Q(b,c) + B_Q(b,c)^{-1}\bigr),\\[-2mm]
\SMatrix_{Z_{s}, Z_{r}} &= \eps \, rs |\Gamma|^{1/2}G(\Gamma,q^{-1})G(\Gamma,Q)= \eps \, rs |\Gamma|^{1/2}\left(\frac{2}{|\Gamma|}\right).
\end{align*}
}%
In the last formula, we used that the product of the two Gauss sums simplifies to a Kronecker symbol; see equation~\eqref{eq_SZZ} or \autoref{rem_Gauss}.

Here, $\eps\in\{\pm1\}$ and $\beta$ is a square root of $\beta^2=1/\alpha^2=\eps\,\bigl(\smash{\frac{2}{|\Gamma|}}\bigr)\e(\sign(\Gamma,Q)/8)$. The choice $\beta=\eps/\alpha$ results in positive quantum dimensions.


\subsubsection*{Even Indecomposable Components}

As another example, for $k\geq2$ we consider the indecomposable discriminant form with Jordan symbol $(2^k)_t^{\pm1}$, i.e.\ the group $\Gamma = \Z_{2^k}$ with the quadratic form $\smash{Q(x)=\e\bigl(\frac{tx^2}{2^{k+1}}\bigr)}$, where $0 \leq x < 2^k$ and $t \in \mathbb{Z}_8^\times$ (see \autoref{tab_disc}). The sign in the Jordan symbol is $s(\Gamma,Q)=\bigl(\frac{t}{2}\bigl)=(-1)^{(t^2-1)/8}=\pm1$, and we write $\smash{(2^k)_t\coloneqq(2^k)_t^{\pm1}}$ for short. The signature is $\smash{\e(\sign(\Gamma,Q)/8)=\e(t/8)\bigl(\frac{t}{2}\bigl)^k}$.

Recall that the choice of $\sigma$ from \autoref{ex_abelina symmetric cyclic}, the choice $q(x) = \e\bigl(\frac{tx^2}{2^{k+2}}\bigr)$ from \autoref{ex_square} and $\bar{\delta}=0\in\Gamma/2\Gamma$ satisfy the conditions in \autoref{prop_ass_q}. To be precise, there we only considered the case $t=1$, while in general we simply raise $\sigma$ and $q$ to the power of $t$.

Moreover, $\Gamma_2 = \langle 2^{k-1} \rangle \cong \Z_2$ (which we write as $\{0,1\}$ and not as $\{0,2^{k-1}\}$ in the following) and $2\Gamma = \langle 2 \rangle \cong \Z_{2^{k-1}}$ so that $\Gamma/2\Gamma\cong\Z_2$ (also represented by $\{0,1\}$). The simple objects are then classified as follows:
\begin{enumerate}
\item Four invertible objects $X_{a,s}$ (of quantum dimension~$1$) for $a \in \{0,1\}$ and $s \in \{\pm 1\}$.
\item $2^{k-1} - 1$ simple objects of quantum dimension~$2$, denoted by $Y_n\coloneqq Y_{\{n,-n\}}$, for $1 \leq n \leq 2^{k-1} - 1$.
\item Four simple objects $Z_{\bar{x},s}$ of quantum dimension $2^{(k-1)/2}$ for $\bar{x} \in \{0,1\}$ and $s \in \{1,-1\}$.
\end{enumerate}

By \autoref{conj_partialGauss1}, the Gauss sum $G_0(\Gamma,q^{-1})$ takes the values
\begin{equation*}
G_0(\Gamma,q^{-1})=\e(-\sign(\Gamma,Q)/8)\Bigl(\frac{t}{2}\Bigr)=\e(-t/8)\Bigl(\frac{t}{2}\Bigr)^k\Bigl(\frac{t}{2}\Bigr).
\end{equation*}
On the other hand, it is not difficult to see that
\begin{equation*}
G_{\bar{x}}(\Gamma,Q)=\sqrt{2}\e(\sign(\Gamma,Q)/8)\cdot
\begin{cases}
\delta_{\bar{x},1}&\text{if }k=2,\\
\delta_{\bar{x},0}&\text{if }k\geq3.
\end{cases}
\end{equation*}
Hence, we obtain for the product of these Gauss sums
\begin{equation*}
G_0(\Gamma,q^{-1})G_{\bar{x}}(\Gamma,Q)=\sqrt{2}\,\Bigl(\frac{t}{2}\Bigl)\cdot
\begin{cases}
\delta_{\bar{x},1}&\text{if }k=2,\\
\delta_{\bar{x},0}&\text{if }k\geq3,
\end{cases}
\end{equation*}
which agrees with \eqref{eq_SZZ} but in addition treats the cases $4_t^{\pm1}$. Then, the modular data for $(2^k)_t$ with $k\geq2$ and $t \in \mathbb{Z}_8^\times$ are given by the $\TMatrix$-matrix with entries
\begin{align*}
\TMatrix_{X_{a,s},X_{a,s}} &= \e(-ta^22^{k-3})=
\begin{cases}
(-1)^a&\text{if }k=2,\\
1&\text{if }k\geq3,\\
\end{cases}\\
\TMatrix_{Y_n,Y_n} &= \e(-tn^2/2^{k+1}),\quad
\TMatrix_{Z_{\bar{x}, s},Z_{\bar{x}, s}} = s\beta^{-1}
\end{align*}
and the $\SMatrix$-matrix with entries
\begin{align*}
\SMatrix_{X_{a,s},X_{b,t}} &=1, &
\SMatrix_{X_{a,s},Y_n} &= 2(-1)^{an}, \\
\SMatrix_{X_{a,s},Z_{\bar{x}, r}} &= s2^{(k-1)/2}
\begin{cases}
(-1)^{a(\bar{x}+1)}&\text{if }k=2,\\
(-1)^{a\bar{x}}&\text{if }k\geq3,\\
\end{cases} &
\SMatrix_{Y_n, Z_{\bar{x}, r}} &= 0, \\
\SMatrix_{Y_m, Y_n} &= 4\cos(tmn\pi/2^{k-1}),\\
\SMatrix_{Z_{\bar{x}, r}, Z_{\bar{y}, s}} &= \eps \, rs 2^{k/2}\Bigl(\frac{t}{2}\Bigr)\cdot
\begin{cases}
\delta_{\bar{x}+\bar{y},1}&\text{if }k=2,\\
\delta_{\bar{x}+\bar{y},0}&\text{if }k\geq3.
\end{cases}
\end{align*}
Here, $\eps\in\{\pm1\}$ and $\beta$ is a square root of $\beta^2=1/\alpha^2=\eps/G_0(\Gamma,q^{-1})$. The choice $\beta=\eps/\alpha$ results in positive quantum dimensions.

\begin{ex}
We take the discriminant form $\smash{\Gamma=(4)_1=(4)_1^{+1}}$ (for $k=2$) with the choice $\eps=1$, some choice of $\alpha$ and $\beta=\eps/\alpha=1/\alpha$. The Gauss sum is $G_0(\Gamma,q^{-1})=\e(7/8)$. For definiteness, take $\alpha=\e(7/16)$, and so $\beta=\eps/\alpha=\e(9/16)$, recalling that $\alpha^2=\eps\,G_0(\Gamma,q^{-1})$. Then the modular data are given by
{\setlength{\arraycolsep}{3pt}
\begin{equation*}
\TMatrix = \operatorname{diag}
\left(\begin{array}{cccc|cccc|c}
X_{0,1} & X_{0,-1} & X_{1,-1} & X_{1,1} & Z_{0,-1} & Z_{1,-1} & Z_{0,1} & Z_{1,1} & Y_1 \\\hline
\vphantom{f^{f^f}}1 & 1 & -1 & -1 &
\e\left(\frac{15}{16}\right) &
\e\left(\frac{15}{16}\right) &
\e\left(\frac{7}{16}\right) &
\e\left(\frac{7}{16}\right) &
\e\left(\frac{7}{8}\right)
\end{array}\right)
\end{equation*}
and
\begin{equation*}
\SMatrix =
\left(\begin{array}{c|cccc|cccc|c}
& X_{0,1} & X_{0,-1} & X_{1,-1} & X_{1,1} & Z_{0,-1} & Z_{1,-1} & Z_{0,1} & Z_{1,1} & Y_1 \\\hline
\vphantom{f^{f^f}}X_{0,1} & 1 & 1 & 1 & 1 & \sqrt{2} & \sqrt{2} & \sqrt{2} & \sqrt{2} & 2 \\
X_{0,-1} & & 1 & 1 & 1 & -\sqrt{2} & -\sqrt{2} & -\sqrt{2} & -\sqrt{2} & 2 \\
X_{1,-1} & & & 1 & 1 & \sqrt{2} & -\sqrt{2} & \sqrt{2} & -\sqrt{2} & -2 \\
X_{1,1} & & & & 1 & -\sqrt{2} & \sqrt{2} & -\sqrt{2} & \sqrt{2} & -2 \\\hline
Z_{0,-1} & & & & & 0 & 2 & 0 & -2 & 0 \\
Z_{1,-1} & & & & & & 0 & -2 & 0 & 0 \\
Z_{0,1} & & & & & & & 0 & 2 & 0 \\
Z_{1,1} & & & & & & & & 0 & 0 \\\hline
Y_1 & & & & & & & & & 0 \\
\end{array}\right).
\end{equation*}
}%
\end{ex}


\subsection{Data for Lattices}\label{sec_latticedata}

We now discuss how the data used to define the braided $\Z_2$-crossed tensor category in \autoref{sec_evenTY} can be obtained from lattice data. If the even lattice is positive-definite, the braided $\Z_2$-crossed tensor category will in this way appear as category of modules of a $\Z_2$-orbifold of the corresponding lattice \voa{} (see \autoref*{GLM1thm_latmain} in \cite{GLM24a}).

\medskip

Any discriminant form $(\Gamma,Q)$ can be realised as dual quotient $\Gamma=L^*\!/L$ of some even lattice $L=(L,\langle\cdot,\cdot\rangle)$ embedded into its ambient quadratic space $(\R^d,\DiscQ)$ with $\R^d=L\otimes_\Z\R$ and $\DiscQ(v)=\e(\langle v,v\rangle/2)$ for $v\in\R^d$ \cite{Nik80}. We then recall from \autoref*{GLM1sec_latticeDiscriminantForm} of \cite{GLM24a} that one can realise $\smash{\Vect_\Gamma^Q=\Vect_\Gamma^{(\sigma,\omega)}}$ concretely as the condensation (i.e.\ going to the local modules)
\begin{equation*}
\Vect_{\R^d}^\DiscQ=\Vect_{\R^d}^{(\DiscS,\DiscO)}\leadsto\Vect_\Gamma^Q=\Vect_\Gamma^{(\sigma,\omega)}
\end{equation*}
of the infinite pointed braided fusion category $\smash{\Vect_{\R^d}^\DiscQ}$ by an infinite commutative, associative algebra $A=\C_\epsilon[L]$ associated with the even lattice $L\subseteq\R^d$.

In particular, in \cite{GLM24a} we describe how the abelian $3$-cocycle $(\DiscS,\DiscO)$ on the ambient quadratic space $\smash{(\R^d,\DiscQ)}$, chosen as, say,
\begin{equation*}
\DiscS(u,v)=\e(\langle u,v\rangle/2),\quad\DiscO(u,v,w)=1
\end{equation*}
for $u,v,w\in\R^d$, induces the abelian $3$-cocycle $(\sigma,\omega)$ on $(\Gamma,Q)$. The latter further depends on a choice of $2$-cocycle $\epsilon\colon L\times L\to\C^\times$ whose skew form equals $\epsilon(u,v)\epsilon(v,u)^{-1}=\DiscS(u,v)$ for all $u,v\in L$, and further a choice of representatives $\hat{a}\in L^*$ for the $L$-cosets $a\in\Gamma$ and the corresponding $2$-cocycle $u\colon\Gamma\times\Gamma\to L$.

\medskip

In this section, we assume for simplicity the following \emph{strong evenness} condition for the lattice $L$:
\begin{ass}\label{ass_strongeven}
The lattice $L$ fulfils one of the following equivalent conditions:
\begin{enumerate}
\item\label{eq_conditionEpsilonOne} $\langle u,v\rangle\in 2\Z$ for all $u,v\in L$,
\item $L=\sqrt{2}K$ where $K$ is an integral lattice,
\item $L\subset 2L^*$,
\item $\DiscS(u,v)=\e(\langle u,v\rangle/2)=1$ for all $u,v\in L$,
\item The $2$-cocycle $\epsilon$ on $L$ whose skew form is $\DiscS|_L$ can be chosen to be trivial,
\item\label{eq_conditionEpsilonOne_6} $\Gamma/2\Gamma\cong\Z_2^{\rk(L)}$.
\end{enumerate}
\end{ass}
\begin{proof}[Proof of Equivalence]
All equivalences except the one to item~\eqref{eq_conditionEpsilonOne_6} are clear. For an integral lattice of rank $d=\rk(L)$, there is a basis $\{\alpha_1,\dots,\alpha_d\}$ of $L^*$ and integers $m_1,\dots,m_d\in\Z$ such that $\{m_1\alpha_1,\dots,m_r\alpha_d\}$ is a basis of $L$. Then, as abelian group $\Gamma=L^*\!/L\cong\Z_{m_1}\times\dots\times\Z_{m_d}$. This implies that $\Gamma/2\Gamma\cong\Z_2^l$ for some $l\leq d$ and moreover that $\Gamma/2\Gamma\cong\Z_2^d$ if and only if all $m_i$ are even if and only if $L\subset 2L^*$.
\end{proof}

The assumption is made precisely so that $\DiscS(u,v)=\e(\langle u,v\rangle/2)=1$ for $u,v\in L$, instead of taking values in $\{\pm 1\}$, and hence that the $2$-cocycle $\epsilon$ on $L$ becomes trivial. The particular representing abelian $3$-cocycle $(\sigma,\omega)$ for $\smash{\Vect_\Gamma^Q}$ from \cite{GLM24a} then simplifies to
\begin{align*}
\sigma(a,b)&=\DiscS(\hat{a},\hat{b})=\e(\langle\hat{a},\hat{b}\rangle/2),\\
\omega(a,b,c)&=\DiscS(\hat{a},u(b,c))=\e(\langle\hat{a},u(b,c)\rangle/2)
\end{align*}
for $a,b,c\in\Gamma$ and hence obeys the additional conditions in \autoref{sec_evenTY}:
\begin{prop}\label{cor_special_lat_disc}
For a lattice $L$ fulfilling the strong evenness condition in \autoref{ass_strongeven}, the representing abelian $3$-cocycle $(\sigma,\omega)$ for the discriminant form $\Gamma=L^*\!/L$ from \autoref*{GLM1sec_latticeDiscriminantForm} in \cite{GLM24a} satisfies the properties in \autoref{ass_strongdisc}.
\end{prop}

\begin{rem}\label{rem_omega_factor}
In particular, this means that $\sigma$ and $\omega(a,\cdot,\cdot)$ for fixed $a\in\Gamma$ are symmetric and that
$\omega(\cdot,b,c)$ factors through a group homomorphism $\Gamma/2\Gamma\to \C^\times$ for fixed $b,c\in\Gamma$.
\end{rem}

By the definition of the dual lattice, $\DiscS$ restricts to a bimultiplicative pairing $L^*\times L\to\{\pm1\}$. Now, if $L$ satisfies \autoref{ass_strongeven}, $\DiscS(\hat a,u)=\e(\langle\hat a,u\rangle/2)$ for $u\in L$ does not depend on the choice of coset representative $\hat{a}$ and thus factors through a pairing $\Gamma\times L\to\{\pm1\}$. Because it only takes values in $\{\pm1\}$, this further factors through a pairing
\begin{equation*}
p\colon\Gamma/2\Gamma\times L/2L\to\{\pm1\},
\end{equation*}
which we can show to be a perfect pairing. This pairing $p$ allows us to identify $\Gamma/2\Gamma\cong\Hom(L/2L,\{\pm1\})$.

This also recovers the interpretation \eqref{eq_conditionEpsilonOne_6} of the stronger evenness condition in \autoref{ass_strongeven}: the $2$-torsion of $\Gamma$ is as large as possible, namely $\smash{\Gamma/2\Gamma\cong\Z_2^{\rk(L)}}$, noting that for a general even lattice $\Gamma/2\Gamma\cong\Z_2^l$ with $l\leq\rk(L)$.

\medskip

We now fix $\bar{\delta}\in\Gamma/2\Gamma$ and a function $q\colon\Gamma\to\C^\times$ with the properties in \autoref{prop_ass_q} stated in \autoref{sec_evenTY}.

To this end, for an even lattice $L$, we consider the function $v\mapsto\e(\langle v,v\rangle/4)$ taking values in $\{\pm1\}$, which is a certain choice of square root of the natural (and trivial) quadratic form $\e(\langle v,v\rangle/2)=1$ on $L$. This descends to a function
\begin{equation*}
\delta\colon L/2L\to\{\pm 1\},\quad \delta(v+2L)\coloneqq\e(\langle v,v\rangle/4).
\end{equation*}
Under the additional evenness condition in \autoref{ass_strongeven}, i.e.\ if $\langle u,v\rangle\in 2\Z$ for $u,v\in L$, this function is a group homomorphism $\delta\in\Hom(L/2L,\{\pm1\})$. Via the pairing $p$, we identify this $\delta$ with a class $\bar{\delta}\in\Gamma/2\Gamma$ (cf.\ \autoref{prop_ass_q}).

On the other hand, we can use the choice of representatives for $\Gamma=L^*\!/L$ to define a function
\begin{equation*}
q\colon\Gamma\to\C^\times,\quad q(a)\coloneqq \e(\langle\hat{a},\hat{a}\rangle/4)
\end{equation*}
satisfying $q(a)^2=\sigma(a,a)$ for $a\in\Gamma$. Under the additional evenness condition in \autoref{ass_strongeven} it satisfies
\begin{equation}\label{formula_halfQadditivity}
q({a+b})q({a})^{-1}q({b})^{-1}=\sigma({a},{b})\omega(a+b,a,b)^{-1}\delta(u(a,b))
\end{equation}
for all $a,b\in\Gamma$. Identifying $\delta$ with its corresponding element in $\Gamma/2\Gamma$ via the pairing~$p$, we may replace $\delta(u(a,b))$ by $\smash{\DiscS(\bar{\delta},u(a,b))=\baromega(\bar{\delta},a,b)}$ so that we recover the conditions in \autoref{prop_ass_q}. Thus, we proved:
\begin{prop}\label{defi_latqqbar}
Given a lattice $L$ fulfilling the stronger evenness condition in \autoref{ass_strongeven} and given the representing abelian $3$-cocycle for the discriminant form $\Gamma=L^*\!/L$ from \autoref*{GLM1sec_latticeDiscriminantForm} in \cite{GLM24a}, the above choices of $q$ and $\bar{\delta}$ satisfy the conditions in \autoref{prop_ass_q}.
\end{prop}
Overall, we have shown how to produce the data used to define the braided $\Z_2$-crossed tensor category in \autoref{sec_evenTY} directly from an even lattice $L$ with $\Gamma=L^*\!/L$.

\medskip

We finally comment on the partial Gauss sum in \autoref{conj_partialGauss1} with the choices made for $\Gamma=L^*\!/L$ in this section:
\begin{prop}\label{conj_partialGauss2}
In the situation of \autoref{defi_latqqbar}, the partial Gauss sum $G_{\bar{\delta}}(\Gamma,q^{-1})$ in \autoref{conj_partialGauss1} evaluates to
\begin{align*}
G_{\bar{\delta}}(\Gamma,q^{-1})=|2\Gamma|^{-1/2}\sum_{{a\in \bar{\delta}}} q(a)^{-1}=\e(-\sign(L)/8)=\e(-\sign(\Gamma)/8).
\end{align*}
\end{prop}

Note that here, in contrast to \autoref{conj_partialGauss1}, the value of the partial Gauss sum is also well-defined for the indecomposable $2$-adic Jordan components $2_t^{\pm1}$ because we are making a specific choice of $q$. For instance, if $L=\sqrt{2}K$ where $K=\Z$ is the (odd) standard lattice, then $\Gamma=L^*\!/L\cong 2_1^{+1}\cong 2_5^{-1}$ and the Gauss sum takes the value $G_{\bar{\delta}}(\Gamma,q^{-1})=\e(-\sign(L)/8)=\e(7/8)$, and not $\e(3/8)$, as would have been possible as well. We discuss this example further:
\begin{ex}\label{ex_gauss}
We consider the Gauss sum $G_{\bar{\delta}}(\Gamma,q^{-1})$, $\Gamma=L^*\!/L$, for the even lattice $L=\sqrt{n}\Z\cong\sqrt{n/2}A_1$ where $n\in 2\Ns$. When $n=q$ is a power of $2$, this is already discussed in \autoref{rem_gauss}. This lattice is of the form $L=\sqrt{2}K$ for the integral lattice $K=\sqrt{n/2}\Z$, i.e.\ $L$ satisfies \autoref{ass_strongeven}. The dual lattice of $L$ is $L^*=(1/\sqrt{n})\Z$ so that $\Gamma\cong\Z_n$, and the signature of $L$ is $\sign(L)=1\pmod{8}$.

Hence, with the choices made for $(q,\bar{\delta})$ in this section, \autoref{conj_partialGauss2} asserts that the Gauss sum evaluates to $G_{\bar{\delta}}(\Gamma,q^{-1})=\e(7/8)$. Note that if $4\mid n$, then already $K$ is even and the Gauss sum will not depend on these choices (cf.\ the discussion after \autoref{conj_partialGauss1}).

In the following, we compute the Gauss sum for $\Gamma$ explicitly, and verify \autoref{conj_partialGauss2} for this particular case. We choose the coset representatives $\hat{a}=k/\sqrt{n}$ with $k=0,\ldots,n-1$ for $\Gamma=L^*\!/L$. Then, the function $q\colon\Gamma\to\C^\times$ is
\begin{equation*}
q(k/\sqrt{n}+L)=\e(k^2/(4n))
\end{equation*}
for $k=0,\ldots,n-1$. Further, $L/2L=(\sqrt{n}\Z)/(2\sqrt{n}\Z)\cong\Z_2$ and $\delta\colon L/2L\to\{\pm1\}$ is given by
\begin{equation*}
\delta(\sqrt{n}+2L)=\e(1/(4n))
\end{equation*}
evaluated at the nontrivial element in $L/2L\cong\Z_2$. Hence, $\delta$ is the trivial character if $n=0\pmod{4}$ and the nontrivial character if $n=2\pmod{4}$. That is, considering $\bar{\delta}\in\Gamma/2\Gamma\cong(\sqrt{n}\Z)/(2\sqrt{n}\Z)$, we obtain $\bar{\delta}=2\sqrt{n}\Z$ when $n=0\pmod{4}$ and $\bar{\delta}=\sqrt{n}(1+2\Z)$ when $n=2\pmod{4}$. Hence, the Gauss sum is
\begin{align*}
G_{\bar{\delta}}(\Gamma,q^{-1})&=\sqrt{2/n}\sum_{\substack{k=1\\k\text{ even}}}^{n-1} \e(-k^2/(4n)),\\
G_{\bar{\delta}}(\Gamma,q^{-1})&=\sqrt{2/n}\sum_{\substack{k=1\\k\text{ odd}}}^{n-1} \e(-k^2/(4n))
\end{align*}
if $n=0\pmod{4}$ and $n=2\pmod{4}$, respectively. Finally, an explicit computation shows that $G_{\bar{\delta}}(\Gamma,q^{-1})=\e(7/8)$ in both cases, as asserted by \autoref{conj_partialGauss2}.
\end{ex}


\section{Proof of Main Theorem}\label{sec_app}

\begin{proof}[Proof of \autoref{thm_evenTY}]
We prove the theorem by explicitly checking all coherence conditions in $\cC=\LM(\Gamma,\sigma,\omega,\bar{\delta},\eps\,|\,q,\alpha,\beta)$. Recall that $G=\langle g\rangle\cong\Z_2$.

We shall use without mentioning that $\sigma(\cdot,\cdot)$ and $\omega(a,\cdot,\cdot)$ for any fixed $a\in\Gamma$ are symmetric and that $\omega(\cdot,b,c)$ is multiplicative in the first argument and takes values in $\{\pm 1\}$. We also remark that for the subsequent calculations it is helpful to initially ignore terms of the form $\omega(a,b,-b)$, which appear in several steps of the computation, and to only cancel them in the very end.

\smallskip

\noindent
\emph{Tensor Products:} We check that the pentagon identities hold in $\LM(\Gamma,\sigma,\omega,\bar{\delta},\eps)$, starting with the ones not involving any $\X^{\bar{x}}$, i.e.\ the ones in $\cC_1=\smash{\Vect_\Gamma^Q}$:
\pentagon{$\C_a$}{$\C_b$}{$\C_c$}{$\C_d$}
{$\omega(a,b,c)$}
{$\omega(a,b+c,d)$}
{$\omega(b,c,d)$}
{$\omega(a+b,c,d)$}
{$\omega(a,b,c+d)$}
Of course, this identity is satisfied precisely because $\omega$ is a $3$-cocycle on $\Gamma$, which was used in the construction of the tensor category $\smash{\Vect_\Gamma^\omega}=\cC_1$ in \autoref{ex_tensorVect}.

\smallskip

We now come to the pentagon relations involving a single $\X^{\bar{x}}$. First, we consider
\pentagon{$\X^{\bar{x}}$}{$\C_a$}{$\C_b$}{$\C_c$}
{$\baromega(\bar{x}+\bar{\delta},a,b)$}
{$\baromega(\bar{x}+\bar{\delta},a+b,c)$}
{$\omega(a,b,c)$}
{$\baromega((\bar{x}+\bar{a})+\bar{\delta},b,c)$}
{$\baromega(\bar{x}+\bar{\delta},a,b+c)$}
which holds because $\omega(a,b,c)$ cancels and $\omega$ with fixed first argument is a $2$-cocycle. Furthermore,
\pentagon{$\C_a$}{$\X^{\bar{x}}$}{$\C_b$}{$\C_c$}
{$\sigma(a,b)$}
{$\sigma(a,c)$}
{$\baromega(\bar{x}+\bar{\delta},b,c)$}{$\baromega(\bar{a}+\bar{x}+\bar{\delta},b,c)$}
{$\sigma(a,b+c)$}
\pentagon{$\C_a$}{$\C_b$}{$\X^{\bar{x}}$}{$\C_c$}
{$\baromega(\bar{a}+\bar{b}+\bar{x},a,b)$}
{$\sigma(a,c)$}
{$\sigma(b,c)$}
{$\sigma( {a+b},c)$}
{$\baromega(\bar{a}+\bar{b}+(\bar{x}+\bar{c}),a,b)$}
The first identity holds because the difference between $\sigma(a,b+c)$ and $\sigma(a, b)\sigma(a,c)$, which is the
coboundary of $\sigma(a,\cdot)$, is equal to $\omega(a,\cdot,\cdot)$. The second identity holds because $\sigma$ is symmetric, and so the same formula holds for the coboundary of $\sigma(\cdot,c)$.
Finally, there is the pentagon identity
\pentagon{$\C_a$}{$\C_b$}{$\C_c$}{$\X^{\bar{x}}$}
{$\omega(a,b,c)$}
{$\baromega(\bar{a}+(\bar{b}+\bar{c})+\bar{x}, a,b+c)$}
{$\baromega(\bar{b}+\bar{c}+\bar{x},b,c)$}
{$\baromega((\bar{a}+\bar{b})+\bar{c}+\bar{x}, a+b,c)$}
{$\baromega(\bar{a}+\bar{b}+(\bar{c}+\bar{x}),a,b)$}
which holds because $\omega(a+b+c,\cdot,\cdot)$ and $\baromega(\bar{x},\cdot,\cdot)$ are $2$-cocycles.

\smallskip

We then come to the pentagon relations involving $\X^{\bar{x}}$ and $\X^{\bar{y}}$:
\pentagon{$\X^{\bar{x}}$}{$\X^{\bar{y}}$}{$\C_a$}{$\C_b$}
{$\baromega(\bar{x},t,a)$}
{$\baromega(\bar{x},t+a,b)$}
{$\baromega(\bar{y}+\bar{\delta},a,b)$}
{$\omega(t,a,b)$}
{$\baromega(\bar{x},t,a+b)$}
Here, the initial basis vector is $(e_t\otimes e_a)\otimes e_b$ and the final basis vector is $e_{t+a+b}$. This identity holds because of the condition $\bar t=\bar x + \bar y + \bar{\delta}$, which we solve for~$\bar y$. Then the terms $\baromega(\bar{x},\cdot,\cdot)$ again cancel because of the $2$-cocycle condition. Also consider the pentagon identity
\pentagon{$\X^{\bar x}$}{$\C_a$}{$\X^{\bar y}$}{$\C_b$}
{$\sigma(a,t)$}
{$\baromega(\bar{x}, t,b)$}
{$\sigma(a,b)$}
{$\baromega(\bar{x}+\bar{a}, t,b)$}
{$\sigma(a,t+b)$}
where the initial basis vector is $e_t\otimes e_b$ and the final one is $e_{t+b}$. This pentagon identity holds because again the difference between $\sigma(a,t+b)$ and $\sigma(a,t)\sigma(a,b)$ is caught by $\omega(a,t,b)$. Furthermore,
\pentagon{$\X^{\bar x}$}{$\C_a$}{$\C_b$}{$\X^{\bar y}$}
{$\baromega(\bar{x}+\bar{\delta},a,b)$}
{$\sigma(a+b,t)$}
{$\baromega(\bar{a}+\bar{b}+\bar{y},a,b)$}
{$\sigma(b,t)$}
{$\sigma(a,t)$}
where the initial and final basis vector is $e_t$. The identity holds because of the condition $\bar{t}=(\bar{x}+\bar{a}+\bar{b})+\bar{y}+\bar{\delta}$ and since again $\omega(t,a,b)$ catches the difference between $\sigma(a+b,t)$ and $\sigma(a,t)\sigma(b,t)$. Then, we look at the pentagon identity
\pentagon{$\C_a$}{$\X^{\bar x}$}{$\X^{\bar y}$}{$\C_b$}
{$\baromega(\bar{a}+\bar{x},a,t)$}
{$\omega(a,t,b)$}
{$\baromega(\bar{x},t,b)$}
{$\baromega(\bar{a}+\bar{x},t+a,b)$}
{$\baromega(\bar{a}+\bar{x},a,t+b)$}
Here, the initial basis vector is $e_{a+t}\otimes e_b$ and the final basis vector is $e_a\otimes e_{b+t}$. This identity is true because the terms involving $\baromega(\bar{x},\cdot,\cdot)$ and $\omega(a,\cdot,\cdot)$ individually cancel due to the $2$-cocycle condition. Moreover, consider
\pentagon{$\C_a$}{$\X^{\bar x}$}{$\C_b$}{$\X^{\bar y}$}
{$\sigma(a,b)$}
{$\baromega(\bar{a}+(\bar{x}+\bar{b}),a,t)$}
{$\sigma(b,t)$}
{$\sigma(b,a+t)$}
{$\baromega(\bar{a}+\bar{x},a,t)$}
where the initial basis vector is $e_{a+t}$ and the final basis vector is $e_a\otimes e_t$. This pentagon identity holds as the terms involving $\baromega(\bar{x},\cdot,\cdot)$ and $\omega(a,\cdot,\cdot)$ individually cancel and again the difference between $\sigma(b,t)\sigma(b,a)$ and $\sigma(b,a+t)$ is compensated by $\omega(b,a,t)$. Finally,
\pentagon{$\C_a$}{$\C_b$}{$\X^{\bar x}$}{$\X^{\bar y}$}
{$\baromega(\bar{a}+\bar{b}+\bar{x},a,b)$}
{$\baromega(\bar{a}+(\bar{b}+\bar{x}),a,b+t)$}
{$\baromega(\bar{b}+\bar{x},b,t)$}
{$\baromega(\bar{a}+\bar{b}+\bar{x},a+b,t)$}
{$\omega(a,b,t)$}
where the initial basis vector is $e_{a+b+t}$ and the final basis vector is $e_a\otimes e_b\otimes e_t$. This identity is true because again the terms $\baromega(\bar{x},\cdot,\cdot)$, $\omega(a,\cdot,\cdot)$ and $\omega(b,\cdot,\cdot)$ cancel individually due to the $2$-cocycle condition.

\smallskip

We now come to the pentagon relations involving $\X^{\bar x}$, $\X^{\bar y}$ and $\X^{\bar z}$, starting with:
\pentagon{$\X^{\bar x}$}{$\X^{\bar y}$}{$\X^{\bar z}$}{$\C_a$}
{$\eps |2\Gamma|^{-1/2}\sigma(t,r)^{-1}$}
{$\baromega(\bar{x}+\bar{\delta},r,a)$}
{$\baromega(\bar{y},r,a)$}
{$\sigma(t,a)$}
{$\eps |2\Gamma|^{-1/2}\sigma(t,r+a)^{-1}$}
Here, the initial basis vector is $v_t$ and mapped to $v_r$ by the first arrow, the final basis vector is $v_{r+a}$. This identity holds because of the condition $\bar{t}=\bar{x}+\bar{y}+\bar{\delta}$, because again terms involving $\baromega(\bar{x},\cdot,\cdot)$ cancel and since the difference between $\sigma(t,r)\sigma(t,a)$ and $\sigma(t,r+a)$ is caught by $\omega(t,r,a)$. The normalisation factor $|2\Gamma|$ cancels independently of its value. Then we consider
\pentagon{$\X^{\bar x}$}{$\X^{\bar y}$}{$\C_a$}{$\X^{\bar z}$}
{$\baromega(\bar{x},t,a)$}
{$\eps |2\Gamma|^{-1/2}\sigma( {t+a},r)^{-1}$}
{$\sigma(a,r)$}
{$\baromega(\bar{t}+\bar{a}+\bar{z},t,a)$}
{$\eps |2\Gamma|^{-1/2}\sigma(t,r)^{-1}$}
where the initial basis vector is $v_{t+a}$ with $\bar{t}=\bar{x}+\bar{y}+\bar{\delta}$ and the final basis vector is $v_r$ with $\bar{r}=\bar{y}+(\bar{a}+\bar{z})+\bar{\delta}$. This means that $\bar{x}+\bar{z}=\bar{t}+\bar{r}+\bar{a}$. This pentagon identity then holds because $\omega(r,t,a)=\sigma(t,r)\sigma(a,r)\sigma(t+a,r)^{-1}$. Next, we look at
\pentagon{$\X^{\bar x}$}{$\C_a$}{$\X^{\bar y}$}{$\X^{\bar{z}}$}
{$\sigma(a,t)$}
{$\eps |2\Gamma|^{-1/2}\sigma(t,a+r)^{-1}$}
{$\baromega(\bar{a}+\bar{y},a,r)$}
{$\eps |2\Gamma|^{-1/2}\sigma(t,r)^{-1}$}
{$\baromega(\bar{x}+\bar{\delta},a,r)$}
Here, the initial basis vector is $v_t$ with $\bar{t}=(\bar{x}+\bar{a})+\bar{y}+\bar{\delta}$. Then this identity holds like in the previous case. Finally, consider the pentagon identity
\pentagon{$\C_a$}{$\X^{\bar x}$}{$\X^{\bar y}$}{$\X^{\bar z}$}
{$\baromega(\bar{a}+\bar{x},a,t)$}
{$\baromega(\bar{a}+\bar{t}+\bar{z},a,t)$}
{$\eps |2\Gamma|^{-1/2}\sigma(t,r)^{-1}$}
{$\eps |2\Gamma|^{-1/2}\sigma(t+a,r)^{-1}$}
{$\sigma(a,r)$}
where the initial basis vector is $v_{t+a}$ with $\overline{t+a}=(\bar{x}+\bar{a})+\bar{y}+\smash{\bar{\delta}}$ and the final basis vector is $e_a\otimes v_r$ with $\bar{r}=\bar{y}+\bar{z}+\smash{\bar{\delta}}$. Then, this identity holds after substituting these two expressions like in the previous cases.

\smallskip

We now come to the last type of pentagon relations, involving $\X^{\bar x}$, $\X^{\bar y}$, $\X^{\bar z}$ and~$\X^{\bar w}$. The innermost product $\X^{\bar x}\otimes\X^{\bar y}$ produces, after multiplication with further $\X$, a basis vector in the multiplicity space denoted by $v_t$, while the outermost product produces the resulting object $\C_s$ with basis vector denoted by $e_s$. These bases in the five products in the above diagram are, from left to right, $v_t\otimes e_{s+t}$, $v_r\otimes e_{s+t}$, $v_r\otimes e_{s+t}$, $v_l\otimes e_{s+t}$ and for the bottom node $v_t\otimes e_{s}$. With these bases, the pentagon identity reads
\pentagon{$\X^{\bar x}$}{$\X^{\bar y}$}{$\X^{\bar z}$}{$\X^{\bar w}$}
{$\eps |2\Gamma|^{-1/2}\sigma(t,r)^{-1}$}
{$\sigma(r,s+t)$}
{$\eps |2\Gamma|^{-1/2}\sigma(r,l)^{-1}$}
{$\baromega(\bar{t}+\bar{z},t,s)$}
{$\baromega(\bar{x},t,s)$}
Again, the conditions $\bar{t}=\bar{x}+\bar{y}+\bar{\delta}$, $\bar{r}=\bar{y}+\bar{z}+\bar{\delta}$ and $\bar{s}=\bar{l}=\bar{z}+\bar{w}+\bar{\delta}$ hold. Again, $\sigma(r,s+t)$ can be expressed as $\sigma(r,t)$, which cancels, times $\sigma(r,s)$ times $\omega(r,s,t)$. Then, after substituting $\bar{r}$ and $\bar{t}$, all $\omega$ cancel, as does the sign $\eps$. However, the factor $|2\Gamma|^{-1/2}|2\Gamma|^{-1/2}=|2\Gamma|^{-1}$ becomes important. The last pentagon identity then follows from the following slightly unusual character sum:
\begin{lem}\label{lm_characterSum}
For $s,l\in\Gamma$ with $\bar{s}=\bar{l}\in\Gamma/2\Gamma$ and for a coset $\bar{r}\in\Gamma/2\Gamma$,
\begin{align*}
|2\Gamma|^{-1}\sum_{r\in \bar r}
\sigma(r,s)
\sigma(r,l)^{-1}
&=\delta_{s,l}.
\end{align*}
\end{lem}
\begin{proof}
The summand can be rewritten as $\sigma({r},{s})\sigma(r,l)^{-1}=\omega(r,s,l-s)\sigma(r,l-s)$. Let $r_0$ be some fixed coset representative for $\bar{r}$, so that we may write the summation variable as $r=r_0+a$ with $a\in 2\Gamma$. Now, the assumption $\bar{s}=\bar{l}$ means that $l-s\in 2\Gamma$, so that we may write $l-s=2y$ for some $y\in\Gamma$. Then
\begin{equation*}
\sigma(x,2y)=\sigma(x,y)^2\baromega(\bar{x},y,y)=B_Q(x,y)\baromega(\bar{x},y,y)
\end{equation*}
for all $x\in\Gamma$, and hence $\sigma(\cdot,2y)$ is a character on $\Gamma$. Moreover, $\omega(r,s,l-s)$ is always a character as a function of $r$, and on any fixed $2\Gamma$-coset $\bar{r}$ it is the constant $\omega(\bar{r},s,l-s)$. Hence, the character sum is
\begin{equation*}
|2\Gamma|^{-1}\!\!\sum_{r\in r_0+2\Gamma}\!\!\sigma(r,s)\sigma(r,l)^{-1}=|2\Gamma|^{-1}\omega(\bar{r},s,l-s)\sigma(r_0,l-s)\sum_{a\in 2\Gamma}\sigma(a,l-s).
\end{equation*}
By the orthogonality of characters,
\begin{equation*}
|2\Gamma|^{-1}\sum_{a\in 2\Gamma}\sigma(a,2y)=
\begin{cases}
1 & \text{if } \sigma(\cdot,2y)=1 \text{ on } 2\Gamma,\\
0 & \text{otherwise.}
\end{cases}
\end{equation*}
Since $\sigma(2u,2y)=\sigma(u,2y)^2=B_Q(u,2y)$ for all $u\in\Gamma$, this character is trivial on $2\Gamma$ if and only if $B_Q(\cdot,2y)$ is trivial on $\Gamma$, i.e.\ if $2y=0$. Hence, the sum vanishes unless $l=s$, and in this case the factor $\omega(\bar{r},s,l-s)\sigma(r_0,l-s)$ is $1$.
\end{proof}

\noindent
\emph{$G$-Action:} We now discuss the $\Z_2$-action on $\LM(\Gamma,\sigma,\omega,\bar{\delta},\eps)$; cf.\ \autoref{def_Gcrossed}. The compatibility between the trivial composition structure and the tensor structure amounts to the fact that all tensor structures are scalars in $\{\pm 1\}$ and hence square to one.

To prove the coherences of the tensor structures, which have to compensate for the associators not being $\Z_2$-invariant, we collect some formulae on $\sigma$ and $\omega$ with negative entries: by definition,
\begin{equation}\label{eq_sigmanegativeR}
\sigma(a,-b)=\sigma(a,b)^{-1}\omega(a,b,-b)^{-1}
\end{equation}
and by symmetry
\begin{equation}\label{eq_sigmanegativeL}
\sigma(-a,b)=\sigma(a,b)^{-1}\omega(b,a,-a)^{-1}
\end{equation}
for $a,b\in\Gamma$. From this we derive the following additivity property for $\omega(a,b,-b)$ in~$b$, which can be solved to express $\omega(a,-b,-c)$:
\begin{align}\label{eq_omeganegative}
\omega(a,b+c,-b-c)^{-1}
&=\sigma(a,b+c)\sigma(a,-b-c)\\
&=\sigma(a,b)\sigma(a,c)\omega(a,b,c)
\sigma(a,-b)\sigma(a,-c)\omega(a,-b,-c) \notag\\
&=\omega(a,b,-b)^{-1}\omega(a,c,-c)^{-1}
\omega(a,b,c)\omega(a,-b,-c)\notag
\end{align}
for $a,b,c\in\Gamma$. We also note for later use that from the $2$-cocycle condition for $\omega(a,\cdot,\cdot)$ it follows immediately that
\begin{equation}\label{eq_bplusc}
\omega(a,-b,b+c)=\omega(a,b,c)^{-1}\omega(a,b,-b)
\end{equation}
for $a,b,c\in\Gamma$.

Formula~\eqref{eq_omeganegative} then (re)proves the first of the coherence identities:
\longsquare{$\C_a$}{$\C_b$}{$\C_c$}
{$\omega(a,b,c)$}{$\omega(-a,-b,-c)$}
{$\omega(a+b,c,-c)$}{$\omega(a,b,-b)$}
{$\omega(a,b+c,-b-c)$}{$\omega(b,c,-c)$}
Similarly, the identities involving a single $\X^{\bar x}$ hold:
\longsquare{$\X^{\bar x}$}{$\C_a$}{$\C_b$}
{$\baromega(\bar{x}+\bar{\delta},a,b)$}{$\baromega(\bar{x}+\bar{\delta},-a,-b)$}
{$\baromega((\bar{x}+\bar{a})+\bar{\delta},b,-b)$}{$\baromega(\bar{x}+\bar{\delta},a,-a)$}
{$\baromega(\bar{x}+\bar{\delta},a+b,-a-b)$}{$\omega(a,b,-b)$}
\longsquare{$\C_a$}{$\X^{\bar x}$}{$\C_b$}
{$\sigma(a,b)$}{$\sigma(-a,-b)$}
{$\baromega((\bar{a}+\bar{x})+\bar{\delta},b,-b)$}{$\baromega(\bar{a}+\bar{x},a,-a)$}
{$\baromega(\bar{a}+(\bar{x}+\bar{b}),a,-a)$}{$\baromega(\bar{x}+\bar{\delta},b,-b)$}
\longsquare{$\C_a$}{$\C_b$}{$\X^{\bar x}$}
{$\baromega(\bar{a}+\bar{b}+\bar{x},a,b)$}{$\baromega(-\bar{a}-\bar{b}+\bar{x},-a,-b)$}
{$\baromega(\bar{a}+\bar{b}+\bar{x},a+b,-a-b)$}{$\omega(a,b,-b)$}
{$\baromega(\bar{a}+(\bar{b}+\bar{x}),a,-a)$}{$\baromega(\bar{b}+\bar{x},b,-b)$}
For the identities involving $\X^{\bar x}$ and $\X^{\bar y}$, we have to be again explicit about the bases in the multiplicity spaces and moreover consider the action of $g_*$ on them:
\longsquare{$\X^{\bar x}$}{$\X^{\bar y}$}{$\C_a$}
{$\baromega(\bar{x},t,a)$}{$\baromega(\bar{x},-t,-a)$}
{$\omega(t,a,-a)$}{$\baromega(\bar{x},t,-t)$}
{$\baromega(\bar{x},t+a,-t-a)$}{$\baromega(\bar{y}+\bar{\delta},a,-a)$}
Here, we used the condition $\bar y+\bar{\delta}=\bar x+\bar t$ for the basis vector $g_*(e_t\otimes e_a)$ (top left) mapped to the basis vector $g_*(e_{t+a})$ (bottom left) and $e_{-t}\otimes e_{-a}$ (top right) to $e_{-t-a}$ (bottom right). The identity is again essentially the formula~\eqref{eq_omeganegative} for $\baromega(\bar{x},t+a,-t-a)$. Then, consider
\longsquare{$\X^{\bar x}$}{$\C_a$}{$\X^{\bar y}$}
{$\sigma(a,t)$}{$\sigma(-a,-t)$}
{$\baromega(\bar{x}+\bar{a},t,-t)$}{$\baromega(\bar{x}+\bar{\delta},a,-a)$}
{$\baromega(\bar{x},t,-t)$}{$\baromega(\bar{a}+\bar{y},a,-a)$}
where we used the condition $\bar y=\bar x+(\bar t+\bar a)+\bar{\delta}$. Next, look at
\longsquare{$\C_a$}{$\X^{\bar x}$}{$\X^{\bar y}$}
{$\baromega(\bar{a}+\bar{x},a,t)$}{$\baromega(-\bar{a}+\bar{x},-a,-t)$}
{$\baromega(\bar{a}+\bar{x},a+t,-a-t)$}{$\baromega(\bar{a}+\bar{x},a,-a)$}
{$\omega(a,t,-t)$}{$\baromega(\bar{x},t,-t)$}
Here, we did not use the condition $\bar y=\bar x+\bar t+\bar{\delta}$. Finally, we come to
\longsquare{$\X^{\bar x}$}{$\X^{\bar y}$}{$\X^{\bar z}$}
{$\eps |2\Gamma|^{-1/2}\sigma(t,r)^{-1}$}{$\eps |2\Gamma|^{-1/2}\sigma(-t,-r)^{-1}$}
{$\baromega(\bar{t}+\bar{z},t,-t)$}{$\baromega(\bar{x},t,-t)$}
{$\baromega(\bar{x}+\bar{\delta},r,-r)$}{$\baromega(\bar{y},r,-r)$}
where $t$ and $r$ satisfy, as for the associator, the conditions $\bar t=\bar x+\bar y+\bar{\delta}$ and $\bar r=\bar y+\bar z+\bar{\delta}$. This lets us rewrite the expressions
\begin{equation*}
\baromega(\bar{z}+\bar{x},t,-t)=\omega(t+r,t,-t),\quad
\baromega(\bar{x}+\bar{\delta}+\bar{y},r,-r)=\omega(t,r,-r).
\end{equation*}
Then, the identity follows from the formulae \eqref{eq_sigmanegativeR} and \eqref{eq_sigmanegativeL} for $\sigma(r,s)\sigma(-r,-s)^{-1}$. This concludes the proof of the tensor structure. We remark that we did not use any information on $\eps$ or $|2\Gamma|$.

\smallskip

\noindent
\emph{Compatibility of $G$-action and $G$-braiding:} We verify the commutativity of the diagram in condition~(b) of \autoref{def_G-crossed}. Since our $G$-action is strict ($\smash{T_2^{g,h}}=\id$), this amounts to the equality
\begin{equation*}
\tau^g_{h_*Y,X}\circ g_*(c_{X,Y})=c_{g_*(X),g_*(Y)}\circ \tau^g_{X,Y}
\end{equation*}
for $g,h\in G$, $X\in\cC_h$, $Y\in\cC$. We check this degree-wise.

If $X,Y\in\cC_1$, the identity reduces to $\sigma(-a,-b)\omega(a,b,-b)=\omega(b,a,-a)\sigma(a,b)$ for all $a,b\in\Gamma$, which holds because $\omega(a,b,-b)=\omega(-a,b,-b)$ since $\omega(\cdot,b,-b)$ factors through~$\Gamma/2\Gamma$.

Now, let $X\in\cC_g$ and $Y\in\cC_1$. Take $X=\X^{\bar x}$ and $Y=\C_a$ for $\bar{x}\in\Gamma/2\Gamma$ and $a\in\Gamma$. Using the braiding \eqref{eq_LM_braiding} and the tensor structures \eqref{eq_LM_tau},
{\allowdisplaybreaks
\begin{align*}
c_{X,Y}=q(a)^{-1}\baromega(\bar{x}+\bar{a},a,-a)&\colon \X^{\bar x}\!\otimes\!\C_a\to \C_{-a}\!\otimes\!\X^{\bar x},\\
\tau^g_{X,Y}=\baromega(\bar{x}+\bar{\delta},a,-a)&\colon g_*(\X^{\bar x}\!\otimes\!\C_a)\to \X^{\bar x}\!\otimes\!\C_{-a},\\
\tau^g_{g_*Y,X}=\baromega(\bar{x}+\bar{a},a,-a)&\colon g_*(\C_{-a}\!\otimes\!\X^{\bar x})\to \C_a\!\otimes\!\X^{\bar x},\\
c_{g_*X,g_*Y}=q(-a)^{-1}\baromega(\bar{x}+\bar{a},a,-a)&\colon \X^{\bar x}\!\otimes\!\C_{-a}\to \C_a\!\otimes\!\X^{\bar x},
\end{align*}
}%
where in the last two lines we used $\baromega(\bar{x}-\bar{a},-a,a)=\baromega(\bar{x}+\bar{a},a,-a)$ since $\bar{x}-\bar{a}=\bar{x}+\bar{a}$ in $\Gamma/2\Gamma$ and $\omega$ is symmetric in the last two arguments.
Hence, the commutativity reduces to the scalar identity
\begin{equation*}
\baromega(\bar{x}+\bar{a},a,-a)\,q(a)^{-1}\baromega(\bar{x}+\bar{a},a,-a)
=
q(-a)^{-1}\baromega(\bar{x}+\bar{a},a,-a)\,\baromega(\bar{x}+\bar{\delta},a,-a).
\end{equation*}
Since $\baromega(\bar{x}+\bar{a},a,-a)^2=1$, the left-hand side simplifies to $q(a)^{-1}$. On the right-hand side, $\baromega(\bar{x}+\bar{a},a,-a)\baromega(\bar{x}+\bar{\delta},a,-a)=\baromega(\bar{a}+\bar{\delta},a,-a)$ by the homomorphism property in the first argument and $\baromega(2\bar{x},a,-a)=1$. Thus, the identity becomes
\begin{equation*}
q(a)^{-1}=q(-a)^{-1}\baromega(\bar{a}+\bar{\delta},a,-a),
\end{equation*}
which is equivalent to \eqref{eq_qminus}.

The case $X\in\cC_1$, $Y\in\cC_g$ is analogous: for $X=\C_a$, $Y=\X^{\bar x}$ with $a\in\Gamma$ and $\bar{x}\in\Gamma/2\Gamma$,
\begin{align*}
c_{X,Y}&=q(a)^{-1}\colon \C_a\!\otimes\!\X^{\bar x}\to \X^{\bar x}\!\otimes\!\C_a,
&
c_{g_*X,g_*Y}&=q(-a)^{-1}\colon \C_{-a}\!\otimes\!\X^{\bar x}\to \X^{\bar x}\!\otimes\!\C_{-a},\\
\tau^g_{X,Y}&=\baromega(\bar{a}+\bar{x},a,-a),
&
\tau^g_{g_*Y,X}&=\tau^g_{\X^{\bar x},\C_a}=\baromega(\bar{x}+\bar{\delta},a,-a),
\end{align*}
and the identity again reduces to \eqref{eq_qminus} by the same computation as above.

Finally, consider the case $X,Y\in\cC_g$ with $X=\X^{\bar x}$, $Y=\X^{\bar y}$ for $\bar{x},\bar{y}\in\Gamma/2\Gamma$, so that $X\otimes Y=\bigoplus_{\bar t=\bar x+\bar y+\bar{\delta}}\C_t$.
Fix one summand $\C_t$. The functor $g_*$ sends $\C_t$ to $\C_{-t}$, so we trace the diagram for the summand $\C_{-t}$ of $g_*(X\otimes Y)$. The relevant scalars are
\begin{equation*}
g_*(c_{X,Y})=\alpha\,q(t),\quad\!\!
c_{g_*X,g_*Y}=\alpha\,q(-t),\quad\!\!
\tau^g_{X,Y}=\baromega(\bar{x},t,-t),\quad\!\!
\tau^g_{g_*Y,X}=\baromega(\bar{y},t,-t)
\end{equation*}
by \eqref{eq_LM_braiding} and \eqref{eq_LM_tau}, where $g_*(c_{X,Y})$ carries the scalar $\alpha\,q(t)$ from the original summand $\C_t$, while $c_{g_*X,g_*Y}$ acts on $\C_{-t}$ with scalar $\alpha\,q(-t)$. Hence, the commutativity reduces to
\begin{equation*}
\baromega(\bar{y},t,-t)\,\alpha\, q(t)=\alpha\, q(-t)\,\baromega(\bar{x},t,-t).
\end{equation*}
Cancelling $\alpha$ and applying \eqref{eq_qminus}, i.e.\ $q(-t)=q(t)\baromega(\bar{t}+\bar{\delta},t,-t)$, the right-hand side becomes $q(t)\baromega(\bar{t}+\bar{\delta}+\bar{x},t,-t)$. Now, $\bar{t}=\bar{x}+\bar{y}+\bar{\delta}$ implies that $\bar{t}+\bar{\delta}+\bar{x}=\bar{y}$ in $\Gamma/2\Gamma$, and hence the identity holds.
\smallskip

\noindent
\emph{$G$-Braiding:} For the $\Z_2$-braiding on $\LM(\Gamma,\sigma,\omega,\bar{\delta},\eps\,|\,q,\alpha,\beta)$ we verify the hexagon identities; cf.\ \autoref{def_G-crossed}.

The hexagon identity for $\C_a$, $\C_b$, $\C_c$ and $\C_d$ is known to hold in $\cC_1=\smash{\Vect_\Gamma^{(\sigma,\omega)}}$, corresponding to the abelian $3$-cocycle $(\sigma,\omega)$. It is satisfied as $\omega(b,c,a)=\omega(b,a,c)$ and $\omega(a,\cdot,\cdot)$ is the coboundary of $\sigma(a,\cdot)$, and similarly for the inverse hexagon identity:
\hexagon{$\C_a$}{$\C_b$}{$\C_b$}{$\C_c$}{$\C_c$}
{$\omega(a,b,c)$}{$\sigma(a,b+c)$}{$\omega(b,c,a)$}
{$\sigma(a,b)$}{$\omega(b,a,c)$}{$\sigma(a,c)$}

\smallskip

We come to the hexagon identities involving one $\X^{\bar x}$, starting with
\hexagon{$\X^{\bar x}$}{$\C_a$}{$\C_{-a}$}{$\C_b$}{$\C_{-b}$}
{$\baromega(\bar{x}+\bar{\delta},a,b)$}{$\substack{q(a+b)^{-1}\baromega(\bar{x}+\bar{a}+\bar{b},a+b,-a-b)\\ \omega(a,b,-b)}$}{$\baromega(-\bar{a}-\bar{b}+\bar{x},-a,-b)$}
{$q(a)^{-1}\baromega(\bar{x}+\bar{a},a,-a)$}{$\sigma(-a,b)$}{$q(b)^{-1}\baromega(\bar{x}+\bar{b},b,-b)$}
Take note of the appearance of a tensor structure in the top arrow. This hexagon identity can be rewritten using formula~\eqref{eq_sigmanegativeL} for $\sigma(-a,b)$ and formula~\eqref{eq_omeganegative} for $\baromega(\bar{x}+\bar{a}+\bar{b},a+b,-a-b)$. Then all terms $\baromega(\bar{x},\cdot,\cdot)$ cancel and all terms $\omega(a+b,\cdot,\cdot)$ cancel except for a new $\omega(a+b,a,b)$. Then the identity is precisely the defining property of $q$, $q(a+b)q(a)^{-1}q(b)^{-1}=\sigma(a,b)\,\omega(a+b,a,b)\baromega(\bar{\delta},a,b)$ for all $a,b\in\Gamma$. Furthermore, consider
\hexagon{$\C_a$}{$\X^{\bar x}$}{$\X^{\bar x}$}{$\C_b$}{$\C_b$}
{$\sigma(a,b)$}{$q(a)^{-1}$}{$\baromega(\bar{x}+\bar{\delta},b,a)$}
{$q(a)^{-1}$}{$\baromega(\bar{x}+\bar{\delta},b,a)$}{$\sigma(a,b)$}
\hexagon{$\C_a$}{$\C_b$}{$\C_b$}{$\X^{\bar x}$}{$\X^{\bar x}$}
{$\baromega(\bar{a}+\bar{b}+\bar{x},a,b)$}{$q(a)^{-1}$}{$\sigma(b,a)$}
{$\sigma(a,b)$}{$\baromega(\bar{b}+\bar{a}+\bar{x},b,a)$}{$q(a)^{-1}$}
Both identities hold directly. If the braiding were not given, these identities would show that the braiding of $\C_a\otimes\X^{\bar x}$ coincides with the braiding of $\C_a\otimes\X^{\bar x+\bar b}$, so that it is independent of $\bar x$.

The inverse hexagon identities in these cases are:
\hexagonInverse{$\X^{\bar x}$}{$\C_a$}{$\C_b$}{$\C_b$}{$\C_{-b}$}
{$\baromega(\bar{x}+\bar{\delta},a,b)^{-1}$}{$q(b)^{-1}\baromega((\bar{x}+\bar{a})+\bar{b},b,-b)$}
{$\sigma(-b,a)^{-1}$}
{$\sigma(a,b)$}
{$\baromega(\bar{x}+\bar{\delta},b,a)^{-1}$}
{$q(b)^{-1}\baromega(\bar{x}+\bar{b},b,-b)$}
\hexagonInverse{$\C_a$}{$\X^{\bar x}$}{$\C_b$}
{$\C_{-b}$}{$\C_{-b}$}
{$\sigma(a,b)^{-1}$}
{$q(b)^{-1}\baromega((\bar{x}+\bar{a})+\bar{b},b,-b)$}
{$\baromega(-\bar{b}+\bar{a}+\bar{x},-b,a)^{-1}$}
{$q(b)^{-1}\baromega(\bar{x}+\bar{b},b,-b)$}
{$\baromega(\bar{a}-\bar{b}+\bar{x},a,-b)^{-1}$}
{$\sigma(a,-b)$}
In these two identities, most terms cancel, and we use formula~\eqref{eq_sigmanegativeL} to express $\sigma(-b,a)$ and $\sigma(a,-b)$ in terms of $\sigma(a,b)$ and $\omega(a,b,-b)$, respectively, which cancels with the corresponding term on the top arrow. If the braiding were not given, these identities would show that the braiding of $\X^{\bar x}\otimes \C_b$ coincides with the braiding of $\X^{\bar x+\bar a} \otimes \C_b $ up to a factor $\omega(a,b,-b)$, so it demands the $\bar{x}$-dependency of $\baromega(\bar{x},b,-b)$. Next, we consider
\hexagonInverse{$\C_a$}{$\C_b$}{$\X^{\bar x}$}{$\X^{\bar x}$}{$\X^{\bar x}$}
{$\baromega(\bar{a}+\bar{b}+\bar{x},a,b)^{-1}$}
{$q(a+b)^{-1}$}
{$\baromega(\bar{x}+\bar{\delta},a,b)^{-1}$}
{$q(b)^{-1}$}
{$\sigma(a,b)^{-1}$}
{$q(a)^{-1}$}
This hexagon identity again amounts to the defining property of $q$. If the braiding were not given, one might wish to treat this identity (involving the easier braiding $\C_a\otimes \X^ {\bar x}$ and no tensor structure) before the analogous hexagon identity $\X^{\bar x}\otimes \C_a\otimes \C_b$ above (involving the braiding $\X^{\bar x}\otimes \C_a$).

\smallskip

We come to the hexagon identities involving $\X^{\bar x}$ and $\X^{\bar y}$. First, consider
\hexagon{$\X^{\bar x}$}{$\X^{\bar y}$}{$\X^{\bar y}$}{$\C_a$}{$\C_{-a}$}
{$\baromega(\bar{x},t,a)$}
{$\substack{ \alpha q(t+a) \\ \baromega(\bar{y}+\bar{\delta},a,-a)}$}
{$\sigma(-a,t+a)$}
{$ \alpha q(t)$}
{$\baromega(\bar{y},t,a)$}
{$q(a)^{-1}\baromega(\bar{x}+\bar{a},a,-a)$}
We apply formula~\eqref{eq_sigmanegativeL} to $\sigma(-a,t+a)$ and the condition $\bar{t}=\bar{x}+\bar{y}+\bar{\delta}$ to $\omega(\cdot,a,-a)$ and $\omega(\cdot,t,a)$. The parameter $\alpha$ simply cancels at this point. Then the hexagon identity amounts to the other defining condition of $q$, $q(a)^2=\sigma(a,a)$ for $a\in\Gamma$, and again the additivity condition on $q$. If the braiding were not given, this identity would show that the braiding of $\X^{\bar x}\otimes \X^{\bar y}$ and that of $\X^{\bar x}\otimes \X^{\bar y+\bar a}$ differ by the braiding of $\X^{\bar x}\otimes \C_a$ by this additivity condition. Next, we look at
\hexagon{$\X^{\bar x}$}{$\C_a$}{$\C_{-a}$}{$\X^{\bar y}$}{$\X^{\bar y}$}
{$\sigma(a,t)$}
{$\substack{ \alpha q(t) \\ \baromega(\bar{a}+\bar{y},a,-a)}$}
{$\baromega(-\bar{a}+\bar{y},-a,a+t)$}
{$q(a)^{-1}\baromega(\bar{x}+\bar{a},a,-a)$}
{$\baromega(-\bar{a}+\bar{x},-a,a+t)$}
{$ \alpha q(a+t) $}
where as initial basis vector we choose $e_t$ so that the final basis vector is $e_{-a}\otimes e_{a+t}$, with $\bar{t}=(\bar{x}+\bar{a})+\bar{y}+\bar{\delta}$. Then we collect and rewrite $\baromega(\bar{x}+\bar{y},-a,a+t)\baromega(\bar{x}+\bar{y},a,-a)$ using formula \eqref{eq_bplusc} to $\baromega(\bar{x}+\bar{y},a,t)$, which is equal to $\baromega(\bar{t}+\bar{a}+\bar{\delta},a,t)$. Then the hexagon identity is again the defining additivity relation of $q$. Moreover,
\hexagon{$\C_a$}{$\X^{\bar x}$}{$\X^{\bar x}$}{$\X^{\bar y}$}{$\X^{\bar y}$}
{$\baromega(\bar{x}+\bar{a},a,t)$}
{$\sigma(a,t)$}
{$\baromega(\bar{x},t,a)$}
{$q(a)^{-1}$}
{$\sigma(a,t+a)$}
{$q(a)^{-1}$}
where as initial basis vector we choose $e_{a+t}$. This hexagon identity holds again because of the defining relation
$q(a)^2=\sigma(a,a)$.

The inverse hexagon identities in these cases are as follows:
\hexagonInverse{$\X^{\bar x}$}{$\X^{\bar y}$}{$\C_a$}{$\C_{-a}$}{$\C_a$}
{$\baromega(\bar{x},t,a)^{-1}$}
{$\sigma(t,a)$}
{$\baromega(\bar{a}+\bar{x},a,t)^{-1}$}
{$q(a)^{-1}\baromega(\bar{y}+\bar{a},a,-a)$}
{$\sigma(-a,t+a)^{-1}$}
{$q(-a)^{-1}\baromega(\bar{x}-\bar{a},-a,a)$}
Here, as initial basis vector we choose $e_{t+a}$ with $\bar{t}=\bar{x}+\bar{y}+\bar{\delta}$ and consequently write $\baromega(\bar{y}+\bar{a},a,-a)\baromega(\bar{x}-\bar{a},a,-a)=\omega(t,a,-a)\baromega(\bar{\delta},a,-a)$. We rewrite, using the coboundary property in \autoref{ass_strongdisc} and equation~\eqref{eq_sigmanegativeL},
\begin{align*}
\sigma(-a,t+a)&=\sigma(-a,t)\sigma(-a,a)\omega(-a,t,a)\\
&=\sigma(a,t)\omega(t,a,-a)\sigma(-a,a)\omega(-a,t,a).
\end{align*}
Then the hexagon identity reduces to the following special case of the additivity property
\begin{equation*}
q(a)q(-a)=\sigma(a,-a)^{-1}\baromega(\bar{\delta},a,-a).
\end{equation*}
For later use, using $\sigma(a,-a)=\sigma(a,a)^{-1}\omega(a,a,-a)$ and $\sigma(a,a)=q(a)^2$, we show further that
\begin{align}\label{eq_qminus}
q(-a)&=q(a)\baromega(\bar{a}+\bar{\delta},a,-a).
\end{align}
We further consider the identities
\hexagonInverse{$\X^{\bar x}$}{$\C_a$}{$\X^{\bar y}$}{$\X^{\bar y}$}{$\X^{\bar y}$}
{$\sigma(a,t+a)^{-1}$}
{$\alpha q(t+a)$}
{$\baromega(\bar{y},t,a)^{-1}$}
{$q(a)^{-1}$}
{$\baromega(\bar{x},t,a)^{-1}$}
{$\alpha q(t)$}
\hexagonInverse{$\C_a$}{$\X^{\bar x}$}{$\X^{\bar y}$}{$\X^{\bar y}$}{$\X^{\bar y}$}
{$\baromega(\bar{x}+\bar{a},a,t)^{-1}$}
{$\alpha q(a+t)$}
{$\sigma(a,a+t)^{-1}$}
{$\alpha q(t)$}
{$\baromega(\bar{y}+\bar{a},a,t)^{-1}$}
{$q(a)^{-1}$}
As initial basis vector in these two diagrams we choose $e_{t+a}$ and $e_a\otimes e_t$, respectively, in both cases with $\bar{t}=\bar{x}+\bar{y}+\bar{\delta}$. In both cases, the hexagon identity follows from the coboundary property for $\sigma(a,a+t)$ and $q(a)^2=\sigma(a,a)$ and the additivity property for $q$.

\smallskip

We finally come to the hexagon identities and inverse hexagon identities involving $\X^{\bar x}$, $\X^{\bar y}$ and $\X^{\bar z}$. These are the only ones where the associator of $\X\otimes \X\otimes \X$ appears, as well as the tensor structure on $\X\otimes \X$. This is also the only identity where the value of $\alpha$ is relevant. We start with the inverse hexagon identity, which is simpler in some regards:
\hexagonInverse{$\X^{\bar x}$}{$\X^{\bar y}$}{$\X^{\bar z}$}{$\X^{\bar z}$}{$\X^{\bar z}$}
{$\eps |2\Gamma|^{-1/2}\sigma(r,t)$}
{$q(t)^{-1}$}
{$\eps |2\Gamma|^{-1/2}\sigma(t,s)$}
{$\alpha q(r)$}
{$\eps |2\Gamma|^{-1/2}\sigma(r,s)$}
{$\alpha q(s)$}
Here, $\bar{r}=\bar{y}+\bar{z}+\bar{\delta}$ and $\bar{t}=\bar{x}+\bar{y}+\bar{\delta}$ and $\bar{s}=\bar{z}+\bar{x}+\bar{\delta}$, and we note that the inverse of the matrix $(|2\Gamma|^{-1/2}\sigma(t,r)^{-1})_{t,r}$ is $(|2\Gamma|^{-1/2}\sigma(r,t))_{t,r}$, where $r,t$ run over arbitrary fixed $2\Gamma$-cosets, by the character sum in \autoref{lm_characterSum}. We spell out the identity in question, using as summation condition $\bar{x}+\bar{y}+\bar{\delta}=\bar{r}+\bar{s}+\bar{\delta}$:
\begin{align*}
\alpha q({r})\sigma({r},{s})\alpha q({s})
&=\eps|2\Gamma|^{-1/2} \sum_{t\in \bar{r}+\bar{s}+\bar{\delta}} \sigma(r,t)q(t)^{-1}\sigma(t,s).
\end{align*}
This is a somewhat modified Gauss sum. We reduce it to a formula that only depends on $r+s$: we use $q(r)\sigma(r,s)q(s)=q(r+s)\omega(r+s,r,s)\baromega(\bar{\delta},r,s)$ on the left-hand side, and on the right-hand side $\sigma(t,r)\sigma(t,s)= \sigma(t,r+s) \omega(t,r,s)$. Then, by the summation condition on $t$, the terms involving $\omega$ on both sides cancel and the identity in question becomes
\begin{align*}
\alpha^2 q(r+s)
&=\eps|2\Gamma|^{-1/2} \sum_{t\in \bar{r}+\bar{s}+\bar{\delta}} \sigma(t,r+s)q(t)^{-1}.
\end{align*}
We may reduce to the case $r+s=0$. Indeed, let $b=r+s$ and substitute $t=a+r+s$ with $a\in \bar{\delta}$. Then, using again the additivity property of $q$ gives
\begin{align*}
\alpha^2 q(b)
&=\eps|2\Gamma|^{-1/2}\sum_{a\in \bar{\delta}}\sigma(a+b,b) \sigma(a,b)^{-1}q(a)^{-1}q(b)^{-1}\baromega(\bar{a}+\bar{b}+\bar{\delta},a,b).
\end{align*}
Furthermore, using the coboundary property for $\sigma(\cdot,b)$, the property $\sigma(b,b)=q(b)^2$ and $\baromega(\bar{a}+\bar{\delta},\cdot,\cdot)=\baromega(0,\cdot,\cdot)=1$ by the summation condition reduces the identity in question to the following identity independent of $r$ and $s$:
\begin{align*}
\alpha^2
&=\eps|2\Gamma|^{-1/2}\sum_{a\in \bar{\delta}} q(a)^{-1}.
\end{align*}
We defined $\alpha$ just so that this holds, and hence the above inverse hexagon identity is satisfied. Next, we consider the hexagon identity
\hexagon{$\X^{\bar x}$}{$\X^{\bar y}$}{$\X^{\bar y}$}{$\X^{\bar z}$}{$\X^{\bar z}$}
{$\eps |2\Gamma|^{-1/2}\sigma(t,r)^{-1}$}
{$\substack{q(r)^{-1}\baromega(\bar{x}+\bar{r},r,-r)\\ \baromega(\bar{y},r,-r)}$}
{$\eps |2\Gamma|^{-1/2}\sigma(-r,s)^{-1}$}
{$\alpha q(t)$}
{$\eps |2\Gamma|^{-1/2}\sigma(t,s)^{-1}$}
{$\alpha q(s)$}
where $\bar{t}=\bar{x}+\bar{y}+\bar{\delta}$ and $\bar{r}=\bar{y}+\bar{z}+\bar{\delta}$ and $\bar{s}=\bar{z}+\bar{x}+\bar{\delta}$. We spell out the identity in question and simplify the right-hand side by using formula \eqref{eq_sigmanegativeL} for $\sigma(-r,s)$ and cancelling $\baromega(\bar{x}+\bar{y}+\bar{r}+\bar{s},r,-r)$ since $\bar{x}+\bar{y}+\bar{r} +\bar{s}=0$ in $\Gamma/2\Gamma$:
\begin{align*}
&\alpha q({t})\sigma({t},{s})^{-1}\alpha q({s})\\
&=\eps |2\Gamma|^{-1/2}\sum_{r\in \bar{y}+\bar{z}+\bar{\delta}} \sigma(t,r)^{-1}
\cdot \sigma(-r,s) \cdot q(r)^{-1}\baromega(\bar{x}+\bar{y}+\bar{r},r,-r) \\
&=\eps |2\Gamma|^{-1/2}\sum_{r\in \bar{s}+\bar{t}+\bar{\delta}} \sigma({t},{r})^{-1}
\sigma({r},{s}) q({r})^{-1}.
\end{align*}
If we substitute $t=-l$ and use formula \eqref{eq_qminus} for $q(-l)$ and formula \eqref{eq_sigmanegativeL} for $\sigma(-l,s)$ and $\sigma(-l,r)$, we collect an additional factor $\baromega(\bar{l}+\bar{\delta},l,-l)\omega(r,l,-l)\omega(s,l,-l)$, which again cancels due to the condition $\bar{t}+\bar{r}+\bar{s}+\bar{\delta}=0$ in $\Gamma/2\Gamma$. Then the hexagon identity reduces again to the modified Gauss sum that we have proved for the inverse hexagon identity above:
\begin{align*}
\alpha q(l)\sigma(l,{s})\alpha q({s})
&=\eps |2\Gamma|^{-1/2}\sum_{r\in \bar{s}+\bar{t}+\bar{\delta}} \sigma(l,{r})
\sigma({r},{s}) q({r})^{-1}.
\end{align*}
This concludes the proof of the hexagon identities.

\smallskip

\noindent
\emph{$G$-Ribbon Structure:} We shall verify that the following defines a ribbon twist on $\LM(\Gamma,\sigma,\omega,\bar{\delta},\eps\,|\,q,\alpha,\beta)$:
\begin{equation*}
\theta_{\C_a}=\sigma(a,a),\quad\theta_{\X^{\bar x}}=\beta
\end{equation*}
for $a\in\Gamma$ and $\bar{x}\in\Gamma/2\Gamma$, with a choice of $\beta=\pm\alpha^{-1}$ and $\alpha^2=\eps |\Gamma|^{-1/2}\sum_{a\in \bar{\delta}} q(a)^{-1}$ as above. More precisely, we check the defining condition in \autoref{def_Gtwist}:
\begin{equation*}
\theta_{\C_{a+b}}\theta_{\C_a}^{-1}\theta_{\C_b}^{-1}
=\sigma(b,a)\sigma(a,b)
\end{equation*}
for $a,b\in\Gamma$. This is a standard calculation for abelian $3$-cocycles. In our case, the coboundary formula in \autoref{ass_strongdisc} applied to both sides of $\sigma(a+b,a+b)$ gives $\sigma(a,a)\sigma(a,b)\sigma(b,a)\sigma(b,b)$, and the further terms $\omega(a+b,a,b)\omega(a,a,b)\omega(b,a,b)=1$. Then, consider, for $a\in\Gamma$, $\bar{x}$ in $\Gamma/2\Gamma$,
\begin{equation*}
\theta_{\X^{\bar a+\bar x}}\theta_{\C_a}^{-1}\theta_{\X^{\bar x}}^{-1}
=\baromega(\bar{a}+\bar{x},a,-a)
\cdot \baromega(\bar{x}+\bar{a},a,-a)q(a)^{-1}
\cdot q(a)^{-1}.
\end{equation*}
This holds because on the left-hand side the twists on $\X^{\bar x}$ and $\X^{\bar a+\bar x}$ are equal and cancel and because on the right-hand side the tensor structure and the additional~$\omega$ in the braiding cancel, leaving $q(a)^{-2}=\sigma(a,a)^{-1}$. Then, we look at
\begin{equation*}
\theta_{\X^{\bar x+\bar a}}\theta_{\X^{\bar x}}^{-1}\theta_{\C_a}^{-1}
=q(a)^{-1} \baromega(\bar{a}+\bar{x},a,-a)
\cdot \baromega(\bar{x}+\bar{a},a,-a)q(a)^{-1}.
\end{equation*}
This holds analogously to the previous case. Finally, we verify
\begin{equation*}
\theta_{\C_t}\theta_{\X^{\bar x}}^{-1}\theta_{\X^{\bar y}}^{-1}
=\alpha q(t)\cdot \alpha q(t)
\end{equation*}
for $\bar{x},\bar{y}\in\Gamma/2\Gamma$ and for any $t\in\bar x+\bar y+\bar{\delta}$. Indeed, this holds because the $\alpha^2$ on both sides cancel and $q(t)^2=\sigma(t,t)$.

This concludes the proof of the $\Z_2$-ribbon structure on $\LM(\Gamma,\sigma,\omega,\bar{\delta},\eps\,|\,q,\alpha,\beta)$.

\smallskip

\noindent
\emph{Pseudo-Unitarity:} We finally verify the assertion that these choices of ribbon structure together with the rigid structure have positive quantum dimensions if and only if $\alpha\beta=\eps$, which is completely analogous to \autoref{cor_TYpseudounitary}. For the objects $\C_a$, $a\in\Gamma$, the quantum dimension is $1$. For the objects $X^{\bar{x}}$, $\bar{x}\in\Gamma/2\Gamma$, with the rigid structure and ribbon structure given above, the quantum dimension of $\X^{\bar{x}}$ is
\begin{align*}
\dim(\X^{\bar{x}})
&=\eval_{\X^{\bar{x}}}\circ c_{\X^{\bar{x}},\X^{-\bar{x}-\bar{\delta}}} \circ (\theta_{\X^{\bar{x}}}\otimes \id_{\X^{-\bar{x}-\bar{\delta}}})\circ \coeval_{\X^{\bar{x}}} \\
&=\eps |\Gamma|^{1/2}\cdot (\alpha\,q(0))\beta \id_{\X^{\bar{x}}}\\
&=\eps (\alpha\beta) |\Gamma|^{1/2}\id_{\X^{\bar{x}}},
\end{align*}
independently of $\bar{x}$. This concludes the proof of \autoref{thm_evenTY}.
\end{proof}


\bibliographystyle{alpha_noseriescomma}
\bibliography{references}{}

\newcommand{\SortNoop}[1]{}
\begin{thebibliography}{DGNO10}

\bibitem[ADL05]{ADL05}
Toshiyuki Abe, Chongying Dong and Haisheng Li.
\newblock Fusion rules for the vertex operator algebra {$M(1)$} and {$V^+_L$}.
\newblock {\em Comm. Math. Phys.}, 253(1):171--219, 2005.
\newblock (\href{http://arxiv.org/abs/math/0310425v1}{arXiv:math/0310425v1
  [math.QA]}).

\bibitem[BN13]{BN13}
Sebastian Burciu and Sonia Natale.
\newblock Fusion rules of equivariantizations of fusion categories.
\newblock {\em J. Math. Phys.}, 54(1):013511, 2013.
\newblock (\href{http://arxiv.org/abs/1206.6625v2}{arXiv:1206.6625v2
  [math.QA]}).

\bibitem[CS99]{CS99}
John~H. Conway and Neil~J.~A. Sloane.
\newblock {\em Sphere packings, lattices and groups}, volume 290 of {\em
  Grundlehren Math. Wiss.}
\newblock Springer, 3rd edition, 1999.

\bibitem[DGNO10]{DGNO10}
Vladimir~G. Drinfeld, Shlomo Gelaki, Dmitri~A. Nikshych and Victor Ostrik.
\newblock On braided fusion categories. {I}.
\newblock {\em Selecta Math. (N.S.)}, 16(1):1--119, 2010.
\newblock (\href{http://arxiv.org/abs/0906.0620v3}{arXiv:0906.0620v3
  [math.QA]}).

\bibitem[DN21]{DN21}
Alexei~A. Davydov and Dmitri~A. Nikshych.
\newblock Braided {Picard} groups and graded extensions of braided tensor
  categories.
\newblock {\em Selecta Math. (N.S.)}, 27(4):65, 2021.
\newblock (\href{http://arxiv.org/abs/2006.08022v2}{arXiv:2006.08022v2
  [math.QA]}).

\bibitem[DNR25]{DNR25}
Chongying Dong, Siu-Hung Ng and Li~Ren.
\newblock Orbifolds and minimal modular extensions.
\newblock {\em Adv. Math.}, 462:110103, 2025.
\newblock (\href{http://arxiv.org/abs/2108.05225v2}{arXiv:2108.05225v2
  [math.QA]}).

\bibitem[EG18]{EG18b}
Pavel~I. Etingof and C\'{e}sar Galindo.
\newblock Reflection fusion categories.
\newblock {\em J. Algebra}, 516:172--196, 2018.
\newblock (\href{http://arxiv.org/abs/1803.05568v2}{arXiv:1803.05568v2
  [math.QA]}).

\bibitem[EGNO15]{EGNO15}
Pavel~I. Etingof, Shlomo Gelaki, Dmitri~A. Nikshych and Victor Ostrik.
\newblock {\em Tensor Categories}, volume 205 of {\em Math. Surveys Monogr.}
\newblock Amer. Math. Soc., 2015.

\bibitem[EM50]{EM50}
Samuel Eilenberg and Saunders {Mac Lane}.
\newblock Cohomology theory of {A}belian groups and homotopy theory. {I}.
\newblock {\em Proc. Nat. Acad. Sci. U.S.A.}, 36:443--447, 1950.

\bibitem[ENO05]{ENO05}
Pavel~I. Etingof, Dmitri~A. Nikshych and Victor Ostrik.
\newblock On fusion categories.
\newblock {\em Ann. of Math.}, 162(2):581--642, 2005.
\newblock (\href{http://arxiv.org/abs/math/0203060v11}{arXiv:math/0203060v11
  [math.QA]}).

\bibitem[ENO10]{ENO10}
Pavel~I. Etingof, Dmitri~A. Nikshych and Victor Ostrik.
\newblock Fusion categories and homotopy theory, with an appendix by {E}hud
  {M}eir.
\newblock {\em Quantum Topol.}, 1(3):209--273, 2010.
\newblock (\href{http://arxiv.org/abs/0909.3140v2}{arXiv:0909.3140v2
  [math.QA]}).

\bibitem[Gal17]{G17}
C\'{e}sar Galindo.
\newblock Coherence for monoidal {$G$}-categories and braided {$G$}-crossed
  categories.
\newblock {\em J. Algebra}, 487:118--137, 2017.
\newblock (\href{http://arxiv.org/abs/1604.01679v2}{arXiv:1604.01679v2
  [math.QA]}).

\bibitem[Gal22]{Gal22}
C\'{e}sar Galindo.
\newblock Trivializing group actions on braided crossed tensor categories and
  graded braided tensor categories.
\newblock {\em J. Math. Soc. Japan}, 74(3):735--752, 2022.
\newblock (\href{http://arxiv.org/abs/2010.00847v1}{arXiv:2010.00847v1
  [math.QA]}).

\bibitem[GJ16]{GJ16}
C\'{e}sar Galindo and Nicol\'{a}s Jaramillo.
\newblock Solutions of the hexagon equation for abelian anyons.
\newblock {\em Rev. Colombiana Mat.}, 50(2):273--294, 2016.
\newblock (\href{http://arxiv.org/abs/1606.01414v4}{arXiv:1606.01414v4
  [math.QA]}).

\bibitem[GJ19]{GJ19}
Terry Gannon and Corey Jones.
\newblock Vanishing of categorical obstructions for permutation orbifolds.
\newblock {\em Comm. Math. Phys.}, 369(1):245--259, 2019.
\newblock (\href{http://arxiv.org/abs/1804.08343v2}{arXiv:1804.08343v2
  [math.QA]}).

\bibitem[GLM24]{GLM24a}
César Galindo, Simon Lentner and Sven Möller.
\newblock Computing {$G$}-crossed extensions and orbifolds of vertex operator
  algebras.
\newblock (\href{http://arxiv.org/abs/2409.16357v2}{arXiv:2409.16357v2
  [math.QA]}), 2024.

\bibitem[GN08]{GN08}
Shlomo Gelaki and Dmitri~A. Nikshych.
\newblock Nilpotent fusion categories.
\newblock {\em Adv. Math.}, 217(3):1053--1071, 2008.
\newblock (\href{http://arxiv.org/abs/math/0610726v2}{arXiv:math/0610726v2
  [math.QA]}).

\bibitem[GNN09]{GNN09}
Shlomo Gelaki, Deepak Naidu and Dmitri~A. Nikshych.
\newblock Centers of graded fusion categories.
\newblock {\em Algebra Number Theory}, 3(8):959--990, 2009.
\newblock (\href{http://arxiv.org/abs/0905.3117v1}{arXiv:0905.3117v1
  [math.QA]}).

\bibitem[GR25]{GR25}
Terry Gannon and Andrew Riesen.
\newblock Orbifolds of pointed vertex operator algebras {I}.
\newblock {\em Adv. Math.}, 482:110546, 2025.
\newblock (\href{http://arxiv.org/abs/2410.00809v1}{arXiv:2410.00809v1
  [math.QA]}).

\bibitem[JS93]{JS93}
Andr{\'e} Joyal and Ross Street.
\newblock Braided tensor categories.
\newblock {\em Adv. Math.}, 102(1):20--78, 1993.

\bibitem[Kir04]{Kir04}
Alexander Kirillov, Jr.
\newblock On {$G$}-equivariant modular categories.
\newblock (\href{http://arxiv.org/abs/math/0401119v1}{arXiv:math/0401119v1
  [math.QA]}), 2004.

\bibitem[KO02]{KO02}
Alexander Kirillov, Jr. and Viktor Ostrik.
\newblock On a {$q$}-analogue of the {M}c{K}ay correspondence and the {ADE}
  classification of {$\mathfrak{sl}_2$} conformal field theories.
\newblock {\em Adv. Math.}, 171(2):183--227, 2002.
\newblock (\href{http://arxiv.org/abs/math/0101219v3}{arXiv:math/0101219v3
  [math.QA]}).

\bibitem[McR21]{McR21b}
Robert McRae.
\newblock Twisted modules and {$G$}-equivariantization in logarithmic conformal
  field theory.
\newblock {\em Commun. Math. Phys.}, 383(3):1939--2019, 2021.
\newblock (\href{http://arxiv.org/abs/1910.13226v2}{arXiv:1910.13226v2
  [math.QA]}).

\bibitem[MS25]{Mar25}
Adrià Marín~Salvador.
\newblock Continuous {T}ambara-{Y}amagami tensor categories.
\newblock (\href{https://arxiv.org/abs/2503.14596v1}{arXiv:2503.14596v1
  [math.QA]}), 2025.

\bibitem[Nik80]{Nik80}
Viacheslav~V. Nikulin.
\newblock Integral symmetric bilinear forms and some of their applications.
\newblock {\em Math. USSR Izv.}, 14(1):103--167, 1980.

\bibitem[Nik08]{Nik08}
Dmitri~A. Nikshych.
\newblock Non-group-theoretical semisimple {H}opf algebras from group actions
  on fusion categories.
\newblock {\em Selecta Math. (N.S.)}, 14(1):145--161, 2008.
\newblock (\href{http://arxiv.org/abs/0712.0585v1}{arXiv:0712.0585v1
  [math.QA]}).

\bibitem[NNW09]{NNW09}
Deepak Naidu, Dmitri~A. Nikshych and Sarah Witherspoon.
\newblock Fusion subcategories of representation categories of twisted quantum
  doubles of finite groups.
\newblock {\em Int. Math. Res. Not.}, (22):4183--4219, 2009.
\newblock (\href{http://arxiv.org/abs/0810.0032v2}{arXiv:0810.0032v2
  [math.QA]}).

\bibitem[Par95]{Par95}
Bodo Pareigis.
\newblock On braiding and dyslexia.
\newblock {\em J. Algebra}, 171(2):413--425, 1995.

\bibitem[Sch09]{Sch09}
Nils~R. Scheithauer.
\newblock The {W}eil representation of {$\mathrm{SL}_2(\mathbb{Z})$} and some
  applications.
\newblock {\em Int. Math. Res. Not.}, 2009(8):1488--1545, 2009.

\bibitem[Tur00]{Tur00}
Vladimir~G. Turaev.
\newblock Homotopy field theory in dimension $3$ and crossed group-categories.
\newblock (\href{http://arxiv.org/abs/math/0005291v1}{arXiv:math/0005291v1
  [math.GT]}), 2000.

\bibitem[Tur10]{Tur10b}
Vladimir Turaev.
\newblock {\em Homotopy quantum field theory}, volume~10 of {\em EMS Tracts
  Math.}
\newblock Eur. Math. Soc., 2010.
\newblock With appendices by Michael M\"{u}ger and Alexis Virelizier.

\bibitem[TY98]{TY98}
Daisuke Tambara and Shigeru Yamagami.
\newblock Tensor categories with fusion rules of self-duality for finite
  abelian groups.
\newblock {\em J. Algebra}, 209(2):692--707, 1998.

\end{thebibliography}

\end{document}